\documentclass[10pt, a4paper]{amsart}
\usepackage{mypreamble}

\calclayout

\newcommand{\patch}{\operatorname{\mathbf{P}}}
\newcommand{\Spf}{\operatorname{Spf}}

\title{Rigidity of Automorphic Galois Representations over CM Fields}
\author{Lambert A'Campo}
\address{University of Oxford, Radcliffe Observatory, Andrew Wiles Building, Woodstock Rd, Oxford OX2 6GG}
\email{lambert.acampo@maths.ox.ac.uk}
\date{\today}

\usepackage[pagewise]{lineno}

\begin{document}
	
	\begin{abstract}
		We show the vanishing of adjoint Bloch--Kato Selmer groups of automorphic Galois representations over CM fields. This proves their rigidity in the sense that they have no deformations which are de Rham. 
		In order for this to make sense we also prove that automorphic Galois representations over CM fields are de Rham themselves. 
		Our methods draw heavily from the 10 author paper, where these Galois representations were studied extensively. Another crucial piece of inspiration comes from the work of P. Allen who used the smoothness of certain local deformation rings in characteristic $0$ to obtain rigidity in the polarized case.
	\end{abstract}
	
	\maketitle
	
	\tableofcontents
	
	\clearpage

	\section{Introduction} 
	
	Let $F$ be a CM number field with absolute Galois group $G_F$ and $\rho : G_F \to \GL_n(\overline{\rationals}_p)$ a continuous Galois representation which is unramified outside a finite set of places $S$ and de Rham at all places above $p$. We say that $\rho$ is rigid if the geometric Bloch--Kato Selmer group 
	\[
	H^{1}_{g}(G_{F, S}, \ad \rho) := \ker \left(H^1(G_{F,S}, \ad \rho) \to \prod_{v \mid p} H^1(F_v, B_{dR} \otimes_{\rationals_p} \ad \rho) \right)
	\]
	vanishes, where $G_{F,S}$ is the maximal quotient of $G_F$ which is unramified outside $S$ and $B_{dR}$ is Fontaine's de Rham period ring. The representation $\rho$ is rigid if and only if it has no non-trivial deformations which are de Rham and almost everywhere unramified. In this paper we prove, under certain assumptions, that the Galois representations attached to cohomological cuspidal automorphic representations of $\GL_n(\mathbb{A}_F)$ are rigid.
	
	The existence of such Galois representations was proven in \cite{harris_lan_taylor_thorne} and \cite{scholze_torsion}. For an isomorphism $\iota : \overline{\rationals}_p \to \complex$ and a cohomological cuspidal automorphic representation $\Pi$ of $\GL_n(\mathbb{A}_F)$ we let 
	\[
	r_\iota(\Pi) : G_F \to \GL_n(\overline{\rationals}_p)
	\]
	denote the Galois representation from the main theorem of \cite{harris_lan_taylor_thorne}.
	See \cite{borel_jacquet_automorphic} for the definition of an automorphic representation and cohomological means algebraic as defined in \cite[Definition 1.8]{clozel1990}. Concretely, $\Pi$ is cohomological if and only if its infinitesimal character coincides with that of an algebraic representation of $\GL_n(F)$. In this paper a number field $F$ is called CM if there exists $\sigma \in \Aut(F)$ such that for all field homomorphisms $\iota : F \to \complex$, we have $\iota \circ \sigma = c \circ \iota$, where $c \in \Aut(\complex)$ is the complex conjugation. Note that this includes totally real fields (when $\sigma = \id$) and one can show that subfields of CM fields are CM.
	Now we can state our main theorem.
	
	\begin{theorem}[= Theorem \ref{main_theorem}]
		Let $F \subset L$ be CM fields and let
		$\rho : G_F \to \GL_n(E)$ be a continuous representation,
		where $E \subset \overline{\rationals}_p$ is a finite extension of $\rationals_p$
		containing the images of all field homomorphisms $L \to \overline{\rationals}_p$. Moreover, let $\Pi$ be a cohomological  cuspidal
		automorphic representation of $\GL_n(\mathbb{A}_L)$ and let $S$ be a finite set of places of $L$, stable under complex conjugation, containing the archimedean ones and those above $p$ such that $\Pi_v$ is unramified for all $v \not \in S$. Let $S_p = S_{cris} \cup S_{ss}$ be a partition of the places of $L$ lying above $p$ which is stable under complex conjugation such that the following are satisfied.
		\begin{enumerate}[(a)]
			\item $p \nmid 2n$;
			\item There exists an isomorphism $\iota : \overline{\rationals}_p \to \complex$ such that $\rho \rvert_{G_L} \otimes_{E} \overline{\rationals}_p \cong r_\iota(\Pi)$;
			\item The residual representation $\overline{\rho} \rvert_{G_L}$ is absolutely irreducible and decomposed generic \cite[Definition 4.3.1]{10author}. Moreover, $\zeta_p \not \in L$ and $\overline{\rho} \rvert_{G_{L(\zeta_p)}}$ has enormous image \cite[Definition 6.2.28]{10author}, where $\zeta_p \in \overline{L}$ is a primitive $p$th root of unity;
			\item If $v \in S \setminus S_p$, then the Weil--Deligne representation $\operatorname{WD}(\rho \rvert_{G_{L_v}})$ is generic
			\cite[Definition 1.1.2]{allen_polarizable};
			\item If $v \in S_{cris}$, then $\Pi_v$ is unramified. If $v \in S_{ss}$, then
			$\rho \rvert_{G_{L_v}}$ is de Rham and for any finite extension $L_v'/L_v$, $\operatorname{WD}(\rho \rvert_{G_{L'_v}})$ is generic.
		\end{enumerate}
		Then the Bloch--Kato Selmer group $H^1_g(G_{F,S}, \ad \rho )$ vanishes, i.e. $\rho$ is rigid. Moreover, 
		$\rho \rvert_{G_{L_v}}$ is crystalline for all $v \in S_{cris}$.
	\end{theorem}
	
	\begin{remark}
		There are two main new ingredients that allow us to prove this theorem. The first one is a generalisation of the degree-shifting argument from \cite[section 4]{10author}, where we allow ramification at $p$ and semistable Galois representations. As a result we obtain the weak semistable local-global compatibility theorem \ref{weak_semistable_local_global_thm}. See the outline below and section 4 for more details. 
		
		The second ingredient is a Taylor--Wiles patching lemma that uses the smoothness of certain characteristic 0 deformation rings to eliminate nilpotent ideals in patched Hecke algebras. The idea of using the smoothness comes from \cite{allen_polarizable}. See the outline below for more context and lemma \ref{patching_lemma} for the precise statement.
		
		The local assumptions \emph{(d)} and \emph{(e)} are as general as I could manage with these methods but they will become unnecessary when more precise local-global compatibility theorems are proven.
		
		Using the techniques of Miagkov--Thorne \cite{miagkov_thorne} it should also be possible to replace ``enormous image" with ``adequate image" \cite[Definition 2.3]{small_image_thorne} in assumption \emph{(c)}.
	\end{remark}
	
	Together with the potential automorphy theorems of \cite{10author} the main theorem implies the more digestible
	
	\begin{corollary}[= Corollary \ref{elliptic_curve_corollary}]
		Let $A/F$ be an elliptic curve without complex multiplication over a CM field $F$ and $p \geq 7$ a prime such that $\zeta_p \not \in F$ and the image of $G_F$ in $\operatorname{Aut}(A[p])$ contains $\SL_2(\mathbb{F}_p)$.
		Let
		\[
		V_p A = \left( \lim_{\longleftarrow} A[p^n](\overline{F}) \right) \otimes_{\ints_p} \rationals_p
		\]
		be the $p$-adic $G_F$-representation attached to $A$. Then $V_p A$ is rigid, or equivalently any short exact sequence
		\[
		0 \to V_p A \to V \to V_p A \to 0,
		\]
		where $V$ is a de Rham $G_F$-representation, splits. Furthermore, if $F/\rationals$ is Galois and $n$ is a positive integer, such that $p > 2n + 3$, then the symmetric power $\operatorname{Sym}^n (V_p A)$ is rigid, too.
	\end{corollary}
	
	\begin{remark}
		Serre's open image theorem \cite[assertion (6)]{serre_proprietes_galoisiennes} shows that for a given elliptic curve without CM, all but finitely many primes satisfy the conditions of the corollary. To check them at a particular prime one can often use the LMFDB \cite{lmfdb}.
	\end{remark}
	
	More generally, the recent theorem \cite[Theorem 1.4]{qian_potential} gives many instances where the conditions of our main theorem are satisfied. In particular, it implies
	
	\begin{corollary} 
		Let $F$ be a CM field and let
		$\rho : G_F \to \GL_n(E)$ be a continuous representation,
		where $E \subset \overline{\rationals}_p$ is a finite extension of $\rationals_p$
		containing the images of all field homomorphisms $F \to \overline{\rationals}_p$. Let $S$ be a finite set of places of $F$ containing those where $\rho$ is ramified. Suppose the following are satisfied.
		\begin{enumerate}[(a)]
			\item $p \nmid 2n$;
			\item For all places $v \mid p$, $\rho \rvert_{G_{F_v}}$ is potentially semistable and ordinary with regular Hodge--Tate weights \cite[Definition 1.3]{qian_potential}. Moreover, for all finite extensions $F'_v/F_v$, $\operatorname{WD}(\rho \rvert_{G_{F'_v}})$ is generic;
			\item The residual representation $\overline{\rho}$ is absolutely irreducible and decomposed generic as defined in \cite[Definition 4.3.1]{10author}. Moreover, $\zeta_p \not \in F$ and $\overline{\rho} \rvert_{G_{F(\zeta_p)}}$ has enormous image as defined in \cite[Definition 6.2.28]{10author};
			\item There exists $\sigma \in G_F \setminus G_{F(\zeta_p)}$ such that $\overline{\rho}(\sigma)$ is a scalar;
			\item If $v \in S \setminus S_p$, then the Weil--Deligne representation $\operatorname{WD}(\rho \rvert_{G_{F_v}})$ is generic.
		\end{enumerate}
		Then $H^1_g(G_{F,S}, \ad \rho )= 0$.
	\end{corollary} 
	
	Along the way we also need to prove a weak form of semistable local-global compatibility and as a consequence we obtain 
	
	\begin{theorem}[= Theorem \ref{derham_theorem}]
Let $F$ be a CM field, $\iota : \complex \to \overline{\rationals_p}$ an isomorphism and $\Pi$ a cohomological cuspidal automorphic representation of $\GL_n(\mathbb{A}_F)$. If the reduction $\overline{r_\iota(\Pi)}$ is 
		absolutely irreducible and decomposed generic, then for every place $v \mid p$ of $F$ the representation $r_{\iota}(\Pi) \rvert_{G_{F_v}}$ is potentially semistable (hence also de Rham) with Hodge--Tate weights
		\[
			HT_{\tau} := \{ \lambda_{\tau, 1} + n - 1, \lambda_{\tau,2} + n -2,  \dots, \lambda_{\tau,n}\},
		\]
		where $\lambda \in (\ints^n)^{\Hom(F, \overline{\rationals}_p)}$ is the unique dominant weight for $\GL_n/F$ such that the infinitesimal character of $\Pi_\infty$ coincides with the infinitesimal character of $(V_\lambda \otimes_{\mathcal{O}, \iota} \complex)^\vee$ (see definition \ref{highest_weight_module_def}).
		If $\Pi_v$ and $\Pi_{v^c}$ are unramified, then $r_{\iota}(\Pi) \rvert_{G_{F_v}}$ is crystalline.
	\end{theorem}
	
	\begin{remark}
		Further local-global compatibility results in the crystalline case, which also treat torsion Galois representations, will appear in forthcoming work of Ana Caraiani and James Newton. 
	\end{remark}

	\subsection{Background}
	
	Let us provide some context for our results.
	We begin by explaining how the main theorem is predicted by a combination of the Fontaine--Mazur conjecture \cite[Conjecture 1]{fontaine_mazur93} and the Bloch--Kato conjecture
	\cite[\S 5]{bloch_kato_conjecture}, to which \cite{bellaiche_BK} is an introduction.

	Let $F$ be a CM field, $p$ a prime which splits completely\footnote{This condition is just to make the exposition a bit simpler. With more bookkeeping one can also treat general primes $p$ and Galois representations with coefficients in a finite extension of $\rationals_p$.} in $F$ and $\rho : G_F \to \GL_n(\rationals_p)$ an absolutely irreducible Galois representation which is unramified outside a finite set of places and potentially semistable at all places above $p$.
	If $\ad \rho$ is semisimple, then 
	the Fontaine--Mazur conjecture implies that $\ad \rho$ is isomorphic to a subquotient of a Tate twist of a $p$-adic \'etale cohomology group of a variety over $F$ and that its L-function $L(\ad \rho, s)$ has meromorphic continuation to the whole complex plane \cite[Conjecture 3b]{fontaine_mazur93}. Moreover, we assume that the subquotient corresponds to a ``motive" of weight zero over $F$ to which we can apply the Bloch--Kato conjecture. In particular, the conjectures stated in \cite[\S 4.2.2]{fontaine_perrin-riou94} together with \cite[Proposition 3.8 and Corollary 3.8.4]{bloch_kato_conjecture} imply that
	\[
	\dim_{\rationals_p} H^1_g(G_F, \ad \rho) - \dim_{\rationals_p} H^0(G_F, \ad \rho) = \operatorname{ord}_{s = 0} L(\ad \rho(1), s) = \operatorname{ord}_{s = 1} L(\ad \rho, s).
	\]
	Since $\rho$ is absolutely irreducible, Schur's lemma implies that $\dim_{\rationals_p} H^0(G_F, \ad \rho) = 1$. If $\rho \otimes_{\rationals_p} \overline{\rationals}_p = r_\iota(\Pi)$ for some cuspidal cohomological automorphic representation $\Pi$, then \cite[Lecture 9 \S3]{cogdell_notes} shows that
	\[
	\operatorname{ord}_{s = 1} L(\ad \rho, s) = \operatorname{ord}_{s = 1} L(\Pi \times \widetilde{\Pi}, s) = -1.
	\]
	Hence the Bloch--Kato Selmer group $H^1_g(G_F, \ad \rho)$ must vanish, as confirmed by our main theorem.
	
	When $\rho$ is conjugate self-dual, many cases of this rigidity have been proven in the last 20 years but even for these our theorem often says something new  because previous results only treat the subspace of $H^1_g(G_F, \ad \rho)$ corresponding to extensions which are conjugate self-dual themselves.
	
	Recent results in the conjugate self-dual case were proven in \cite{allen_polarizable}, with even more generalisations (but still under the conjugate self-dual assumption) in \cite{newton2020adjoint}. Besides confirming cases of the Bloch--Kato conjecture, these results have also led to exciting applications to symmetric power functoriality \cite{newton_thorne_symmetric1,newton_thorne_symmetric2, newton_thorne_symmetric3}.
	
	Without the conjugate self-dual assumption very little was known except for the paper \cite{calegari_geraghty_harris_blochkato} which shows how conjectures on local-global compatibility imply the expected rigidity and also manages to show a special case unconditionally. 
	In the present paper, we obtain rigidity in a large class of cases without the conjugate self-dual assumption and without having to prove the full conjectures on local-global compatibility. 
	
	We continue the tradition of using automorphy lifting theorems to prove rigidity and use the 10 author paper \cite{10author} which recently established many cases of automorphy lifting for Galois representations over CM fields which are not necessarily conjugate self-dual. However, due to some nilpotent ideals in Hecke algebras their results do not directly imply the result we seek. This is one issue we address in the present paper. Moreover, the techniques of \cite{10author} require that $\rho$ satisfies either Fontaine--Laffaille or ordinary conditions at $p$. We weaken this and require much milder conditions on $\rho \rvert_{G_{F_v}}$ for the places $v \mid p$.

	\subsection{Notation} 
	If $F$ is a number field, i.e. a finite extension of $\rationals$, and $v$ is a place of $F$ and $S$ is a finite set of places of $F$, then
	we denote by 
	\begin{figure}[ht]
		\centering
		\begin{tabular}{l | l}
			$F_v$ & the completion of $F$ at $v$ \\
			$\mathbb{A}_F$ & the ring of adeles of $F$ \\
			$\mathbb{A}_F^\infty$ & the ring of finite adeles of $F$ \\
			$\mathbb{A}_F^S$ & the ring of $S$-adeles of $F$ \\
			$F_S$ & the maximal extension of $F$, unramified outside a set of places $S$ \\
			$G_{F,S}$ & the Galois group $\Gal(F_S/F)$ \\
			$\mathcal{O}_F$ & the ring of integers of $F$  \\
			$\mathcal{O}_{F_v}$ & the ring of integers of $F_v$
		\end{tabular}
	\end{figure}
	
	Moreover, for a general field $F$ we let $\overline{F}$ be a separable closure and $G_F := \Gal(\overline{F}/F)$.
	For any place $v$ of a number field $F$, the choice of an embedding $\overline{F} \hookrightarrow \overline{F_v}$ gives rise to an embedding of absolute Galois groups
	$G_{F_v} \hookrightarrow G_F$. The conjugacy class of $G_{F_v} \hookrightarrow G_F$ is independent of choices, hence it makes sense to restrict the isomorphism class of a representation of $G_F$ to $G_{F_v}$.
	
	We let $B_{dR}$ denote Fontaine's de Rham period ring. See \cite{berger_introduction_padic} for an introduction.
	If $p$ is a prime, $E$ is a finite extension of $\rationals_p$, $F$ is a number field such that $|\Hom(F, E)| = |\Hom(F, \overline{\rationals}_p)|$ and $\rho : G_F \to \GL_n(E)$ is a continuous representation, then we say that $\rho$ is de Rham if for each place $v \mid p$, the restriction $\rho \rvert_{G_{F_v}}$ is de Rham. 
	Assuming that $\rho$ is de Rham, we denote by $\operatorname{WD}(\rho \rvert_{G_{F_v}})$ the Weil--Deligne representation attached to $\rho \rvert_{G_{F_v}}$ for any place $v$ of $F$. This construction is recalled in \cite[1.1.4 and 1.1.6]{allen_polarizable}. 

	If $F$ is any field in which the prime $p$ is invertible, then for a continuous representation $\rho : G_F \to \GL_n(E)$ and $n \in \ints$, we define the Tate twist
	$\rho(n) := \rho \otimes_{\rationals_p} \chi_{cyc}^{n}$, where $\chi_{cyc} : G_F \to \rationals_p^\times$ is the cyclotomic character.
	
	\subsection{Outline of the Proof}
	
	Let us sketch the proof of the main theorem. Certain things we say here will be slightly wrong in order to simplify the exposition but we hope that nothing essential will be misrepresented. 
	Fix a CM field $F$ and a cohomological cuspidal automorphic representation $\Pi$ of $\GL_n(\mathbb{A}_F)$. We wish to show that the associated Galois representation $\rho := r_\iota(\Pi)$ 
	is rigid. We will get to the conditions needed to impose on $F$ and $\rho$ as we move along.
	
	Firstly, since $\Pi$ is cohomological we can find the Hecke modules attached to $\Pi$ as direct summands of the cohomology of locally symmetric spaces. 
	These spaces are defined for each open compact subgroup $K < \GL_n(\mathbb{A}_F^\infty)$ as the double quotient topological space
	\[
	X^K_{\GL_n/F} := \GL_n(F) \backslash \GL_n(\mathbb{A}_F) / \reals^\times U_n(F \otimes \reals) \cdot K.
	\]
	When $F = \rationals$, then $X^K_{\GL_2/F}$ is closely related to modular curves.
	Here is what happens when $F = \rationals(\sqrt{-2})$ and 
	\[
	K(\mathfrak{n}) = \left\{ k \in \GL_2(\widehat{\ints[\sqrt{-2}]}) : k \equiv \id \pmod{\mathfrak{n}} \right\},
	\] 
where $\widehat{\ints[\sqrt{-2}]}$ is the profinite completion of $\ints[\sqrt{-2}]$ and $\mathfrak{n} < \ints[\sqrt{-2}]$ is an ideal such that $2 \not \in \mathfrak{n}$. We have $F \otimes \reals = \complex$ and
	\[
	\GL_2(\complex)/\reals^\times U_2(\complex) \cong \mathbb{H}_3 := \{ (x,y,z) \in \reals^3 : z > 0\}.
	\] 
	Thus, by strong approximation for $\SL_2$, the connected components of $X^K_{\GL_2/F}$ are indexed by the ray class group $F^\times \backslash (\mathbb{A}^\infty_F)^\times / \det (K) \cong \Cl^{\mathfrak{n}}(F)$. Each connected component is isomorphic to the Bianchi manifold $\Gamma \backslash \mathbb{H}_3$, where
	\[
	\Gamma = \GL_2(F) \cap K(\mathfrak{n}) = \left\{\gamma \in \SL_2(\ints[\sqrt{-2}]) : \gamma \equiv \id \pmod{\mathfrak{n}} \right\}.
	\]
	
	Returning to our Galois representation, let $E \subset \overline{\rationals}_p$ be a finite extension of $\rationals_p$ containing the images of all $p$-adic embeddings of $F$.  Assume that $\rho$ takes values in $\GL_n(\mathcal{O})$, where
	$\mathcal O$ is the ring of integers of $E$. The cohomological assumption on $\Pi$ translates into the existence of an open compact subgroup $K \subset \GL_n(\mathbb{A}_F^\infty)$ and a local system $\mathcal{V}$ of finite free $\mathcal{O}$-modules on $X^K_{\GL_n/F}$ such that the characteristic polynomials of $\rho(\Frob_v)$ for all but finitely many $v$ can be read off from the action of certain Hecke operators on sheaf cohomology $H^\bullet(X^K_{\GL_n/F}, \mathcal{V})$. Namely, there is a commutative ring $\mathbb{T} \subset \End(H^\bullet(X^K_{\GL_n/F}, \mathcal{V}))$, a homomorphism
	$\lambda_\rho : \mathbb{T} \to \mathcal{O}$ and a sequence of
	 polynomials $P_v \in \mathbb{T}[Y]$ such that 
	$\det(Y \id - \rho(\Frob_v)) = \lambda_\rho(P_v)(Y)$ for all but finitely many $v$.
	The ring $\mathbb{T}$ and the polynomials $P_v$ can all be described explicitly but there is no need for that at this time. We just need to know that the automorphy of $\rho$ means that its Frobenius characteristic polynomials
	come from $\mathbb{T}$ acting on the cohomology of 
	some $X^K_{\GL_n/F}$.
	
	Conversely, the existence results from \cite{harris_lan_taylor_thorne} and \cite{scholze_torsion} show that for every homomorphism $\lambda : \mathbb{T} \to \overline{\rationals}_p$, there exists a semisimple Galois representation
	$\rho : G_F \to \GL_n(\overline{\rationals}_p)$ such that
	$\det(Y \id - \rho(\Frob_v)) = \lambda(P_v)(Y)$. In fact there is a stronger theorem proven in  \cite{scholze_torsion} which 
	states that there is a Galois representation
	$\rho : G_F \to \GL_n(\mathbb{T}^{red})$ such that
	$\det(Y \id - \rho(\Frob_v)) = P_v(Y)$, where $\mathbb{T}^{red}$ is the quotient of $\mathbb{T}$ by its nilradical. Moreover, this $\rho$ satisfies a number of ramification conditions $(C)$ depending on $K$ and $\mathcal{V}$ and the theory of Galois deformations tells us that there exists a ring $R$ and universal 
	Galois representation $\rho_{univ} : G_F \to \GL_n(R)$ satisfying $(C)$. Universality means that for any $\rho : G_F \to \GL_n(S)$ satisfying $(C)$, there exists a unique 
	ring homomorphism $\phi : R \to S$ such that $\rho = \phi \circ \rho_{univ}$.
	Now Scholze's theorem may be rephrased as saying that there exists a surjective homomorphism
	$R \to \mathbb{T}^{red}$.
	
	Finally, to prove that $\rho$ is rigid, 
	we need to show that the corresponding deformation space $H^1_g(G_{F,S}, \ad \rho)$ vanishes. By universality of $R$, there exists a unique homomorphism $R \to \mathcal{O}$ inducing $\rho$. Let $x$ be the kernel of this homomorphism. General Galois deformation theory shows that the tangent space of $\Spec R$ at $x$ is isomorphic to $H^1_g(G_{F,S}, \ad \rho)$.
	Thus, it suffices to show that the localisation $R_x$ is a field. As it happens, a generalisation of Matsushima's formula implies that $\mathbb{T}[1/p]$ is a product of fields and $R_x$ maps onto one of them by Scholze's theorem. Hence it would suffice to show that the map
	$R_x \to \mathbb{T}_{x}$ is injective. In fact, the Langlands philosophy predicts an isomorphism 
	$R \to \mathbb{T}$ and \cite{10author} already establishes that $R^{red}_x \to \mathbb{T}^{red}_x$ is an isomorphism in many cases. However, this only implies that $R_x$ is Artinian and not necessarily a field. This is why we have to do something new. But before we get there we have to set up the Taylor--Wiles patching method.
	
	Rather than proving something about the homomorphism $R \to \mathbb{T}$ directly, one puts it in a countable family of morphisms $R_n \to \mathbb{T}_{n}$. Then one can analyse a type of projective limit of these morphisms and finds that it is a lot more well-behaved. 
	We can be a bit more precise. Define $S_n := \mathcal{O}[x_1, \dots, x_q]/((x_i+1)^{p^n}-1)$ for $n \geq 0$ and $S_\infty = \lim_{\leftarrow} S_n \cong \mathcal{O}[[x_1, \dots, x_q]]$. 
	One can find ramification conditions $(C_n)$ and open compact subgroups $K_n$ such that the corresponding rings, $R_n$ and $\mathbb{T}_{n}$ are $S_n$-algebras
	and $R_n/\mathfrak{a} \cong R$ and $\mathbb{T}_{n}/\mathfrak{a} \cong \mathbb{T}$,
	where $\mathfrak{a} = (x_1, \dots, x_q)$ is the augmentation ideal of $S_n$.
	Then we obtain surjections of $S_n$-algebras
	$R_n \to \mathbb{T}_{n}/I_n$, where $I_n < \mathbb{T}_{n}$ is a nilpotent ideal. In fact, $I_n^\delta = 0$ for some $\delta$ which is independent of $n$.
	The patching construction yields $S_\infty$-algebras $R_\infty$ and $\mathbb{T}_\infty$ with a nilpotent ideal $I_\infty < \mathbb{T}_\infty$ such that $R_\infty/\mathfrak{a} = R$, 
	$\mathbb{T}_\infty/\mathfrak{a} = \mathbb{T}$ and a surjective $S_\infty$-algebra morphism
	$R_\infty \to \mathbb{T}_\infty/I_\infty$. 
	
	Now comes the main innovation of this paper. As observed in \cite{allen_polarizable}, we know that $(R_\infty)_x$ is a regular ring, hence we can find a lifting 
	\[
	\begin{tikzcd}
		& (\mathbb{T}_\infty)_x \arrow{d} \\
		(R_\infty)_x \arrow[dashed]{ru} \arrow{r} & (\mathbb{T}_\infty)_x/(I_\infty)_x
	\end{tikzcd}
	\]
	It will not necessarily be $S_\infty$-linear. 
	Nonetheless, by a depth and dimension argument explained in section 2.6 we find a finitely generated faithful $(\mathbb{T}_\infty)_x$-module (patched and localised cohomology) which is also faithful over $(R_\infty)_x$. Hence the lifting $(R_\infty)_x \to (\mathbb{T}_\infty)_x$ is injective. Using that $(R_\infty)_x \to (\mathbb{T}_\infty)_x/ (I_\infty)_x$ is already surjective by construction, one concludes that it must be an isomorphism.
	By Matsushima's formula we know that $\mathbb{T}_x$ is a field and 
    upon quotienting by
	$\mathfrak{a}$ we obtain an isomorphism $R_x \to \mathbb{T}_{x}$. Thus, $R_x$ is a field, too!
	
	The ramification conditions $(C_n)$ mentioned above are one of the most difficult pieces of the proof. To obtain them we generalise the Fontaine--Laffaille local-global compatibility from \cite{10author} to a weak form of semistable local-global compatibility in section 4. For this we adapt the degree shifting argument from \cite[Chapter 4]{10author} to the case where $p$ ramifies in $F$. The basic inductive strategy is the same but now there is a Leray--Serre spectral sequence which does not necessarily converge at the second page anymore. However, by restricting to suitable level subgroups we can arrange that its version with $\ints/p^m\ints$ coefficients converges at the second page. This allows us to extract the Hecke congruences we need for the local-global compatibility.

	\subsection{Acknowledgements} 
	First and foremost I would like to express my gratitude to my advisor James Newton without whom this paper would never have been completed. I thank him for suggesting this problem and guiding me throughout the whole project. He spent many hours discussing earlier drafts with me and explained a lot of things to me.
	I thank Ana Caraiani for generously sharing her ideas on how to generalise the degree-shifting argument. Moreover, I thank Richard Taylor for conversations about the Hodge--Tate weights of automorphic Galois representations.
	I am grateful to Toby Gee, Matteo Tamiozzo and the anonymous referees for their comments on previous versions of this text. 
	Finally, I thank Ehud Meir for answering a question of mine on MathOverflow and sharing his proof of lemma \ref{faithful_cancellation_lemma} there.
	
	\section{Abstract Patching}
	
	This section consists of pure commutative algebra and forms the foundation of the main argument. Its purpose is to prove lemma \ref{patching_lemma} which allows us to deduce the isomorphism $R_x \cong \mathbb{T}_x$ mentioned in the introduction. To do so we employ a variant of the Taylor--Wiles patching method \cite{taylor_wiles}. The main modifications that we need were proposed in \cite{calegari_geraghty_beyond} but we make a few further modifications in the main lemma below to deal with nilpotent ideals. Moreover, we use ultrafilters as a convenient organisational tool, as first introduced in this context in \cite{scholze_lubin_tate}. See also the exposition in \cite[section 4]{manning_multiplicity}.
	
	\subsection{Non-Canonical Limits} 
	It is common to define the $p$-adic integers as
	\[
	\ints_p = \lim_{\leftarrow} \ints/p^n \ints.
	\]
	Implicit in this definition is the system of reduction maps
	\[
	\ints/p\ints \leftarrow \ints/p^2 \ints \leftarrow \ints/p^3\ints \leftarrow \dots
	\]
	which is really what defines the projective limit. The goal of patching is to be able to define similar objects \emph{without} reference to such a system of morphisms. To the disciple of category theory this is of course blasphemy. However, sometimes in nature we encounter sequences of objects $(M_n)$ without any maps between them but we still could use a limit object $M_\infty$. The construction of such a limit is exactly the purpose of patching. 
	
	Let us sketch the idea. First, suppose that each $M_n$ is a $\ints/p^n\ints$-module such that $M_n$ is generated by at most $C$ elements for some constant $C$. We cannot write $M_\infty = \lim_{\leftarrow} M_n $ since we have no maps between the $M_n$. However, for each $e \geq 1$ there are only finitely many isomorphism classes of $\ints/p^e \ints$-modules which are generated by at most $C$ elements. Hence the sequence $(M_n/ p^e)_{n \geq 1}$ contains at least one isomorphism class infinitely often.
	
	In particular, choosing $e = 1$, there is a sequence $(n^{1}_k)_{k \geq 1}$ and a $\ints/p\ints$-module $M^1$ such that $M^1 \cong M_{n^1_k}/p$ for all $k \geq 1$.
	Next, there is a subsequence $(n^2_k)_{k \geq 1}$ of $(n^1_k)_{k \geq 1}$ and a $\ints/p^2 \ints$-module
	$M^2$ such that $M^2 \cong M_{n^2_k}/p^2$ for all $k$. Moreover, we have an isomorphism
	$M^2/p \cong M_{n_1^2}/p \cong M^1$. Now we can continue inductively and find a subsequence $(n^{j + 1}_k)_{k \geq 1}$ of $(n^j_k)_{k \geq 1}$ such that there exists a module $M^{j + 1}$ with  $M^{j + 1} \cong M_{n^{j + 1}_k}/p^{j + 1}$ for all $k \geq 1$ and an isomorphism $M^{j + 1}/p^j \cong M^{j}$. Let $f^{j + 1} : M^{j + 1} \to M^j$ be defined as the composition $M^{j + 1} \to M^{j + 1}/p^j \cong M^j$ and let
	\[
	M_\infty = \lim \left(M^1 \xleftarrow{f^2} M^2 \xleftarrow{f^3} M^3 \leftarrow \dots \right).
	\]

	\begin{figure}[ht]
		\centering
		\[
		\begin{tikzcd}
			M^1 \arrow[dashed, leftrightarrow]{r} & M_1/p  & M_2/p & M_3/p & M_4/p & \dots \\
			M^2 \arrow{u}{f^2} \arrow[dashed, leftrightarrow]{r} & M_1/p^2   \arrow{u} \arrow[dashed,leftrightarrow]{rr} & &  M_3/p^2   \arrow{u} & M_4/p^2 \arrow{u} &    \dots \\
			M^3 \arrow{u}{f^3} \arrow[dashed, leftrightarrow]{rrr} & & &  M_3/p^3 \arrow{u} & &   \dots \\
			\vdots \arrow{u}{f^4} & & & \vdots \arrow{u} & &
		\end{tikzcd}
		\]
		\caption{Construction of $M_\infty$}
		\label{patching_construction_figure}
	\end{figure}
	
	This construction is illustrated in figure \ref{patching_construction_figure} for some example of a sequence of modules $(M_n)_{n \geq 1}$. The $j$th row of the figure displays the sequence of modules $(M_{n^j_k}/p^j)_{k \geq 1}$, which are all isomorphic to $M^j$. The dashed arrows are isomorphisms chosen in a way such that the diagram commutes and
	in our example we have
	 \begin{linenomath*} \begin{align*}
		(n_k^1)_{k \geq 1} &= (1,2,3,4, \dots ) \\
		(n_k^2)_{k \geq 1} &= (1,3,4, \dots ) \\
		(n_k^3)_{k \geq 1} &= (1,3, \dots ) \\
	\end{align*} \end{linenomath*} 
	The module $M_\infty$ constructed in this way of course depends on many choices of subsequences. But that is not a problem in the applications as long as we make consistent choices. Ultraproducts turn out to be a very convenient tool to achieve this consistency. 
	
	\begin{definition}
		Let $X$ be a non-empty set and $\mathcal{P}(X)$ its power set. A non-empty subset $\mathcal{F} \subset \mathcal{P}(X)$	is called an ultrafilter on $X$ if the following are satisfied.
		\begin{itemize}
			\item If $A, B \in \mathcal{F}$, then $A \cap B \in \mathcal{F}$.
			\item If $A \in \mathcal{F}$ and $B \supset A$, then $B \in \mathcal{F}$.
			\item If $A \subset X$, then exactly one of $A$ and $X \setminus A$ belongs to $\mathcal{F}$.
		\end{itemize}
	\end{definition}

	For example, given an element $x \in X$, we have the ultrafilter 
	\[ 
	\mathcal{F}_x := \{A \in \mathcal{P}(X) : x \in A\} .
	\]
	An ultrafilter $\mathcal{F}$ on $X$ is called \emph{non-principal} if there is no $x \in X$ such that $\mathcal{F} = \mathcal{F}_x$. Equivalently, an ultrafilter $\mathcal{F}$ is non-principal if it contains no finite sets.
	It can be shown with Zorn's lemma that non-principal ultrafilters exist whenever $X$ is infinite.
	Given an ultrafilter $\mathcal{F}$ on $\{1,2,3, \dots \}$ and sets $M_n$, we define the ultraproduct
	\[
	\left( \prod_{n \geq 1} M_n \right)_{\mathcal{F}} := \left( \prod_{n \geq 1} M_n \right)_{/\sim},
	\] 
	where $(x_n) \sim (y_n)$ if $\{n : x_n = y_n\} \in \mathcal{F}$. If $\mathcal{F} = \mathcal{F}_{n_0}$ is principal, then the projection to $M_{n_0}$ induces a bijection $\left( \prod_{n \geq 1} M_n \right)_{\mathcal{F}} \cong M_{n_0}$. Hence ultraproducts are most interesting when $\mathcal{F}$ is non-principal.
	
	In this case we still have the key property that for any finite partition
	\[
	X = \bigcup_{i} X_i,
	\]
	exactly one of the (necessarily infinite) $X_i$  belongs to $\mathcal{F}$. This can be used to show that if all the $M_n$ have cardinality bounded by a constant, then the cardinality of $\left( \prod_{n \geq 1} M_n \right)_{\mathcal{F}}$ is bounded by the same constant and is in bijection with $M_n$ for infinitely many $n$ (but not induced by a projection!).
	Moreover, for maps $f_n : M_n \to N_n$ there is a canonical induced map $\left( \prod_{n \geq 1} M_n \right)_{\mathcal{F}} \to \left( \prod_{n \geq 1} N_n \right)_{\mathcal{F}}$. Hence
	it makes sense to form the limit 
	\[
	\lim_{\substack{\longleftarrow \\ e \geq 1}} \left( \prod_{n \geq 1} M_n/p^e \right)_{\mathcal{F}}
	\]
	for any sequence of $\ints_p$-modules $M_n$. This behaves very similarly to the construction of $M_\infty$ via subsequences above but only relies on one choice: the non-principal ultrafilter $\mathcal{F}$. Below we will define the patching functor as precisely this limit.
	In the applications we make one choice of such a filter in the beginning and never change it. This hides the non-canonical nature of the patching method quite well and saves the paper from many indices.
	
	\subsection{Ultraproducts of Finite Sets}
	
	Let $X$ be an infinite set and $\mathcal{F}$ a non-principal ultrafilter on it.
	Let $\mathbf{Set}_M$ be the category of sets of cardinality bounded by $M \in \nats$.
	Then we have a functor
	\[
	\operatorname{ulim}_{\mathcal{F}} : \prod_{\alpha \in X} \mathbf{Set}_M \to \mathbf{Set}_M : (S_\alpha) \mapsto \left( \prod_{\alpha \in X} S_\alpha \right)_{\mathcal{F}}
	\]
	To see that this is well-defined, note that there are only finitely many isomorphism classes of objects in $\mathbf{Set}_M$, one for each non-negative integer $\leq M$. This yields the partition
	\[
	X = \bigcup_{i = 0}^M X_i,
	\]  
	where $X_i = \{\alpha \in X : |S_{\alpha}| = i\}$. Let $j$ be the unique index such that $X_j \in \mathcal{F}$. For each $\alpha \in X_j$, let $f_{\alpha} : S_\alpha \cong \{1, 2, \dots, j\}$ be a bijection. Then there is a unique bijection
	\[
	f : \left( \prod_{\alpha \in X} S_\alpha \right)_{\mathcal{F}} \to \{1, 2, \dots, j\}
	\]
	such that for all $(s_\alpha)_{\alpha \in X} \in \prod_{\alpha \in X} S_\alpha$
	\[
		\{\beta \in X_j : f_\beta(s_\beta) = f((s_\alpha)_{\alpha \in X})\} \in \mathcal{F}.
	\]
	Alternatively, having cardinality bounded by $M$ is expressible as a first order formula, hence \L o\'s' theorem \cite[Theorem 9.5.1]{hodges_model} implies that the ultraproduct of such sets has again cardinality
	bounded by $M$. One can consider the same functor for groups, rings, modules over a finitely generated ring, etc. since there are only finitely many isomorphisms classes of groups, rings, modules over a finitely generated ring, etc. with cardinality bounded by $M$. 
	
	\begin{lemma} \label{ultrafilter_quotient_lemma}
		Let $A$ be a ring of finite cardinality. There is a unique ring isomorphism
		$f : \left(\prod_{\alpha \in X} A\right)_\mathcal{F} \to A$ such that for each $(x_\alpha) \in \prod_{\alpha \in X} A$ we have
		\[
		\{\beta : x_\beta = f((x_\alpha))\} \in \mathcal{F}.
		\]
	\end{lemma}
	
	\begin{proof}
		Given $(x_{\alpha})_{\alpha} \in \prod_{\alpha \in X} A$
		there exists a unique $y \in A$ such that
		$\{\alpha \in X : x_\alpha = y\} \in \mathcal{F}$ 
		since we have the finite disjoint union
		\[
		X = \bigcup_{y \in A} \{\alpha \in X : x_\alpha = y\}.
		\]
		This uniquely characterises $f$ as the function
		$f((x_\alpha)) := y$. It is clear that $(x_\alpha) \sim 0$ if and only if 
		$f((x_\alpha)) = 0$.
		Moreover, $f$ admits a section
		$s : A \to \prod_{\alpha \in X} A : a \mapsto (a)_{\alpha \in X}$. Hence $f$
		is a bijection $\left(\prod_{\alpha} A \right)_\mathcal{F} \to A$.
		If $f((x_\alpha)) = y$
		and $f((z_\alpha)) = w$,
		then 
		\[
		\{\alpha : x_\alpha = y\} \cap \{\alpha : z_\alpha = w\} \subset \{\alpha : x_\alpha z_\alpha = yw\}.
		\]
		Since filters are closed under finite intersection and supersets, we have
		$f((x_\alpha z_\alpha)) = y w$ and similarly
		$f((x_\alpha + z_\alpha)) = y + w$.
		Thus, $f$ is a ring homomorphism. We already saw that it is bijective, hence $f$ is an isomorphism.
	\end{proof}
	
	\begin{lemma} \label{ultrafilter_localisation_lemma}
		Let $A$ be a local ring of finite cardinality with maximal ideal $\mathfrak{m}$. Let $f$ be the map from lemma \ref{ultrafilter_quotient_lemma}.
		The composition 
		$g : \prod_{\alpha \in X} A \to \left(\prod_{\alpha \in X} A \right)_\mathcal{F} \xrightarrow{f} A$ is a localisation\footnote{Meaning that it satisfies the universal property of localisation.} of $\prod_{\alpha \in X} A$ at $g^{-1}(\mathfrak{m})$ in $\prod_{\alpha \in X} A$. 
	\end{lemma}
	
	\begin{proof}
		Firstly, note that the complement of $g^{-1}(\mathfrak{m})$ maps to units under $g$ since $A$ is local. 
		Now let $h : \prod_{\alpha \in X} A \to B$ be any ring homomorphism such that the complement of $g^{-1}(\mathfrak{m})$ is mapped to $B^\times$. Appealing to the universal property of quotients it remains to show that $\ker g \subset \ker h$.
		If $(x_\alpha) \in \ker g$,
		then $S := \{\alpha : x_\alpha = 0\} \in \mathcal{F}$. Let 
		$1_S$ be the characteristic function on $S$, i.e. the element of $\prod_{\alpha \in X} A$ such that $1_{S,\alpha} = 1$ if $\alpha \in S$ and zero otherwise.
		By assumption $1_S \cdot (x_\alpha) = 0$. Moreover $g(1_S) = 1 \not \in \mathfrak{m}$, hence $1_S \not \in g^{-1}(\mathfrak{m})$. Finally, $h((x_\alpha)) = h(1_S)^{-1} \cdot h(1_S (x_\alpha)) = 0$
		and so $\ker g \subset \ker h$.
	\end{proof}
	
	\begin{lemma} \label{ultrafilter_exact_lemma}
		Let $A$ be a local ring of finite cardinality. The functor 
		\[
		\prod_{\alpha \in X} \mathbf{Mod}_A \to \mathbf{Mod}_{\left(\prod_{\alpha \in X} A \right)_\mathcal{F}} :
		\prod_{\alpha \in X} M_\alpha \mapsto \left(\prod_{\alpha \in X} M_\alpha \right)_\mathcal{F}
		\]	
		is exact.
	\end{lemma}
	
	\begin{proof}
		Let $f$ be the isomorphism from lemma \ref{ultrafilter_quotient_lemma} and $g = f \circ q$, where 
		\[
		q : \prod_{\alpha \in X} A \to \left(\prod_{\alpha \in X} A \right)_\mathcal{F}
		\] 
		is the quotient map. We have the commutative diagram of functors
		\[
		\begin{tikzcd}[column sep=huge, row sep = large]
			& \mathbf{Mod}_A \arrow[d, "f^*"] \\
			\prod\limits_{\alpha \in X} \mathbf{Mod}_A \arrow[r] \arrow[ru, "M\mapsto M/\ker(g)M"] & \mathbf{Mod}_{\left(\prod_{\alpha \in X} A \right)_\mathcal{F}}
		\end{tikzcd}
		\]
		The functor $f^*$ is an equivalence of categories and we have shown
		in lemma \ref{ultrafilter_localisation_lemma} that $M \mapsto M/\ker(g)M$
		is a localisation, hence an exact functor. Thus, the functor in question is exact itself.
	\end{proof}
	
	\begin{definition}
		Let $A$ be a ring. A family $(M_\alpha)_{\alpha \in X}$ of $A$-modules is called \emph{bounded} if there exists an integer $s$ such that each $M_\alpha$ is a quotient of $A^s$.	
	\end{definition}
	
	\begin{lemma} \label{finite_selection_lemma}
		Let $A$ be a ring of finite cardinality and $(M_\alpha)_{\alpha \in X}$ a bounded family of $A$-modules. There exists a set $S \in \mathcal{F}$ such that for $\beta \in S$, there is an isomorphism
		\[
		\left( \prod_{\alpha \in X} M_\alpha \right)_{\mathcal{F}} \cong M_{\beta}.
		\]
	\end{lemma}
	
	\begin{proof}
		If each $M_\alpha$ is a quotient of $A^s$, then $|M_\alpha| \leq |A|^s$. There are finitely many isomorphism classes of $A$-modules of cardinality at most $|A|^s$. Choose a set of representatives $\mathcal{R}$ for them.
		Since $\mathcal{R}$ is finite, there exists a unique $N \in \mathcal{R}$ such that
		\[
		S = \{\alpha \in X : M_\alpha \cong N\} \in \mathcal{F}.
		\]
		Let $f_\alpha : M_\alpha \to N$ be an isomorphism for $\alpha \in S$.
		It remains to find an isomorphism
		$\left(\prod_{\alpha \in X} M_\alpha \right)_\mathcal{F} \to N$.
		To define it, note that if $(m_\alpha) \in \prod_{\alpha \in X} M_\alpha$,
		then there exists a unique $n \in N$ such that 
		\[
		\{\alpha \in S : f_\alpha(m_\alpha) = n\} \in \mathcal{F} 
		\]
		and mapping $(m_\alpha) \mapsto n$ does the job. 
	\end{proof}
	
	\begin{lemma} \label{ultrafilter_hom_lemma}
		Let $A$ be a ring of finite cardinality and $(M_\alpha)_{\alpha \in X}, (N_\alpha)_{\alpha \in X}$ bounded families of 
		$A$-modules. The natural map 
		\[
		\left( \prod_{\alpha \in X} \Hom_A(M_\alpha,N_\alpha) \right)_{\mathcal{F}} \to \Hom_{A}\left(\left( \prod_{\alpha \in X} M_\alpha \right)_{\mathcal{F}},\left( \prod_{\alpha \in X} N_\alpha \right)_{\mathcal{F}}\right)
		\] 	
		is an isomorphism.
	\end{lemma}
	
	\begin{proof}
		By lemma \ref{finite_selection_lemma} there are sets $S_1, S_2 \in \mathcal{F}$ such that we have isomorphisms
		$M_{\beta} \cong \left( \prod_{\alpha \in X} M_\alpha \right)_{\mathcal{F}}$ for $\beta \in S_1$ and
		$N_{\beta} \cong \left( \prod_{\alpha \in X} N_\alpha \right)_{\mathcal{F}}$ for $\beta \in S_2$. Hence for $\beta \in S_1 \cap S_2 \in \mathcal{F}$ we have
		$\psi : \Hom_{A} \left( \left( \prod_{\alpha \in X} M_\alpha \right)_{\mathcal{F}}, \left( \prod_{\alpha \in X} M_{\alpha} \right)_{\mathcal{F}} \right) \cong \Hom_A(M_{\beta}, N_{\beta})$.
		The lemma will follow if we can check that the diagram
		\[
		\begin{tikzcd}
			\left( \prod_{\alpha \in X} \Hom_A(M_\alpha,N_\alpha) \right)_{\mathcal{F}} \arrow{r} \arrow{d}{\phi} & \Hom_{A}\left(\left( \prod_{\alpha \in X} M_\alpha \right)_{\mathcal{F}},\left( \prod_{\alpha \in X} N_\alpha \right)_{\mathcal{F}}\right) \arrow{d}{\psi} \\
			\Hom_A(M_{\beta}, N_{\beta}) \arrow{r}{\id} & \Hom_A(M_{\beta}, N_{\beta})
		\end{tikzcd}
		\]
		commutes, where $\phi$ is the isomorphism from lemma \ref{ultrafilter_quotient_lemma}. If 
		\[
		(f_\alpha)_{\alpha \in X} \in \prod_{\alpha \in X} \Hom_A(M_\alpha,N_\alpha),
		\]
		then
		$\phi((f_\alpha)) = f$, where
		$\{\alpha \in S_1 \cap S_2 : f_\alpha = f\} \in \mathcal{F}$. If $m \in M_{\beta}$, then $\psi((f_\alpha))(m) = n$,
		where $\{\alpha \in S_1 \cap S_2 : n = f_\alpha(m_\alpha)\} \in \mathcal{F}$ and $(m_\alpha)$ is any sequence such that $\{\alpha \in X : m_\alpha = m\} \in \mathcal{F}$. It follows that 
		$\{\alpha \in S_1 \cap S_2 : n = f_\alpha(m)\} \in \mathcal{F}$ and since $\{ \alpha \in S_1 \cap S_2 : f = f_\alpha\} \in \mathcal{F}$ we find that $f(m) = n$, as required.
	\end{proof}
	
	\subsection{Patching Functor}
	
	Consider a complete Noetherian local ring $R$ with maximal ideal $\mathfrak{m}$ and finite residue field $k := R/\mathfrak{m}$. We define the patching functor by
	\[
	\patch_R : \prod_{\alpha \in X} \mathbf{Mod}_R \to \mathbf{Mod}_R : (M_\alpha)_{\alpha \in X} \mapsto \lim_{\substack{\longleftarrow \\ n \geq 1}} \left(\prod_{\alpha \in X} M_\alpha/\mathfrak{m}^n \right)_{\mathcal{F}},
	\]
	where $M_\alpha / \mathfrak{m}^n$ is short for $M_\alpha / \mathfrak{m}^n M_\alpha$.
	The $R$-module structure is given by acting diagonally and makes $\patch_R$ into an $R$-linear functor.
	Note that this functor also depends on the ultrafilter $\mathcal{F}$ but we omit it from the notation. This is a special case of a \emph{cataproduct} as defined in \cite[Chapter 8]{schoutens_book}.
	
	\begin{lemma} \label{patch_R_lemma}
		If $M$ is a finitely generated $R$-module, then
		\[
		\patch_R \left( (M)_{\alpha \in X} \right) \cong M.
		\]
	\end{lemma}
	
	\begin{proof}
		Since $k$ is finite and $M$ is finitely generated, $M/\mathfrak{m}^n$ is finite, too. By lemma \ref{ultrafilter_quotient_lemma} there is a natural isomorphism $\left(\prod_{\alpha \in X} M/\mathfrak{m}^n \right)_\mathcal{F} \cong M/\mathfrak{m}^n$ for each $n$.
		Now we note that $M$ is complete as it is finitely generated and
		taking the inverse limit yields the desired isomorphism.
	\end{proof}
	
	\begin{lemma} \label{patch_finite_irrelevant_lemma}
		For each finite set $S \subset X$ the we obtain a non-principal ultrafilter $\mathcal{F}^S = \{ A \subset X \setminus S : A \in \mathcal{F}\}$ on $X \setminus S$.
		With respect to this ultrafilter, the functor $\patch_{R}$ factors through the category $\prod_{\alpha \in X \setminus S} \mathbf{Mod}_R$.
		\[
		\begin{tikzcd}
			\prod\limits_{\alpha \in X} \mathbf{Mod}_R  \arrow[r] \arrow{rd}[swap]{\patch_{R}} & \prod\limits_{\alpha \in X \setminus S} \mathbf{Mod}_R \arrow{d}{\patch_{R}} \\
			& \mathbf{Mod}_{R}
		\end{tikzcd}
		\]	
	\end{lemma}
	
	\begin{proof}
		This follows from the fact that $\mathcal{F}$ is non-principal.	
	\end{proof}
	
	\begin{lemma}[Selection Lemma] \label{selection_lemma}
		If $(M_\alpha)$ is a bounded family of $R$-modules, then
		\[
		\patch_R \left( M_\alpha \right) = \lim_{\longleftarrow} M_{\alpha_n}/\mathfrak{m}^n
		\]
		for some sequence $\alpha_n$ of elements of $X$ and transition maps 
		\[
		f_{n} : M_{\alpha_{n + 1}}/\mathfrak{m}^{n+1} \to M_{\alpha_{n+1}}/\mathfrak{m}^n \xrightarrow{\sim} M_{\alpha_{n}}/\mathfrak{m}^n.
		\]
	\end{lemma}
	
	\begin{proof}
		Since $k$ is finite and $\mathfrak{m}$ is finitely generated, $R/\mathfrak{m}^n$ is finite, too. Hence this is a corollary to 
		lemma \ref{finite_selection_lemma}.
	\end{proof}
	
	\begin{lemma} \label{inv_limit_tensor}
		Let $A$ be a ring and $(M_n, \phi_{ij})$ a countable inverse system of finite length $A$-modules and let $N$ be a finitely presented $A$-module. The natural map
		\[
		\left(\lim_{\longleftarrow} M_n \right) \otimes_A N \to \lim_{\longleftarrow} \left(M_n \otimes_A N \right)
		\]
		is an isomorphism.	
	\end{lemma}
	
	\begin{proof}
		Let $A^r \to A^s \to N \to 0$ be a finite presentation of $N$. Since tensor products are right exact we have an exact sequence
		\begin{equation} \label{inv_system_seq1}
			\left( \lim_{\longleftarrow} M_n \right) \otimes_A A^r \to \left( \lim_{\longleftarrow} M_n \right) \otimes_A A^s \to \left( \lim_{\longleftarrow} M_n \right) \otimes_A N \to 0
		\end{equation}
		On the other hand the finite length assumption implies that
		$\{ \ker (M_n \otimes_{A} A^r \to M_n \otimes_A A^s)\}$ is a 
		Mittag-Leffler system, hence
		\begin{equation} \label{inv_system_seq2}
			\lim_{\longleftarrow} \left( M_n \otimes_A A^r \right) \to \lim_{\longleftarrow} \left( M_n \otimes_A A^s \right) \to \lim_{\longleftarrow} \left( M_n \otimes_A N \right) \to 0
		\end{equation} 
		is exact.
		Note that the map in the claim of the lemma is an isomorphism when $N = A^d$ for some finite $d$ as can be
		checked directly by viewing the two limits as submodules of $(\prod_{n} M_n )^d \cong \prod_{n} M_n^d$. Hence the natural map from sequence \ref{inv_system_seq1} to sequence \ref{inv_system_seq2} induces the desired isomorphism. 
	\end{proof}
	
	\begin{lemma} \label{patch_quotient_lemma}
		If $(M_\alpha)$ is a bounded family of $R$-modules, then there are natural isomorphisms
		\[
		\patch_R \left( M_\alpha \right) / \mathfrak{m}^n \cong \patch_R(M_\alpha / \mathfrak{m}^n) \cong \left( \prod_{\alpha \in X} M_\alpha / \mathfrak{m}^n \right)_{\mathcal{F}}
		\]
	\end{lemma}
	
	\begin{proof}
		Firstly, note that under isomorphism from lemma \ref{ultrafilter_quotient_lemma} 
		the ideal $\mathfrak{m} \subset R/\mathfrak{m}^n$ is mapped to $\left( \prod_{\alpha \in X} \mathfrak{m} \right)_{\mathcal{F}}$, hence $\left( \prod_{\alpha \in X} M_\alpha / \mathfrak{m}^{n} \right)_{\mathcal{F}} / \mathfrak{m}^k = \left(\prod_{\alpha \in X} M_\alpha / (\mathfrak{m}^n + \mathfrak{m}^k) \right)_\mathcal{F}$.
		Moreover, each $\left( \prod_{\alpha \in X} M_\alpha/\mathfrak{m}^{n} \right)_{\mathcal{F}}$ is finite length over $R$ by lemma \ref{finite_selection_lemma}.
		Now lemma \ref{inv_limit_tensor} yields the result.
	\end{proof}
	
	\begin{lemma}\label{patch_right_exact_lemma}
		$\patch_{R}$ is right exact on the category of bounded families of $R$-modules.
	\end{lemma}
	
	\begin{proof}
		Suppose 
		we are given a short exact sequence 
		\[
		0 \to \prod_{\alpha \in X} A_\alpha \to \prod_{\alpha \in X} B_\alpha \to \prod_{\alpha \in X} C_\alpha \to 0
		\]
		Using lemma \ref{ultrafilter_exact_lemma} and the fact that tensor products are right exact we already have exact sequences
		\[
		\left( \prod_{\alpha \in X} A_\alpha/\mathfrak{m}^n \right)_\mathcal{F} \to \left(\prod_{\alpha \in X} B_\alpha/\mathfrak{m}^n\right)_\mathcal{F} \to \left(\prod_{\alpha \in X} C_\alpha/\mathfrak{m}^n\right)_\mathcal{F} \to 0
		\]
		for each $n$. Since $(A_\alpha)$ is bounded, a Mittag-Leffler argument
		shows that we also have exactness in the limit. 
	\end{proof}
	
	\begin{lemma} \label{patch_bounded_fg_lemma}
		If $(M_\alpha)$ is bounded, then $\patch_R \left( M_\alpha \right)$ is finitely generated.
	\end{lemma}
	
	\begin{proof}
		Let $s$ be sufficiently large and choose for each $\alpha$, a surjection
		$f_\alpha : R^s \to M_\alpha$. Then by lemma \ref{patch_R_lemma} and 
		lemma \ref{patch_right_exact_lemma} we have a surjection
		\[
		R^s \cong \patch_R(R^s) \xrightarrow{\patch_R(f_\alpha)} \patch_R(M_\alpha). \qedhere
		\]
	\end{proof}
	
	\begin{lemma} \label{patch_hom_lemma}
		If $(M_\alpha)$ and $(N_\alpha)$ are bounded families of $R$-modules, then there is a natural isomorphism
		\[
		\Hom_{R} \left( \patch_R \left( M_\alpha \right), \patch_R\left( N_\alpha \right) \right) \cong 
		\lim_{\substack{\longleftarrow \\ n}} \left( \prod_{\alpha \in X} \Hom_R(M_\alpha/\mathfrak{m}^n, N_\alpha / \mathfrak{m}^n) \right)_\mathcal{F}
		\]	
	\end{lemma}
	
	\begin{proof}
		For any finitely generated $R$-modules $M$ and $N$, the natural map 
		\[\Hom_R(M,N) \to \lim_{\substack{\leftarrow \\ n}} \Hom_R(M/\mathfrak{m}^n, N/\mathfrak{m}^n)\]
		is an isomorphism.
		Thus, lemmas \ref{patch_bounded_fg_lemma} and \ref{patch_quotient_lemma} imply that the natural map
		\[
		\Hom_{R} \left( \patch_R \left( M_\alpha \right), \patch_R\left( N_\alpha \right) \right) \to
		\lim_{\substack{\longleftarrow \\ n}} \Hom_{R} \left( \left( \prod_{\alpha \in X} M_\alpha/\mathfrak{m}^n \right)_\mathcal{F}, \left( \prod_{\alpha \in X} N_\alpha / \mathfrak{m}^n \right)_\mathcal{F} \right)
		\]
		is an isomorphism. Now lemma \ref{ultrafilter_hom_lemma} finishes the proof.
	\end{proof}
	
	\begin{lemma} \label{patching_change_ring_lemma}
		Let $R'$ be another complete Noetherian local ring with maximal ideal $\mathfrak{m}'$ and finite residue field $k'$ and suppose that $(M_\alpha)$ is a bounded family of $R$-modules such that each $M_\alpha$ is also an $R'$-module. If the action of $R$ on $M_\alpha$ is given by a local ring homomorphism $f_\alpha : R \to R' / \Ann_{R'}(M_\alpha)$, then the natural map $\patch_{R}(M_\alpha) \to \patch_{R'}(M_\alpha)$ is an isomorphism of $R$-modules, where the $R$-module structure on the second term comes from functoriality of $\patch_{R'}$.
	\end{lemma}	
	
	\begin{proof}
	We first show surjectivity. The assumption yields the inclusions 
	\[
	\mathfrak{m}^n M_\alpha \subset (\mathfrak{m}')^n M_\alpha
	\] 
	for all $n \geq 1$ and $\alpha \in X$. Note that the kernel of the surjective $R$-module map
		\[
			\left( \prod_{\alpha \in X} M_{\alpha} / \mathfrak{m}^n \right)_{\mathcal{F}} \twoheadrightarrow \left( \prod_{\alpha \in X} M_{\alpha} / (\mathfrak{m'})^n \right)_{\mathcal{F}}
		\]
		has finite length for each $n \geq 1$. Hence a Mittag-Leffler argument shows that the map 
		$\patch_{R}(M_\alpha) \to \patch_{R'}(M_\alpha)$ is surjective, too. Similarly to lemma \ref{selection_lemma} we can write the map in question as the limit (with $R'$-linear transition maps) of the maps
		\[
		M_{\alpha_n} / \mathfrak{m}^n \to M_{\alpha_n} / (\mathfrak{m}')^n
		\]
		for some sequence $\alpha_n$. Now let 
		\[
		(m_n) \in \lim_{\leftarrow} M_{\alpha_n}/\mathfrak{m}^n
		\] 
		be an element of the kernel. Then $m_n \in (\mathfrak{m}')^n M_{\alpha_n}$ for each $n \geq 1$ and by using the transition maps $M_{\alpha_{n+1}} \to M_{\alpha_n}$ we see that 
		$m_n \in \bigcap_{k \geq 1} (\mathfrak{m}')^{k} M_{\alpha_n} + \mathfrak{m}^n M_{\alpha_n}$. Since $M_{\alpha_n}$ is finite over $R'$, this implies that $m_n \in \mathfrak{m}^n M_{\alpha_n}$. Thus, the map is injective, too.
	\end{proof}
	
	\subsection{Derived Patching Functor}
	
	Keep the notations from the previous section. We have shown in lemma \ref{patch_right_exact_lemma} that $\patch_R$ is a right exact functor on the category of bounded families of $R$-modules, viewed as a full subcategory of $\prod_{\alpha \in X} \mathbf{Mod}_R$. Unfortunately, this category is not necessarily abelian since it is not necessarily closed under subobjects. Thus, the standard methods of homological algebra do not apply here. Instead we will only prove the minimal functoriality property needed below but with more care it should be possible to develop a more general derived patching functor.
	
	\begin{definition}
		Let $e : X \to \ints_{\geq 1} \cup \{\infty\}$ be a function. Then we say a family $(M_\alpha)_{\alpha \in X}$ in $\prod_{\alpha \in X} \mathbf{Mod}_R$ is an $e$-family if $\mathfrak{m}^{e(\alpha)} M_\alpha = 0$. Equivalently, $(M_\alpha)_{\alpha \in X}$ is an $e$-family if it comes from an object in $\prod_{\alpha \in X} \mathbf{Mod}_{R/\mathfrak{m}^{e(\alpha)}}$. We say it is $e$-free if each $M_\alpha$ is a free $R/\mathfrak{m}^{e(\alpha)}$-module, where we use the convention that $\mathfrak{m}^\infty = 0$.
	\end{definition}
	
	\begin{lemma} \label{patch_efamily_invariance_lemma}
		Let $e : X \to \ints_{\geq 1} \cup \{\infty\}$ be a function such that $\{\alpha \in X : e(\alpha) > N \} \in \mathcal{F}$ for all integers $N$. If $(M_\alpha)$ is a bounded family of $R$-modules,
		then the natural map $(M_\alpha) \to (M_\alpha/\mathfrak{m}^{e(\alpha)})$ induces an isomorphism $\patch_R(M_\alpha) \to \patch_R(M_\alpha/ \mathfrak{m}^{e(\alpha)})$.
	\end{lemma}
	
	\begin{proof}
		By lemma \ref{patch_bounded_fg_lemma} we know that 	$\patch_R(M_\alpha / \mathfrak{m}^{e(\alpha)})$ is a finitely generated $R$-module. Hence with lemma \ref{patch_quotient_lemma} we find isomorphisms
		 \begin{linenomath*} \begin{align*}
			\patch_R(M_\alpha / \mathfrak{m}^{e(\alpha)}) & \cong \lim_{N} \patch_R(M_\alpha / \mathfrak{m}^{e(\alpha)} M_\alpha)/ \mathfrak{m}^N \\
			&\cong \lim_{N} \patch_R(M_\alpha / \mathfrak{m}^{\min(N,e(\alpha))} M_\alpha) \\
			& \cong \lim_N \patch_R(M_\alpha / \mathfrak{m}^N M_\alpha) \\
			& \cong \patch_R(M_\alpha)
		\end{align*} \end{linenomath*} 
	whose inverse is the map induced by $(M_\alpha) \to (M_\alpha/\mathfrak{m}^{e(\alpha)})$.
	\end{proof}	
	
	\begin{lemma} \label{patch_efree_lemma}
		Let $e : X \to \ints_{\geq 1} \cup \{\infty\}$ be a function such that $\{\alpha \in X : e(\alpha) > N \} \in \mathcal{F}$ for all integers $N$. If $(M_\alpha)$ is a bounded $e$-free family of $R$-modules,
		then $\patch_{R}(M_\alpha)$ is a  finitely generated free $R$-module.
	\end{lemma}
	
	\begin{proof}
		By lemma \ref{patch_bounded_fg_lemma} we know that $\patch_R(M_\alpha)$ is a finitely generated $R$-module. Moreover, lemma \ref{patch_quotient_lemma} implies that 
		\[
			\patch_R(M_\alpha)/\mathfrak{m}^N \cong \patch_R(M_\alpha / \mathfrak{m}^N)
		\]
		is a free $R/ \mathfrak{m}^N$ module for each $N \geq 1$. Hence $\patch_R(M_\alpha)$ is free over $R$, too.
	\end{proof}	
	
	\begin{lemma} \label{patch_derived_natural}
		Suppose that $(C_\alpha^\bullet)_{\alpha \in X}$ is a bounded complex of bounded $e$-free families of $R$-modules, where $e : X \to \ints_{\geq 1} \cup \{\infty\}$ is a function such that $\{\alpha \in X : e(\alpha) > N \} \in \mathcal{F}$ for all integers $N$, then 
		\begin{itemize}
			\item $\patch_{R}(C_\alpha^\bullet)$ is a bounded complex of finitely generated free $R$-modules.
			\item If $(D_\alpha)_{\alpha \in X}$ is another such complex and $f_\alpha^\bullet : C_\alpha^\bullet \to D_\alpha^\bullet$ is a sequence of nullhomotopies, then $\patch_R(f_\alpha^\bullet)$ is a nullhomotopy.
\end{itemize}				
	\end{lemma}	
	
	\begin{proof}
		The first bullet point is lemma \ref{patch_efree_lemma}. The second bullet point holds because $\patch_R$ is an additive functor.
	\end{proof}
	
	\subsection{Depth and Dimension}
	
	Here we state some results from commutative algebra that will be used to prove dimension inequalities in the main lemma.
	
	\begin{definition}
		Let $R$ be a Noetherian local ring. If $M$ is a finitely generated module over $R$, then we define $\depth_R(M)$ as the supremum of the lengths of sequences of elements $x_1, \dots, x_n \in \mathfrak{m}_R$ such that
		$M/(x_1,\dots,x_n)M \neq 0$ and $x_i$ is not a zero divisor on
		$M/(x_1, \dots, x_{i-1})$.
	\end{definition}
	
	\begin{lemma}
		Let $R$ be a Noetherian local ring and $M$ a non-zero finitely generated $R$-module, then $\depth_R(M)$ is the smallest integer $i$ such that
		$\Ext^i_{R}(R/\mathfrak{m}_R,M) \neq 0$.
	\end{lemma}
	
	\begin{proof}
		See \cite[\href{https://stacks.math.columbia.edu/tag/00LW}{Lemma 00LW}]{stacks-project}.
	\end{proof}
	
	\begin{lemma}[Auslander--Buchsbaum formula] \label{auslander_buchsbaum}
		Let $R$ be a Noetherian local ring and $M$ a finitely generated $R$-module of finite projective dimension, then
		\[
		\operatorname{pd}_R(M) + \depth_R(M) = \depth(R)
		\]
	\end{lemma}
	
	\begin{proof}
		See \cite[\href{https://stacks.math.columbia.edu/tag/090V}{Proposition 090V}]{stacks-project}.
	\end{proof}
	
	\begin{lemma} \label{cohom_dim_bound}
		Let $l_0 \geq 0$ and $q_0$ be integers and $S$ a Cohen--Macaulay local ring of dimension $n \geq l_0$. Let $C^\bullet$ be a bounded complex of finite free $S$-modules which is not exact. Suppose
		that $H^i(C^\bullet/\mathfrak{m}_S)$ is non-zero only if $i \in [q_0, q_0 + l_0]$. Then $\dim_S H^*(C^\bullet) \geq \dim S - l_0$ and if equality holds, then $H^i(C^\bullet)$ is non-zero only if $i = q_0 + l_0$ and $H^{q_0 + l_0}(C^\bullet)$ has projective dimension $l_0$.	
	\end{lemma}
	
	\begin{proof}
		The idea goes back to \cite{calegari_geraghty_beyond}. See \cite[Lemma 2.9]{khare_thorne} for a detailed proof.
	\end{proof}
	
	\begin{lemma} \label{depth_le_depth_max}
		Let $(R, \mathfrak{m})$ be a Noetherian local ring and $R \to S$ a finite ring map. If $\mathfrak{n}$ is a maximal ideal of $S$ and $A$ is a finitely generated $S$-module, then
		\[
		\depth_{R} (A) \leq \depth_{S_\mathfrak{n}} (A_{\mathfrak{n}}).
		\]
	\end{lemma}
	
	\begin{proof}
		See \cite[\href{https://stacks.math.columbia.edu/tag/0AUK}{Lemma 0AUK}]{stacks-project}.
	\end{proof}
	
	\begin{lemma} \label{depth_completion}
		Let $R$ be a Noetherian local ring with maximal ideal $\mathfrak{m}$ and $\widehat{R}$ its $\mathfrak{m}$-adic completion. Then for any finitely generated $R$-module $M$ we have
		\[
		\depth_R(M) = \depth_{\widehat{R}}(M \otimes_R \widehat{R})
		\]
	\end{lemma}
	
	\begin{proof}
		Since $R$ is Noetherian, $\widehat{R}$ is a flat $R$-module and so 
		\cite[\href{https://stacks.math.columbia.edu/tag/0338}{Lemma 0338}]{stacks-project} applies.
	\end{proof}
	
	\begin{lemma} \label{depth_eq_finite_ext}
		Let $S,R$ be Noetherian local rings and $f : S \to R$ a local ring homomorphism. If $M$ is an $R$-module which is finite over $S$, then 
		\[
		\depth_S(M) = \depth_R(M)
		\]
	\end{lemma}
	
	\begin{proof}
		This proof is from \cite[IV. 16.4.8]{EGA} but we include it for convenience. First suppose that $\depth_{S} (M) = 0$, then $0 \neq \Hom_{S}(S/\mathfrak{m}_S, M)$.
		Now 
		\[
		P := \Hom_S(S/\mathfrak{m}_S, M) \subset \Hom_S(S,M) = M
		\]
		is a sub $S$-module of $M$. Since $M$ is finite over $S$, $P$ has finite $S$-length. But $P$ is also naturally an $R$-submodule of $M$
		and has finite length over $R$ since the action of $S$ factors through $R$. Hence $P$ has a simple $R$-submodule which must be isomorphic to 
		$R/\mathfrak{m}_R$, hence $\Hom_R(R/\mathfrak{m}_R, M) \neq 0$ and 
		$\depth_R(M) = 0$ as well.
		
		If $\depth_S(M) > 0$, then we may choose an element $x \in \mathfrak{m}_S$ such that $M/f(x)M \neq 0$ and $f(x)$ is not a zero divisor on $M$. Then $\depth_S(M/xM) = \depth(M) - 1$ and $\depth_R(M/f(x)M) = \depth(M) - 1$, so the claim follows by induction. 
	\end{proof}
	
	\begin{lemma} \label{dim_le_finite_ext}
		Let $R$ be a Noetherian ring and $S \subset R$ a Noetherian subring over which $R$ is integral. Then 
		$\dim (S) = \dim(R)$.
	\end{lemma}
	
	\begin{proof}
		See \cite[13.C]{matsumura}.
	\end{proof}
	
	\begin{lemma} \label{faithful_cancellation_lemma}
		Let $k$ be a field and $R = k[[x_1, \dots, x_n]]$. If $\mathfrak{a} < R$ is an ideal and $M$ is a finitely generated faithful $R$-module such that $\mathfrak{a} M = \mathfrak{m} M$, then $\mathfrak{a} = \mathfrak{m}$, where $\mathfrak{m} = (x_1, \dots, x_n)$ is the maximal ideal of $R$.
	\end{lemma}
	
	\begin{proof}
		We can write $\mathfrak{a} + \mathfrak{m}^2 = I + \mathfrak{m}^2$, where  $I < R$ is an ideal generated by linear polynomials. It suffices to show that $I = \mathfrak{m}$ because then
		\[
		\mathfrak{a}/\mathfrak{m} \mathfrak{a} \twoheadrightarrow \mathfrak{a}/(\mathfrak{a} \cap \mathfrak{m}^2) = I/(I \cap \mathfrak{m}^2) = \mathfrak{m}/\mathfrak{m}^2
		\] 
		which implies that $\mathfrak{a} = \mathfrak{m}$ by Nakayama's lemma.  
		
		By construction $\mathfrak{a} \subset I + \mathfrak{m}^2$, hence
		\[
		\mathfrak{m} M = \mathfrak{a} M \subset I M + \mathfrak{m}^2 M \subset \mathfrak{m} M.
		\]
		It follows that $\mathfrak{m} (M/IM) = \mathfrak{m}^2 (M/IM)$. By Nakayama's lemma this implies that $\mathfrak{m} (M/ IM) = 0$, i.e. $\mathfrak{m} M \subset I M$. Now \cite[Proposition 2.4]{atiyah_macdonald} implies that for each $x \in \mathfrak{m}$, the endomorphism $\phi_x : M \to M : m \mapsto xm$ satisfies
		an equation of the form 
		\[
		\phi_x^k + a_{k - 1} \phi_x^{k - 1} + \dots + a_0 = 0
		\]
		with $a_i \in I$. Since $M$ is faithful, $x$ itself satisfies the same equation and 
		$x^k \in I$.
		Note that $I < R$ is prime since it is generated by linear elements, hence $x \in I$ as well. Thus, $I = \mathfrak{m}$ as desired.
	\end{proof}
	
	\subsection{Main Lemma}
	
	We begin with a standard lemma which forms a crucial ingredient of the argument below.
	
	\begin{lemma}[Formal Implicit Function Theorem] \label{formal_implicit_function_lemma}
		Let $E$ be a field of characteristic 0 and $R$ be a regular local $E$-algebra. Then for every complete noetherian local $E$-algebra $T$ and local $E$-algebra map $R \to T/I$, where  $I \subset T$ is a nilpotent ideal, there exists a lift $R \to T$ as indicated in the following diagram.
		\[
		\begin{tikzcd}
			& T \arrow{d} \\
			R \arrow{r} \arrow[dashed]{ru} & T/I 
		\end{tikzcd}
		\]
	\end{lemma}
	
	\begin{proof}
		By \cite[\href{https://stacks.math.columbia.edu/tag/07EJ}{Lemma 07EJ}]{stacks-project} the map $E \to R$ is formally smooth in the $\mathfrak{m}_R$-adic topology. Now the lemma follows directly from \cite[\href{https://stacks.math.columbia.edu/tag/07NJ}{Lemma 07NJ}]{stacks-project}.
	\end{proof}
	
	\begin{example}
		Consider a non-constant polynomial $F \in \reals[x,y]$ with a solution $(x_0,y_0) \in \reals^2$, and consider the ring $R = \reals[x,y]_{(x-x_0,y-y_0)}/(F)$ and take $T = \reals[t]/(t^N)$ and $I = (t^2)$ for some $N > 1$. Then the lemma follows from the implicit function theorem since $R$ is regular if and only if $dF_{(x_0,y_0)}$ is surjective.
	\end{example}
	
	\begin{definition}
		Let $R$ be a local ring with maximal ideal $\mathfrak{m}$. A complex $C^\bullet$ of $R$-modules is called minimal if the differentials of $C^\bullet \otimes_{R} R/\mathfrak{m}$ are zero.
	\end{definition}
	
	\begin{lemma}[Main Lemma] \label{patching_lemma}
		Let $S_\infty = W[[x_1, \dots, x_r]]$, where $W$ is a complete discrete valuation ring with finite residue field of characteristic $p$ and uniformizer $\varpi$ and set $\mathfrak{a} = (x_1, \dots, x_r)$ and $\mathfrak{a}_N = ((1 + x_1)^{p^N} - 1, \dots, (1 + x_r)^{p^N} - 1)$ and $S_N = S_{\infty}/\mathfrak{a}_N$ for $N \geq 0$.
		Suppose we are given 
		\begin{enumerate}[(1)]
			\item A complete noetherian local $W$-algebra $R_\infty$.
			\item For each $N \geq 0$, a local $S_N$-algebra $R_N$ which is a quotient of $R_\infty$ such that $R_N / \mathfrak{a} = R_0$.
			\item \label{givenCN} For each $N \geq 0$, a minimal complex $C_N^\bullet$ of free $S_N$-modules such that $C_N^\bullet / \mathfrak{a} = C_0^\bullet$ and
			$C_0^\bullet$ is finitely generated.
			\item Integers $q_0,l_0$ such that $C_0^\bullet[1/p]$ is not exact and concentrated in degrees $[q_0, q_0 + l_0]$ and $\dim R_\infty[1/p] + l_0 \leq \dim (S_\infty)_{\mathfrak{a}} = r$.
			\item \label{givenTN} For each $N \geq 0$, a commutative $S_N$-subalgebra $T_{N} \subset \End_{\mathbf{D}(S_N)}(C_N^\bullet)$ whose image in $\End_{\mathbf{D}(S_0)}(C_0^\bullet)$ is contained in $T_0$.
			\item A constant $\delta > 0$ such that for each $N \geq 0$, there is an ideal $I_N$ of $T_N$ satisfying $I_N^\delta = 0$ with $I_0$ equal to the nilradical of $T_0$ and such that there exists a surjective $S_N$-algebra homomorphism $R_N \twoheadrightarrow T_N/I_N$ making the square
			\[
				\begin{tikzcd}
					R_N \arrow{r} \arrow{d} & T_N/I_N \arrow{d} \\ 
					R_0 \arrow{r} & T_0/I_0
				\end{tikzcd}
			\]
			commute.
			\item A maximal ideal $\mathfrak{q}$ of $T_0[1/p]$ with image $\mathfrak{p}$ in $R_0$ under the map 
			\[
			(\Spec T_0[1/p])_{red} \to (\Spec T_0)_{red} \cong (\Spec (T_0/I_0))_{red} \to (\Spec R_0)_{red} \hookrightarrow (\Spec R_\infty)_{red}
			\] 
			such that $\widehat{(R_\infty)_\mathfrak{p}}$ is regular and $(T_0)_{\mathfrak{q}}$ is a field.
		\end{enumerate}
		Then $\dim \widehat{(R_\infty)_{\mathfrak{p}}} = r - l_0$ and there is an isomorphism $(R_0)_\mathfrak{p} \cong (T_0)_\mathfrak{q}$ which is compatible with the map $R_0 \to T_0/I_0$.
	\end{lemma}
	
	\begin{proof}
		Let $\mathcal{F}$ be a non-principal ultrafilter on $\ints_{\geq 0}$ and set
		 \[
			C_\infty^\bullet := \patch_{S_\infty} \left( C_N^\bullet \right) ,\quad
			T_\infty' :=
			\patch_{S_\infty} \left( T_{N} \right) ,\quad
			I_\infty' := \patch_{S_\infty} \left( I_N \right).
		\]
		Then 
		\begin{enumerate}[(a)]
			\item Using \eqref{givenCN} and lemma \ref{patch_derived_natural} we find that $C_\infty^\bullet$ is a minimal complex of finitely generated free $S_\infty$-modules. By right exactness of $\patch_{S_\infty}$ we have
			$C_\infty^\bullet \otimes_{S_\infty} S_0 = C_0^\bullet$.
			\item From \eqref{givenTN} we obtain a commutative square for each $N \geq 1$
			\[
			\begin{tikzcd}
				T_N \arrow{r} \arrow{d} & \End_{\mathbf{D}(S_N/\mathfrak{m}_{S_N}^{e(N)})}(C_N^\bullet/\mathfrak{m}^{e(N)}) \arrow{d} \\
				T_0 \arrow{r} & \End_{\mathbf{D}(S_0/\mathfrak{m}_{S_0}^{e(N)})}(C_0^\bullet/\mathfrak{m}^{e(N)})
			\end{tikzcd}
			\]
			and by the second bullet point of lemma \ref{patch_derived_natural} we obtain a commutative square
			\[
			\begin{tikzcd}
				T_\infty' \arrow{d} \arrow{r} & \End_{\mathbf{D}(S_\infty)}(C_\infty^\bullet) \arrow{d} \\
				T_0 \arrow{r} &  \End_{\mathbf{D}(S_0)}(C_0^\bullet)
			\end{tikzcd}
			\]
			We denote the image of $T_\infty' \to \End_{\mathbf{D}(S_\infty)}(C_\infty^\bullet)$ by $T_\infty$. Since the family $(C_N^\bullet)$ is bounded (as $C_N^\bullet/\mathfrak{m}_{S_N} \cong C_0/\varpi$), we find that $C_\infty^\bullet$ is a finitely generated $S_\infty$-module
			by lemma \ref{patch_bounded_fg_lemma} and so is $\End_{\mathbf{D}(S_\infty)}(C_\infty^\bullet)$. 
			Hence $T_\infty$ is a finite $S_\infty$-algebra as $S_\infty$ is Noetherian.
			From the diagram we also conclude that 
			the image of $T_\infty$ in $\End_{\mathbf{D}(S_0)}(C_0^\bullet)$
			is contained in $T_0$. 
			\item We define $I_\infty$ as the image of $I_\infty'$ in $T_\infty$. It satisfies $I_\infty^\delta = 0$ since that can be described by a first order formula which holds for each $I_N$.
			\item By the right exactness of $\patch_{R_\infty}$, we find a $W$-algebra surjection 
			$R_\infty \twoheadrightarrow \patch_{R_\infty}(R_N)$.
			Moreover, the $S_\infty$-algebra homomorphism $\patch_{R_\infty}(R_N) \twoheadrightarrow \patch_{R_\infty}(T_N/I_N)$ is surjective. Lemma \ref{patching_change_ring_lemma} gives a natural isomorphism of $S_\infty$-algebras
			$\patch_{R_\infty}(T_N/I_N) \cong \patch_{S_\infty}(T_N/I_N) \cong T_\infty'/I'_\infty$.
			Together with the commutative diagram from assumption $(6)$ we obtain a commmutative diagram
			\[
			\begin{tikzcd}
				R_\infty \arrow[two heads]{d} \arrow[two heads]{r} & \patch_{R_\infty}(R_N) \arrow[two heads]{r} & T_\infty/I_\infty \arrow{d} \\
				R_0  \arrow[two heads]{rr} & & T_0/I_0
			\end{tikzcd}
			\]
			where all arrows except those starting at $R_\infty$ are $S_\infty$-linear.
		\end{enumerate}

		We apply lemma \ref{cohom_dim_bound} to $(C_\infty^\bullet)_{\mathfrak{a}}$,
		noting that $(C_\infty^\bullet)_{\mathfrak{a}}/\mathfrak{a} = C_0^\bullet[1/p]$ is
		concentrated in degrees $[q_0, q_0 + l_0]$.
		Hence $\dim_{(S_\infty)_{\mathfrak{a}}} H^{q_0 + l_0}((C_\infty^\bullet)_{\mathfrak{a}}) \geq \dim (S_\infty)_{\mathfrak{a}} - l_0$ and if equality holds, then $H^{q_0 + l_0}((C_\infty^\bullet)_{\mathfrak{a}})$ has projective dimension $l_0$.
		
		On the other hand, since $C_\infty^\bullet$ is finitely generated, the map
		\[
		(S_\infty)_\mathfrak{a}/\Ann_{(S_\infty)_\mathfrak{a}} (H^{q_0 + l_0}(C_\infty^\bullet)_\mathfrak{a}) \hookrightarrow
		(T_\infty)_{\mathfrak{a}}/\Ann_{(T_\infty)_\mathfrak{a}}(H^{q_0 + l_0}(C_\infty^\bullet)_\mathfrak{a})
		\]
		is finite and with lemma \ref{dim_le_finite_ext} we conclude 
		\[
			\dim_{(S_\infty)_{\mathfrak{a}}} H^{q_0 + l_0}(C_\infty^\bullet)_\mathfrak{a} = \dim_{(T_\infty)_\mathfrak{a}} H^{q_0 + l_0}(C_\infty^\bullet)_\mathfrak{a}.
		\]
		Since Krull dimension decreases with surjective ring homomorphisms and is insensitive to nilpotent thickenings we find that
		\[
			\dim_{(S_\infty)_{\mathfrak{a}}} H^{q_0 + l_0}(C_\infty^\bullet)_\mathfrak{a}  \leq \dim (T_\infty)_\mathfrak{a} = \dim (T_\infty)_\mathfrak{a}/I_\infty
			 \leq \dim \patch_{R_\infty} (R_N)_\mathfrak{a}.
		\]
		Moreover, $\patch_{R_\infty}(R_N)_{\mathfrak{a}}$ is a localisation of $\patch_{R_\infty}(R_N)[1/p]$ since $p \not\in \mathfrak{a}$. Hence
		\[
		\dim \patch_{R_\infty}(R_N)_{\mathfrak{a}} \leq \dim \patch_{R_\infty}(R_N)[1/p].
		\]
		With the bound above and the surjection $R_\infty \to \patch_{R_\infty}(R_N)$ we conclude
		\[
		\dim_{(S_\infty)_{\mathfrak{a}}} H^{q_0 + l_0}(C_\infty^\bullet)_\mathfrak{a} \leq \dim R_\infty[1/p] \leq \dim (S_\infty)_{\mathfrak{a}} - l_0,
		\]
		hence we must have
		equalities throughout. Thus, the projective dimension of the $(S_\infty)_{\mathfrak{a}}$-module
		$H^{q_0 + l_0}((C_\infty^\bullet)_{\mathfrak{a}})$ is $l_0$ and by the Auslander--Buchsbaum formula (lemma \ref{auslander_buchsbaum}), we find that
		$\depth_{(S_\infty)_\mathfrak{a}} H^{q_0 + l_0}((C_\infty^\bullet)_{\mathfrak{a}}) = \dim (S_\infty)_{\mathfrak{a}} - l_0 = \dim R_\infty[1/p]$.
		Now $(C_\infty^\bullet)_\mathfrak{a}$ is a projective resolution of $H^{q_0 + l_0}(C_\infty^\bullet)_{\mathfrak{a}}$ so that in fact we can view
		$(T_\infty)_{\mathfrak{a}} \subset \End_{(S_\infty)_{\mathfrak{a}}}(M_{\mathfrak{a}})$, where
		$M := H^{q_0 + l_0}(C_\infty^\bullet)$.
		
		Denote the kernel of $T_\infty \to T_0/\mathfrak{q}$ also by
		$\mathfrak{q}$, then $\mathfrak{a} T_\infty \subset \mathfrak{q}$ and
		the image of $\mathfrak{p} < R_\infty$ in $T_\infty/ I_\infty$ equals 
		$\mathfrak{q}$ as can be seen from the diagram
		\[
		\begin{tikzcd}
			R_\infty \arrow{r} \arrow[two heads]{d} & R_0 \arrow{r} \arrow{d} & R_0/\mathfrak{p}  \arrow[hook]{d} \\
			T_\infty/I_\infty \arrow{r} & T_0/I_0 \arrow{r} & T_0/\mathfrak{q}
		\end{tikzcd}
		\]
		Thus, we obtain a surjective local map $(R_\infty)_{\mathfrak{p}} \to (T_\infty)_{\mathfrak{q}}/I_\infty$, hence also a surjective local map on completions $\widehat{(R_\infty)_{\mathfrak{p}}} \to \widehat{ (T_\infty)_{\mathfrak{q}}/I_\infty}$.
		To simplify the notation we define 
		 \begin{linenomath*} \begin{align*}	
			\widehat{R} & := \widehat{(R_\infty)_{\mathfrak{p}}}, \\
			\widehat{T} & := \widehat{(T_\infty)_{\mathfrak{q}}}, \\
			\widehat{I} & := \ker( \widehat{(T_\infty)_{\mathfrak{q}}} \to \widehat{(T_\infty)_{\mathfrak{q}}/I_\infty}).
		\end{align*} \end{linenomath*} 
		Since $I_\infty^\delta = 0$, we have $(I_\infty + \mathfrak{q}^n)^\delta \subset \mathfrak{q}^n$ for all $n \geq 1$ and thus $\widehat{I}^\delta = 0$ as well.
		Since $\widehat{R} = \widehat{(R_\infty)_\mathfrak{p}}$ is a regular $W[1/p]$-algebra, we can apply Lemma \ref{formal_implicit_function_lemma} to choose a lifting
		\[
		\begin{tikzcd}
			&  \widehat{T}\arrow{d} \\
			\widehat{R} \arrow{r} \arrow{ru} & \widehat{T}/\widehat{I}
		\end{tikzcd}
		\]
		We use it to equip $\widehat M := M_{\mathfrak{a}} \otimes_{(T_\infty)_{\mathfrak{a}}} \widehat{(T_\infty)_{\mathfrak{q}}}$ with a $\widehat{R}$-module structure. This is the key step for dealing with the nilpotent ideals. 
		
		Note that $\widehat{R} \to \widehat{T}$ is a finite local ring map since $\widehat{R} \to \widehat{T}/ \widehat{I}$ is surjective and $\widehat{I}$ is nilpotent and finitely generated, i.e. $\widehat{T}$ is a quotient of
		$\widehat{R}[\epsilon_1, \dots, \epsilon_t]/(\epsilon_i^N)$. Hence lemma \ref{depth_eq_finite_ext} implies that
		\[
		\depth_{\widehat{R}} \widehat{M} = \depth_{\widehat{T}} \widehat{M} = \depth_{(T_\infty)_{\mathfrak{q}}} (M_{\mathfrak{a}} \otimes_{(T_\infty)_{\mathfrak{a}}} (T_\infty)_{\mathfrak{q}})
		\]
		where the second equality follows from lemma \ref{depth_completion}. Now lemma \ref{depth_le_depth_max} implies
		\[
		\depth_{\widehat{R}} \widehat{M} \geq \depth_{(S_\infty)_\mathfrak{a}} M_\mathfrak{a} \geq \dim R_\infty[1/p] \geq \dim (R_{\infty})_{\mathfrak{p}} = \dim \widehat{R}.
		\]
		
		Applying the Auslander--Buchsbaum formula to the regular ring $\widehat{R}$, we find that $\widehat{M}$ is a projective, hence free $\widehat{R}$-module. It is non-zero since $\mathfrak{q}$ is in the support of $\widehat{M} / \mathfrak{a}$ as $C_0$ is not exact. In particular, $\dim \widehat{R} = \dim R_\infty[1/p] = r - l_0$. 
		
		Since $\widehat{M}$ is free over $\widehat{R}$, the map 
		$\widehat{R} \to \widehat{T}$ is injective.
		The regularity of $\widehat{R}$ implies that it is reduced, hence $\widehat{R} \to \widehat{T}/\widehat{I}$ is injective as well. 
		By construction it is surjective and factors through the map $\widehat{\patch_{R_\infty}(R_N)_{\mathfrak{p}}} \to \widehat{T}/ \widehat{I}$. Thus, 
		we have isomorphisms
		\[
		\widehat{R} \cong \widehat{\patch_{R_\infty}(R_N)_{\mathfrak{p}}} \cong \widehat{T} / \widehat{I}
		\]
		and we equip $\widehat{R}$ with an $S_\infty$-algebra structure via these isomorphisms.
		
		$(T_0)_{\mathfrak{q}}$ is a field by assumption $(7)$, i.e. $\mathfrak{q} (T_0)_{\mathfrak{q}} = 0$. Recall that from (b) above we know that the image of $T_\infty$ in $\End_{S_0}(C_0)$ is contained in $T_0$. Consequently, $\mathfrak{q} T_\infty$ acts as $0$ on  
		\[
		\widehat{H^{q_0 + l_0}(C_0)_{\mathfrak{q}}} \cong \widehat{M} / \mathfrak{a} \widehat{M}
		\]
		and $\Ann_{\widehat{T}}(\widehat{M}/ \mathfrak{a} \widehat{M})$ must be the maximal ideal of $\widehat{T}$. Thus,
		\[
		\mathfrak{m}_{\widehat{T}} \widehat{M} \subset \mathfrak{a} \widehat{M} \subset \mathfrak{m}_{\widehat{T}} \widehat{M}.
		\]
		In particular, we have an equality of $\widehat{R}$-modules
		\[
		\mathfrak{a} \left( \widehat{M}/ \widehat{I} \widehat{M} \right) = \mathfrak{m}_{\widehat{R}} \left( \widehat{M} / \widehat{I}\widehat{M} \right) .
		\]
		The module $\widehat{M}/ \widehat{I} \widehat{M}$ is faithful over $\widehat{R}$
		since if $r \in \widehat{R}$ kills it, then any lift $\tilde{r} \in \widehat{T}$ satisfies $\tilde{r} \widehat{M} \subset \widehat{I} \widehat{M}$. Now \cite[Proposition 2.4]{atiyah_macdonald} shows that $\tilde{r}^k \in \widehat{I}$ for some $k \geq 1$. Hence $r^k = 0$ and the regularity of $\widehat{R}$ implies that $r = 0$.
		
		By the Cohen structure theorem we can apply lemma \ref{faithful_cancellation_lemma}
		to the $\widehat{R}$-module $\widehat{M}/ \widehat{I} \widehat{M}$ and find that
		\[
		\mathfrak{a} \widehat{R} = \mathfrak{m}_{\widehat{R}}.
		\]
		Hence $(R_0)_{\mathfrak{p}} = \widehat{(R_0)_{\mathfrak{p}}} = \widehat{R} / \mathfrak{a} \widehat{R}$ is a field and the natural map 
		$(R_0)_{\mathfrak{p}} \to (T_0)_{\mathfrak{q}}$ must be injective. By construction it is compatible with the initial $R_0 \to T_0/I_0$, hence also surjective.
	\end{proof}
	
	\section{Locally Symmetric Spaces}
	
	In this section we define the spaces we need for the main argument. Their cohomology will carry the Hecke eigensystems corresponding to the Galois representations we wish to study. This is essentially a recollection of results from \cite{borel_serre} and \cite{newton_thorne_16}.
	
	\subsection{The Symmetric Space of a Linear Algebraic Group}
	The starting point for our definitions is a linear algebraic group $G$, defined over a number field $F$. Moreover, for the remainder of this section we assume that we have picked an algebraic maximal compact subgroup $K_\infty \subset G(F \otimes \reals)$, i.e. a maximal compact subgroup such that its Cartan involution is the restriction of an algebraic automorphism of the Zariski closure of $G(F \otimes \reals)$ in $G(F \otimes \complex)$.
	
	\begin{definition}
		Let $R_u G$ be the unipotent radical of $G$ and let $S_G$ be the identity component of the real points of the $\rationals$-split part of the centre of $\Res_{F/\rationals}(G/ R_u G)$.
	\end{definition}

	\begin{example}
		If $G = \GL_n/F$, where $F$ is any number field, then $S_G \cong \reals^\times_{>0}$ embedded as scalar matrices in $\GL_n(F \otimes \reals)$.
	\end{example}
	
	\begin{definition}
		Let $G$ be a linear algebraic group over a number field $F$. For a section $s : (G/R_u G)(F \otimes \reals) \to G(F \otimes \reals)$ we define the symmetric space (a smooth real manifold)
		\[
		X_{G,s} := G(F \otimes \reals) / K_\infty s(S_G) .
		\]
	\end{definition}
	
	\begin{example} Here are some examples of these symmetric spaces
		\begin{enumerate}
			\item If $G = \mathbb{G}_m$ over $F$, then $R_u G = 1$ and we can choose $s$ as the identity. Then we find a diffeomorphism
			\[
			X_{G,s} \to\mathbb{R}^{N_\infty}/ \reals (1, 1, \dots, 1): x \mapsto (\log | x |_{i} )_{i \in N_\infty}, 
			\]
			where $N_\infty$ denotes the set of equivalence classes of archimedean absolute values on $F$. Note that the regulator of $F$ is the volume of $\mathcal{O}_F^\times \backslash X_{G,s}$, where $X_{G,s}$ is given its euclidean metric.
			\item If $G = \GL_2/\rationals$, $s = \id$ and $K_\infty = O_2(\reals)$, then 
			\[
			X_{G,s} \to \mathbb{H} : \begin{bmatrix}
				a & b \\ c & d 
			\end{bmatrix} \mapsto  \frac{a \varepsilon i + b}{c \varepsilon i + d}
			\]
			is a diffeomorphism, where $\mathbb{H} = \{x + i y \in \complex : y > 0 \}$ and $\varepsilon = \operatorname{sign}(ad - bc)$.
			\item If $G = \GL_n / \rationals$, $K_\infty = O_n(\reals)$, $s = \id$ and $Y \subset M_n(\reals)$ is the set of positive definite, symmetric matrices of determinant 1, then 
			\[
			X_{G,s} \to Y : A \mapsto A^T A |\det A|^{-2/n}
			\]
			is a diffeomorphism.
			\item If $G = \begin{bmatrix}
				* & * & * \\
				0 & * & * \\
				0 & * & * 
			\end{bmatrix} \subset \GL_3$ over $F = \rationals$,
			$K_\infty = \{\pm 1\} \times O_2(\reals)$
			and $s : G/R_u G \cong \GL_1 \times \GL_2 \to \GL_3$ is the natural embedding, then
			\[
			X_{G,s} \to \mathbb{H} \times \reals^2 : \begin{bmatrix}
				1 & t & u \\
				0 & a & b \\
				0 & c & d
			\end{bmatrix} \mapsto \left(  \frac{a \varepsilon i + b}{c \varepsilon i + d}, \frac{t}{|ad - bc|^{1/2}}, \frac{u}{|ad - bc|^{1/2}} \right)
			\]
			is a diffeomorphism, where $\varepsilon = \operatorname{sign}(ad - bc)$.
		\end{enumerate}
	\end{example}
	
	\begin{prop}
		Let $G$ be a reductive group over $F$ and $P \subset G$ a parabolic subgroup. Then $P(F \otimes \reals)$ contains a unique Levi subgroup stable under the Cartan involution associated with $K_{\infty}$. Thus, there is a unique section $s : (P/R_u P)(F \otimes \reals) \to P(F \otimes \reals) \subset G(F \otimes \reals)$ whose image is stable under the Cartan involution.
	\end{prop}
	
	\begin{proof}
		See \cite[Proposition 1.6]{borel_serre}.
	\end{proof}
	
	\begin{definition}
		Let $G$ be a reductive group over $F$ and let $K_\infty$ be an algebraic maximal compact subgroup of $G(\reals \otimes F)$. For a parabolic subgroup $P \subset G$ the proposition provides us with a canonical choice of section $s$ and we will simply write
		\[
		X_P := X_{P,s}
		\]
		for this choice of $s$.
	\end{definition}
	
	\begin{definition}
		Let $P \subset G$ be a parabolic subgroup of a reductive group over $F$, then we define $A_P := s(S_P)/S_G$, where $s$ is the section stable under Cartan involution.
	\end{definition}
	
	\begin{lemma}
		If $P \subset G$ is a parabolic subgroup of a reductive group over $F$, then
		there is a natural diffeomorphism $X_P \cong X_G/A_P$.
	\end{lemma}
	
	\begin{proof}
		We have
		$G(F \otimes \reals) = P(F \otimes \reals) K_\infty$ by the Iwasawa decomposition.
		Hence naturally $X_G \cong P(F \otimes \reals)/ (K_\infty \cap P(F \otimes \reals)) S_G$. Now $K_\infty \cap P(F \otimes \reals)$ is a maximal compact subgroup of $P(F \otimes \reals)$, hence $X_G/A_P \cong X_P$. 
	\end{proof}
	
	\begin{example} Let $P = \begin{bmatrix}
			* & * \\ 0 & * 
		\end{bmatrix} \subset \GL_2/\rationals$. 
		Then $A_P = \left\{ \begin{bmatrix}
			t & 0 \\ 0 & 1  \end{bmatrix} : t > 0 \right\}$ and under the identification
		$X_{\GL_2} = \mathbb{H}$ from above we have
		\[
		(x + i y) \cdot \begin{bmatrix}
			t & 0 \\ 0 & 1 
		\end{bmatrix} = \begin{bmatrix}
			y & x \\ 0 & 1 
		\end{bmatrix} \begin{bmatrix}
			t & 0 \\ 0 & 1
		\end{bmatrix} = x + i ty,
		\]
		hence we can see how $X_{\GL_2} / A_P = \reals = X_P$.
	\end{example}
	
	\begin{definition}
		If $A$ is a finite product of copies of $\reals^\times_{>0}$, we define a manifold with corners $\overline{A}$ by replacing each $\reals^\times_{>0}$ factor with $\reals^\times_{>0} \cup \{\infty\}$.
		We let $A$ act on $\overline{A}$ in the obvious way fixing $\infty$.
	\end{definition}
	
	\begin{definition}
		Given topological spaces $X, Y$ and a topological group $G$ acting on the right on $X$ and on the left on $Y$, we define 
		$X \times_G Y := (X \times Y) / G$, where $g \in G$ acts diagonally. 
	\end{definition}	
	
	\begin{lemma}
		Let $P \subset G$ be a parabolic subgroup of a reductive group over $F$, then
		the topological space $Y_P := X_G \times_{A_P} \overline{A}_P$ is a disjoint union
		\[
		Y_P = \bigcup_{P \subset Q} X_Q
		\]
		over all parabolic subgroups of $G$ containing $P$ and the $X_Q$ are naturally identified with locally closed subsets of $Y_P$.
	\end{lemma}
	
	\begin{proof}
		See \cite[\S 5]{borel_serre}.
	\end{proof}

	\begin{example} Here are two examples of $Y_P$ for different $P$.
		\begin{enumerate}
			\item	Let $P = \begin{bmatrix}
				* & * \\ 0 & * 
			\end{bmatrix} \subset \GL_2/\rationals$. 
			Then $A_P = \left\{ \begin{bmatrix}
				t & 0 \\ 0 & 1  \end{bmatrix} : t > 0 \right\} \cong \reals_{>0}^\times$
			and so we have
			\[
			Y_P = X_{\GL_2} \times_{A_P} A_P \cup X_{\GL_2} \times_{A_P} \{\infty\} = X_{\GL_2} \cup X_{\GL_2}/A_P = X_{\GL_2} \cup X_P.
			\]
			From the previous example we know that the $A_P$ orbits in $X_{\GL_2}$ are of the form $\gamma_x(t) = x + i t$ and to each such orbit we add a limit at $t = \infty$ to obtain $Y_P$.
			
			\item Let $B = \begin{bmatrix}
				* & * & * \\
				0 & * & * \\
				0 & 0 & *
			\end{bmatrix} \subset \GL_3/\rationals$ and 
			$P = \begin{bmatrix}
				* & * & * \\
				* & * & * \\
				0 & 0 & *
			\end{bmatrix}, Q = \begin{bmatrix}
				* & * & * \\
				0 & * & * \\
				0 & * & *
			\end{bmatrix}$. Then 
			\[
			A_B \cong \reals_{>0}^\times \times \reals_{>0}^\times \cong \left\{ \begin{bmatrix}
				s t^{-2} & 0 & 0 \\
				0 & st & 0 \\
				0 & 0 & t s^{-2}
			\end{bmatrix} : s,t > 0 \right\}
			\]
			and $A_P = \{(s,1) \in A_B\}$ and $A_Q = \{(1,t) \in A_B\}$. Hence
			$\overline{A}_B \cong A_B \cup A_B/A_P \cup A_B/A_Q \cup A_B/A_B$ so that
			\[
			Y_B = X_G \cup X_P \cup X_Q \cup X_B.
			\]
			$X_B \subset Y_B$ is closed, $X_G \subset Y_B$ is open and dense, and $\partial X_P = \partial X_Q = X_B$. Combinatorially, $X_G$ is the interior of a polyhedron $Y_B$, $X_P$ and $X_Q$ are faces and $X_B$ is an edge. These closure relations can also be read off from the lattice of parabolic subgroups containing $B$:
			\[
			\begin{tikzcd}
				& P  \arrow{rd} & \\
				B \arrow{ru} \arrow{rd} &  & G \\
				& Q  \arrow{ru} &	
			\end{tikzcd}
			\]
		\end{enumerate}
	\end{example}
	
	\begin{theorem}[Borel--Serre] \label{borel_serre_thm}
		Let $G$ be a reductive group over $F$. Then the disjoint union over parabolic subgroups of $G$
		\[
		\overline{X}_G := \bigcup_{P \subset G} X_P
		\]
		carries the structure of a manifold with corners such that for each $P$, the set 
		\[
		\bigcup_{P \subset Q} X_Q
		\]
		is open and has the topology given by $X_G \times_{A_P} \overline{A}_P$.
		Moreover, the action of $G(F)$ on $X_G$ extends to $\overline{X}_G$.
	\end{theorem}
	
	\begin{proof}
		See \cite[\S 7.1 and Proposition 7.6]{borel_serre}.
	\end{proof}
	
	\subsection{Quotients by Congruence Subgroups}
	The symmetric spaces $X_G$ and their Borel--Serre compactifications $\overline{X}_G$ carry an action by $G(F)$. This action places us in the subject of number theory. 
	The Hecke actions that we study in this paper make use of the family of congruence subgroups of $G(F)$. In particular we are interested in the cohomology of quotients $\Gamma \backslash X_G$ and $\Gamma \backslash \overline{X}_G$ for congruence subgroups $\Gamma < G(F)$ and how it varies with $\Gamma$.\footnote{The first of these quotients is called an arithmetic locally symmetric space and the second is its Borel--Serre compactification.}  In our case, $\Gamma$ usually acts freely and $X_G$ is contractible so that the cohomology of $\Gamma \backslash X_G$ is isomorphic to the group cohomology of $\Gamma$.
	
	\subsubsection{Adelic Formalism} We introduce some notation to define congruence subgroups and the corresponding locally symmetric spaces in modern (adelic) terms. Let $F$ be a number field. Then we denote the adeles of $F$ by $\mathbb{A}_F$ and the finite adeles by $\mathbb{A}_F^\infty$. Let $G$ be a linear algebraic group over $F$ and define
	\[
	\mathfrak{X}_G := G(F) \backslash (X_G \times G(\mathbb{A}_F^\infty)) \qquad \overline{\mathfrak{X}}_G := G(F) \backslash (\overline{X}_G \times G(\mathbb{A}_F^\infty)),
	\]
	where $G(F)$ acts diagonally by left multiplication and $G(\mathbb{A}_F^\infty)$ is given the \emph{discrete} topology. Now $\mathfrak{X}_G$ is a topological space with a right action by the discrete group $G(\mathbb{A}_F^\infty)$. In particular for any open compact subgroup $K \subset G(\mathbb{A}_F^\infty)$, we can define
	\[
	X_G^K := \mathfrak{X}_G/K \qquad \overline{X}_G^K := \overline{\mathfrak{X}}_G/K.
	\]
	
	\begin{prop} \label{stratification_lemma}
		Let $G$ be a reductive group over $F$ and 
		let $\mathcal{P}$ be a set of representatives of $G(F)$ conjugacy classes of $F$-rational parabolic subgroups of $G$. 
		There is a $G(\mathbb{A}_F^\infty)$-equivariant stratification 
		\[
		\overline{\mathfrak{X}}_G \cong \bigcup_{P \in \mathcal{P}} \mathfrak{X}_P \times_{P(\mathbb{A}_F^\infty)} G(\mathbb{A}_F^\infty)
		\]
	\end{prop}
	
	\begin{proof}
		By \cite[Proposition 7.6]{borel_serre}, the natural map $X_P \to X_{gPg^{-1}}$ is an isomorphism for all parabolic subgroups $P$ and $g \in G(F)$.
		Since each parabolic group is its own stabilizer under the conjugation action of $G(F)$ we find that
		\[
		\overline{\mathfrak{X}}_G \cong G(F) \backslash \left(\bigcup_{P \subset G} X_P \times G(\mathbb{A}_F^\infty) \right) \cong \bigcup_{P \in \mathcal{P}} P (F) \backslash (X_P \times G(\mathbb{A}_F^\infty))
		\]
		Now there is a straightforward identification
		\[
		P (F) \backslash (X_P \times G(\mathbb{A}_F^\infty)) \cong \mathfrak{X}_P \times_{P(\mathbb{A}_F^\infty)} G(\mathbb{A}_F^\infty)
		\]
		and the claim follows.
	\end{proof}
	
	\begin{corollary} \label{adelic_torus_bundle_cor}
		Let $G$ be a reductive group over $F$, $P \subset G$ a proper parabolic subgroup and $M = P/R_u P$.
		There are $P(\mathbb{A}_F^\infty)$-equivariant maps
		\[
		\begin{tikzcd}
			\overline{\mathfrak{X}}_G \setminus \mathfrak{X}_G =: \partial \mathfrak{X}_G & \mathfrak{X}_P  \arrow[hookrightarrow]{l}{i_P} \arrow[twoheadrightarrow]{d}{f_P} \\
			& \mathfrak{X}_M 
		\end{tikzcd}
		\]
		where $P(\mathbb{A}_F^\infty)$ acts on $\mathfrak{X}_M$ via the quotient $M(\mathbb{A}_F^\infty)$ and $f_P : \mathfrak{X}_P \to \mathfrak{X}_M$ is the map induced by the morphism $P \to M$. Moreover, if $N = R_u P$, then $f_P$ is a $N(F) \backslash N(\mathbb{A}_F)$-bundle.
	\end{corollary}
	
	\begin{definition}
		An open compact subgroup $K < G(\mathbb{A}_F^\infty)$
		is said to be neat if all of its
		elements are neat. An element $g = (g_v)_v \in G(\mathbb{A}_F^\infty)$ is said to be neat if the
		intersection $\bigcap_v \Gamma_v$ is trivial, where $\Gamma_v \subset \overline{\rationals}^\times$ 
		is the torsion subgroup of the
		subgroup of $\overline{F}_v^\times$ generated by the eigenvalues of $g_v$ under the faithful representations of $G_{F_v}$.
	\end{definition}

	\begin{prop} \label{borel_serre_prop}
		Let $G/F$ be a reductive group.
		If $K$ is neat, then $X^K_G$ is a smooth manifold and $\overline{X}^K_G$ is a smooth compact manifold with corners. $X_G^K$ is the interior of $\overline{X}_G^K$ and the inclusion $X_G^K \hookrightarrow \overline{X}_G^K$ is a homotopy equivalence.
	\end{prop}
	
	\begin{proof}
		See \cite[Theorem 9.3]{borel_serre} and \cite[Lemma 8.3.1]{borel_serre} for the homotopy equivalence part.
	\end{proof}
	
	\begin{definition}
		An open compact subgroup $K < G(\mathbb{A}_{F}^\infty)$
		is called good if it is neat
		and of the form $K = \prod_{v} K_v$.
	\end{definition}
	
	\begin{lemma} \label{neat_lemma}
		Let $K = \prod_{v} K_v < \GL_n(\mathbb{A}_F^\infty)$ be an open compact subgroup and let $v$ be a place of $F$ such that $v$ does not divide $\prod_{j = 2}^{n [F:\rationals] + 1} \Phi_j(1)$, where $\Phi_j$ denotes the $j$th cyclotomic polynomial.  If $K_v$ is contained in the subgroup $Iw_{v,1} < \GL_n(\mathcal{O}_{F_v})$ of matrices which are unipotent and upper triangular mod $\varpi_v$, then $K$ is neat.
	\end{lemma}
	
	\begin{proof}
		If $\alpha$ is an eigenvalue of $g_v \in Iw_{v,1}$ under some faithful representation of $\GL_n/F_v$, then
		$\alpha - 1$ lies in the maximal ideal of $\mathcal{O}_{\overline{F}_v}$, hence $\Phi_j(\alpha)$ is a unit for all $ 1 < j \leq n [F: \rationals] + 1$ and so if $\alpha$ is a non-trivial root of unity, it must have order at least $n [F: \rationals] + 2$. Thus, if $K$ is not neat, then there exists a root of unity $\zeta \in \Gamma_v$ of prime order $q > n[F:\rationals] + 1$. But this would also imply that $q$ is the residue characteristic of $v$ since $\zeta - 1$ is in the maximal ideal of $\mathcal{O}_{\overline{F}_v}$. In this case $[\rationals_q(\zeta) : \rationals_q] = q - 1 > n [F: \rationals]$ contradicts the fact that $g_v$ satisfies its characteristic polynomial. 
	\end{proof}	
	
	\begin{definition}
		Let $K$ be a good open compact subgroup of $G(\mathbb{A}_F^\infty)$ and $S$ a finite set of places such that $K_v$ is hyperspecial for $v \not \in S$. If $V$ is a $\ints[K_S]$-module, where $K_S := \prod_{v \in S} K_v$, then we 
		let $\mathcal{V}$ denote the $K_S \times G(\mathbb{A}_F^S)$-equivariant constant sheaf on $\overline{\mathfrak{X}}_G$ associated with $V$. In other words, if we identify $V$ with an equivariant sheaf on a point, then its pull back to $\overline{\mathfrak{X}}_G$ under the constant map is denoted by $\mathcal{V}$.
		By abuse of notation we also let $\mathcal{V}$ be the same sheaf restricted to $\partial \mathfrak{X}_G$ or $\mathfrak{X}_G$.
	\end{definition}
	
	\begin{definition}
		Let $G$ be a group and $X$ a topological space with $G$-action. The category of $G$-equivariant sheaves of abelian groups on $X$ is denoted $\mathbf{Sh}^G(X)$.
		Let $Y$ be a topological space equipped with trivial $G$-action. If $f : X \to Y$ is a $G$-equivariant map and $\mathcal{F}$ is a $G$-equivariant sheaf on $X$, then we define a sheaf
		$f_*^G \mathcal{F}$ on $Y$ by the formula
		\[
		f_*^G \mathcal{F} (U) = \mathcal{F}(f^{-1}(U))^G
		\]
		for any open $U \subset Y$. 
	\end{definition}
	
	\begin{lemma} \label{good_quotient_exact}
		If $K < G(\mathbb{A}_F^\infty)$ is a good subgroup and $p^K : \mathfrak{X}_G \to X_G^K$ is the quotient map, then $p_*^K : \mathbf{Sh}^{K_S \times G(\mathbb{A}^S_F)}(\mathfrak{X}_G) \to \mathbf{Sh}(X_G^K)$ is exact and preserves injectives. 
	\end{lemma}
	
	\begin{proof}
		Since the forgetful functor $\mathbf{Sh}^{K_S \times G(\mathbb{A}_F^S)}(\mathfrak{X}_G) \to \mathbf{Sh}^{K}(\mathfrak{X}_G)$ is exact and preserves injectives, it suffices to show that 
		$p_*^K : \mathbf{Sh}^{K}(\mathfrak{X}_G) \to \mathbf{Sh}(X_G^K)$ is exact and preserves injectives. But this is even an equivalence of abelian categories by \cite[Lemma 2.17]{newton_thorne_16}
		and in particular exact. 
	\end{proof}
	
	\subsection{Hecke Algebras}
	
	The locally symmetric spaces we just introduced carry many cohomological correspondences that give their cohomology the structure of a module over a Hecke algebra. 
	The rings $T_N$ and the complexes $C_N$ appearing in our application of lemma \ref{patching_lemma} will be a special case. For an introduction to Hecke algebras of $p$-adic groups we refer to \cite{cartier_representations}.
	
	\begin{definition}
		Let $G$ be a locally profinite group and $K$ an open compact subgroup of $G$, then we define a ring with underlying set 
		\[
		\mathcal{H}(G,K) := \{f : G \to \ints \mid f \text{ compactly supported and } \forall k_1,k_2 \in K, \, f(k_1 g k_2) = f(g)\}.
		\] The addition is pointwise and the multiplication is given by
		\[
		(f_1 * f_2)(g) = \int_{G} f_1(x) f_2(x^{-1} g) d\mu(x),
		\]
		where $\mu$ is the left Haar measure on $G$ such that $\mu(K) = 1$. This ring has a $1$ given by the indicator function of $K \subset G$. For $g \in G$, we denote the
		the indicator function of the double coset $K g K$ by
		$[KgK] \in \mathcal{H}(G,K)$.
	\end{definition}
	
	\begin{lemma} \label{hecke_integral_action}
		For any $\ints[G]$-module $M$, the invariants $M^K$ form a module over the ring $\mathcal{H}(G,K)$ via the formula
		\[
		f * m = \int_{G} f(g) gm \, d\mu(g)
		\]
	\end{lemma}
	
	\begin{proof}
		See \cite[Lemma 2.3]{newton_thorne_16}.
	\end{proof}
	
	\begin{prop}
		Let $G$ be a reductive group over $F$ and $K$ a good open compact subgroup of $G(\mathbb{A}_F^\infty)$ and $S$ a finite set of places such that $K_v$ is hyperspecial for all $v \not \in S$. The Hecke algebra
		$\mathcal{H}(G(\mathbb{A}_F^S), K^S)$ acts by natural transformations on
		each of the functors 
		\begin{itemize}
			\item $\mathbf{D}(\mathbf{Sh}^{K_S \times G(\mathbb{A}_F^{S})}(\mathfrak{X}_G)) \to \mathbf{D}(\ints) : A \mapsto R\Gamma(X_G^K, p_*^K A)$
			\item $\mathbf{D}(\mathbf{Sh}^{K_S \times G(\mathbb{A}_F^{S})}(\partial \mathfrak{X}_G)) \to \mathbf{D}(\ints) : A \mapsto R\Gamma(\partial X_G^K, p_*^K A)$
		\end{itemize}
		such that for injective objects $I$, the action on $\Gamma(X_G^K, p_*^K I) = \Gamma(\mathfrak{X}_G, I)^K$ is the one given by lemma \ref{hecke_integral_action}.
	\end{prop}
	
	\begin{proof}
		This follows from the fact that $p_*^K$ is exact and preserves injectives, see lemma \ref{good_quotient_exact}. Namely, if we choose an injective resolution $A \to I^\bullet$, then
		\[
			R\Gamma(X_G^K, p_*^K A) \cong \Gamma(X_G^K, p_*^K I^\bullet) \cong  \Gamma(\mathfrak{X}_G, I^\bullet)^K
		\]
		carries a Hecke action by lemma \ref{hecke_integral_action}. This is well-defined since injective resolutions are unique up to homotopy.
	\end{proof}
	
	\begin{lemma} \label{hecke_flath_lemma}
		Let $G$ be a reductive group over $F$ and $K$ a good open compact subgroup of $G(\mathbb{A}_F^\infty)$ and $S$ a finite set of places such that $K_v$ is hyperspecial for all $v \not \in S$. There is an isomorphism
		\[
		\mathcal{H}(G(\mathbb{A}_F^S), K^S) \cong \bigotimes_{v \not \in S}' \mathcal{H}(G(F_v), G(\mathcal{O}_{F_v}))
		\]
		and each $\mathcal{H}(G(F_v), G(\mathcal{O}_{F_v}))$ is commutative.
	\end{lemma}
	
	\begin{proof}
		See \cite{flath_tensor} for the version with complex coefficients. The $\ints$-coefficient statement follows from this since $	\mathcal{H}(G(\mathbb{A}_F^S), K^S)$ is a free $\ints$-module.
	\end{proof}
	
	\begin{definition}
		Let $G$ be a reductive group over $F$ and $K$ a good open compact subgroup of $G(\mathbb{A}_F^\infty)$ and $S$ a finite set of places such that $K_v$ is hyperspecial for all $v \not \in S$. Suppose $P \subset G$ is a parabolic subgroup with levi $M \subset P$, then
		we use the previous lemma to define 
		 \begin{linenomath*} \begin{align*}
			r_P &: \mathcal{H}(G(\mathbb{A}_F^S), K^S) \to \mathcal{H}(P(\mathbb{A}_F^S), K \cap P(\mathbb{A}_F^S)) \\
			r_M &: \mathcal{H}(P(\mathbb{A}_F^S), K \cap P(\mathbb{A}_F^S)) \to 
			\mathcal{H}(M(\mathbb{A}_F^S), K \cap M(\mathbb{A}_F^S))
		\end{align*} \end{linenomath*} 
		with local components given by \cite[2.2.3 and 2.2.4]{newton_thorne_16} for each $v \not \in S$. Moreover, we set $\operatorname{Sat} = r_M \circ r_P$.
	\end{definition}
	
	\begin{definition}
		We abbreviate the following functors
		\begin{itemize}
			\item $E_G^K : \mathbf{D}(\mathcal{O}[K_S]) \to \mathbf{D}(\mathbf{Sh}^{K_S \times G(\mathbb{A}_F^{S})}(\mathfrak{X}_G)) \to \mathbf{D}(\ints) : V \mapsto R\Gamma(X_G^K, p_*^K \mathcal{V})$
			\item $\partial E_G^K : \mathbf{D}(\mathcal{O}[K_S]) \to \mathbf{D}(\mathbf{Sh}^{K_S \times G(\mathbb{A}_F^{S})}(\partial \mathfrak{X}_G)) \to \mathbf{D}(\ints) : V \mapsto R\Gamma(\partial X_G^K, p_*^K \mathcal{V})$
		\end{itemize}
	\end{definition}
	
	\begin{prop} \label{hecke_parabolic_satake_prop}
		Let $G$ be a reductive group over $F$ and $P \subset G$ a parabolic subgroup and let $M \subset P$ be a Levi subgroup. Suppose $K$ is a good open compact subgroup of $G(\mathbb{A}_F^\infty)$ such that $M(\mathbb{A}_F^\infty) \cap K \to (P(\mathbb{A}_F^\infty) \cap K)/(U(\mathbb{A}_F^\infty) \cap K)$ is an isomorphism, where $U = R_u P$ is the unipotent radical of $P$. Let $S$ be a finite set of places containing those above $p$, such that $K_v$ is hyperspecial for $v \not \in S$. If $V$ is a continuous $\ints_p[K_S]$-module which is finitely generated over $\ints_p$, then  there is a commutative diagram
		\[
		\begin{tikzcd}
			E_M^{K \cap M} R\Gamma((U \cap K)_S, V) \arrow{r}{\operatorname{Sat}(T)} & E_M^{K \cap M} R\Gamma((U \cap K)_S, V)  \\
			E_P^{K \cap P} (V) \arrow{r}{r_P(T)} \arrow{u}{\phi} & E_P^{K \cap P} (V) \arrow{u}{\phi} \\
			\partial E_G^K (V) \arrow{r}{T} \arrow{u}{i_P^*} & \partial E_G^K (V) \arrow{u}{i_P^*}
		\end{tikzcd}
		\] 
		where $T$ is any element of $\mathcal{H}(G(\mathbb{A}_F^S), K^S)$, $\phi$ is an isomorphism and $R\Gamma((U \cap K)_S, -)$ denotes the complex computing continuous $(U \cap K)_S$-group cohomology, seen as $\ints_p[M \cap K]$-modules.
	\end{prop}
	
	\begin{proof}
		By corollary \ref{adelic_torus_bundle_cor} we obtain a $(K \cap P)_S \times P(\mathbb{A}_F^S)$-equivariant diagram
		\[
		\begin{tikzcd}
			\partial \mathfrak{X}_G & \mathfrak{X}_P \arrow[hookrightarrow]{l}{i_P} \arrow[twoheadrightarrow]{d}{f_P} \\
			& \mathfrak{X}_M
		\end{tikzcd}
		\]
		where $f_P$ is a $U(F) \backslash (U(F \otimes \reals) \times U(\mathbb{A}_F^{\infty}))$-bundle. 
		Now we have a natural transformation of functors
		\[
		\Res^{K_S \times G(\mathbb{A}_F^S)}_{(K \cap P)_S \times P(\mathbb{A}_F^S)} \Gamma(\partial \mathfrak{X}_G, \mathcal{F}) \to \Gamma(\mathfrak{X}_P , i_P^* \mathcal{F}) 
		\]
		where $\mathcal{F}$ runs through the $K_S \times G(\mathbb{A}_F^S)$-equivariant sheaves on $\mathfrak{X}_G$. Hence \cite[Lemma 2.4]{newton_thorne_16} and taking derived functors gives us the lower square of the theorem. 
		
		Let $V = \lim V/p^m$ be a continuous $\ints_p[K_S]$-module, finite over $\ints_p$. Each $V/p^m$ is a discrete, smooth $\ints/p^m\ints[K]$-module
		and
		\begin{equation} \label{torus_bundle_sections_eq}
			R\Gamma(U \cap K, R \Gamma(\mathfrak{X}_P, \mathcal{V}/p^m)) = R \Gamma(\mathfrak{X}_M, R f_{P*}^{U \cap K} (\mathcal{V}/p^m) ).
		\end{equation}
		Now for any open subset 
		$W \subset \mathfrak{X}_M$ and smooth $\ints/p^m\ints[(P \cap K)_S]$-module $V/p^m$ we have 
		\[
		f_{P *} (\mathcal{V}/p^m) (W) = \{h : \pi_0( f_P^{-1}(W) ) \to V/p^m \}
		\]
		with $(U \cap K)$-action given by $(u \cdot h)(x) = u h(x u)$.
		If $w \in \mathfrak{X}_M$, then $f_P^{-1}(w) = U(F) \backslash (U(F \otimes \reals) \times U(\mathbb{A}_F^\infty))$ is homeomorphic to a disjoint union of copies of $U(F \otimes \reals)$, which is contractible as $U$ is unipotent. Hence $R^i f_{P *} = 0$ for $i > 0$. 
		By strong approximation, we have that for all $w \in W$,
		$f_P^{-1}(w)/(U \cap K) \cong (U(F) \cap K) \backslash U(F \otimes \reals)$ 
		is connected since $U$ is unipotent. Thus, we can also write
		\[
		f_{P *}^{U \cap K} \mathcal{V}/p^m (W) = \{ h : \pi_0(W) \to (V/p^m)^{U \cap K} \} = \{ h : \pi_0(W) \to (V/p^m)^{(U \cap K)_S} \}
		\]
		Hence $f_{P *}^{U \cap K} (\mathcal{V}/p^m)$ is the constant sheaf associated with $(V/p^m)^{(U \cap K)_S}$. 
		Consequently, we obtain a natural map 
		\[
		s^* R\Gamma((U\cap K)_S, V/p^m) \to R f_{P*}^{U \cap K} (\mathcal{V}/p^m)
		\]
		where $s$ is the constant map. We wish to show that this is a quasi-isomorphism. 
		We may check it on stalks, so let $x \in \mathfrak{X}_M$ and $W_{\alpha}$ a basis
		of contractible open neighbourhoods of $x$. If $\mathcal I$ is an injective $(U \cap K)$-equivariant sheaf on $\mathfrak{X}_M$, then we claim that
		$\mathcal{I}(W_{\alpha})$ is an injective $(U \cap K)$-module. To see this, note that
		\[
		\Hom^{U \cap K} ( A, \mathcal{I}(W_{\alpha})) = \Hom_{\mathbf{Sh}^{U \cap K}(W_{\alpha})}\left( \mathcal{A}, \mathcal{I} \rvert_{W_{\alpha}} \right) = \Hom_{\mathbf{Sh}^{U \cap K}(\mathfrak{X}_M)}(i_! \mathcal{A}, \mathcal{I}),
		\]
		where $i : W_{\alpha} \to \mathfrak{X}_M$ is the inclusion and $\mathcal{A}$ is the constant $U \cap K$-equivariant sheaf on $W_{\alpha}$ associated with $A$. But $i_!$ is exact and $\mathcal{I}$ is injective, so we have written $\Hom^{U \cap K} ( -, \mathcal{I}(W_{\alpha}))$ as the composition of two exact functors. Now choose an injective resolution \[
		f_{P *} (\mathcal{V}/p^m) \to \mathcal{I}^\bullet,\]
		then $\mathcal{I}^\bullet(W_{\alpha})$ is exact by contractibility of $W_{\alpha}$ and thus it is an injective resolution of $\mathcal{V}/p^m(f_P^{-1}(W_\alpha))$ and
		\[
		R\Gamma(W_\alpha, R f_{P*}^{U \cap K} (\mathcal{V}/p^m)) = \mathcal{I}^\bullet(W_{\alpha})^{U \cap K}.
		\]
		Since $W_{\alpha}$ is contractible and the $R^i f_{P *}^{U \cap K} (\mathcal{V}/p^m)$ are locally constant we see that the cohomology of the left hand side is independent of $\alpha$.
		Taking the colimit over all $\alpha$, which is exact and commutes with $U \cap K$-invariants, we find the stalk 
		\[
		R^i f_{P*}^{U \cap K} (\mathcal{V}/p^m)_x = \lim_{\longrightarrow}  H^i\left(\mathcal{I}^\bullet(W_{\alpha})^{U \cap K} \right) = H^i\left(\mathcal{I}^\bullet(W_{\alpha_0})^{U \cap K} \right)
		\]
		for any fixed $\alpha_0$. But $\mathcal{V}/p^m(f_P^{-1}(W_{\alpha_0})) = (f_{P *} \mathcal{V}/p^m)_x$ and $\mathcal{I}^\bullet(W_{\alpha_0})$ is an injective resolution of $\mathcal{V}/p^m(f_P^{-1}(W_{\alpha_0}))$  and so
		\[
		R^i f_{P*}^{U \cap K} (\mathcal{V}/p^m)_x = H^i_{disc}(U \cap K, (f_{P *} \mathcal{V}/p^m)_x),
		\]
		where $H^i_{disc}$ denotes \emph{discrete} group cohomology. The description of the fibres implies 
		$f_{P*}(\mathcal{V}/p^m)_x \cong \Ind_{U(F) \cap K}^{U(\mathbb{A}_F^\infty) \cap K} (V/p^m)$
		and by Shapiro's lemma  $R^i f_{P*}^{U \cap K} (\mathcal{V}/p^m)_x = H^i_{disc}(U(F) \cap K, V/p^m)$.
		Now to verify the quasi-isomorphism in question it suffices to check that the natural map
		\[
		H^i_{cont}( (U \cap K)_S, V/p^m) \to H^i_{disc}( U(F) \cap K, V/p^m)
		\]
		is an isomorphism. To see this, first observe that $H^i_{cont}( (U \cap K)_S, V/p^m) \to H^i_{cont}( U \cap K, V/p^m)$ is an isomorphism since $V/p^m$ is a finite abelian $p$-group and $(U \cap K)_S$ contains the pro-$p$ part of $U \cap K$. 
		
		Further, we can find a finite index subgroup of $U(F) \cap K$ which admits a filtration with free abelian quotients. 
		Now using \cite[I. \S 2.6 Exercise 2]{serre_galois_cohomology} we see that 
		$U(F) \cap K$ is good because $\ints^r$ is good for any $r$, where a group $G$ is called good if 
		$H^i_{cont}(\hat{G}, M) \to H^i_{disc}(G, M)$ is an isomorphism for all $i$, whenever $M$ is a finite abelian group with continuous $\hat{G}$-action.
		It follows that $H^i_{cont}( U \cap K, V/p^m) \to H^i_{disc}( U(F) \cap K, V/p^m)$ is an isomorphism, as required.
		
		By taking derived $K \cap P$-invariants in \eqref{torus_bundle_sections_eq} and passing to the limit, we get the desired
		isomorphism between $E_P^{K \cap P}(V)$ and $E_M^{K\cap M} ( R \Gamma((U \cap K)_S, V))$. Moreover, \cite[Lemma 2.7]{newton_thorne_16} implies that the corresponding isomorphism of functors
		\[
		\phi : E_M^{K\cap M} ( R \Gamma((U \cap K)_S, V)) \to E_P^{K \cap P}(V)
		\]
		satisfies $\operatorname{Sat}(T) \circ \phi = \phi \circ r_P(T)$ for all $T \in \mathcal{H}(G(\mathbb{A}_F^S), K^S)$. 
	\end{proof}
	
	\begin{remark}
		The previous proof is very similar to the arguments used in the proof of \cite[Theorem 4.2]{newton_thorne_16}.
	\end{remark}
	
	\subsection{The Quasi-Split Unitary Group}
	
	Let $F/F^+$ be a CM field with complex conjugation $c \in \Gal(F/F^+)$. Let $n \geq 2$ be an integer and 
	\[
	J_n = \begin{bmatrix}
		0 & \Psi_n \\
		- \Psi_n & 0
	\end{bmatrix},
	\]
	where $\Psi_n$ has $1$'s on the anti-diagonal and $0$'s elsewhere. We consider the hermitian pairing 
	\[
	\mathcal{O}_F^{2n} \times \mathcal{O}_F^{2n} \to \mathcal{O}_F : (x,y) \mapsto x^{T} J_n y^c
	\]
	and use it to define the $\mathcal{O}_{F^+}$-group scheme ${\widetilde{G}}$ with functor of points
	\[
	{\widetilde{G}}(R) = \{ g \in \GL_{2n}(\mathcal{O}_F \otimes_{\mathcal{O}_{F^+}} R) : g^{T} J_n g^{c} = J_n\}.
	\]
	We define the Siegel parabolic $P \subset \widetilde{G}$ as the closed subgroup scheme preserving the subspace $\mathcal{O}_F^n \oplus 0^n \subset \mathcal{O}_F^{2n}$, and 
	$G \subset P$ as the subgroup scheme which preserves both 
	$\mathcal{O}_F^n \oplus 0^n$ and $0^n \oplus \mathcal{O}_F^n$.
	
	\begin{lemma}
		There is an isomorphism $G \cong \Res^F_{F^+} \GL_n$.
	\end{lemma}
	
	\begin{proof}
		See \cite[Lemma 5.1]{newton_thorne_16}.
	\end{proof}
	
	Now for a good open compact subgroup $\widetilde{K} \subset \widetilde{G}(\mathbb{A}_{F^+}^\infty)$ satisfying the conditions of proposition \ref{hecke_parabolic_satake_prop} with the Siegel parabolic $P$, the following diagram summarises the relation between the locally symmetric space attached to $\widetilde{G}$ and the one attached to $G$.

	\begin{equation} \label{locally_symmetric_spaces_diagram}
		\begin{tikzcd}
			X_P^{\widetilde K \cap P} \arrow{d} \arrow[hookrightarrow]{r} & \partial X_{\widetilde G}^{\widetilde K} \arrow[hookrightarrow]{r} & \overline{X}_{\widetilde G}^{\widetilde K} & X_{\widetilde G}^{\widetilde K} \arrow[hookrightarrow]{l} \\
			X_G^{\widetilde{K} \cap G} & & &
		\end{tikzcd}
	\end{equation}
	
	Here is a sketch of the roles that these spaces will play later. We started with an automorphic representation of $\GL_n/F$, which will give rise to some cohomology classes on $X_G^K$. Hence we will need to associate Galois representations to Hecke modules appearing in the singular cohomology of $X_G^K$. The cohomology of $X_{\widetilde G}^{\widetilde K}$ is known to have good properties by \cite{caraiani_scholze_non_compact}. More precisely, in degree $d = 1 + \dim X_G$, a certain direct summand $H^d(X_{\widetilde G}^{\widetilde K}, \ints_p)_{\widetilde{\mathfrak{m}}} \subset H^d(X_{\widetilde G}^{\widetilde K}, \ints_p)$ is $p$-torsion free, hence can be computed in terms of
	automorphic representations on $\widetilde{G}$ as in \cite{franke_schwermer}. Now we use \cite[Appendix]{goldring_galois_reps} to attach Galois representations of $G_F$ to these automorphic representations.
	
	Thus, we would like to pass from $H^q(X_G^K)$ to $H^d(X_{\widetilde G}^{\widetilde K})_{\widetilde{\mathfrak{m}}}$ by walking along the arrows of diagram \eqref{locally_symmetric_spaces_diagram}.
	The vertical arrow is addressed in proposition \ref{hecke_parabolic_satake_prop}. 
	The comparison between $H^\bullet(X_P^{\widetilde{K} \cap P})$ and $H^\bullet(\partial X_{\widetilde{G}}^{\widetilde{K}})$ is done in proposition \ref{non_eisenstein_boundary_prop}. Finally to reach $X_{\widetilde{G}}^{\widetilde{K}}$ from $\partial X_{\widetilde{G}}^{\widetilde{K}}$ we use the vanishing theorem of \cite{caraiani_scholze_non_compact}.
	
	Then it remains to show why it is enough to study only the degree $d$ cohomology on $X_{\widetilde{G}}^{\widetilde{K}}$. This is done by the degree-shifting argument from \cite{10author} which we generalise in the next section. Before we can execute this strategy we need to set up some more definitions.
	
	\begin{definition}
		If $S$ is a finite set of places of $F$, then we define the ring
		\[
		\mathbb{T}^{S} := \mathcal{H}(\GL_n(\mathbb{A}_F^S), \prod_{v \not \in S} \GL_n(\mathcal{O}_{F_v}))
		\] 
		which is commutative by lemma \ref{hecke_flath_lemma}.
		For a ring $R$ we also set $\mathbb{T}^S_R := \mathbb{T}^S \otimes_\ints R$. Similarly we
		let $\overline{S}$ denote the set of places of $F^+$ lying below a place of $S$ and define $\widetilde{\mathbb{T}}^{\overline{S}}$ as $\mathcal{H}(\widetilde{G}(\mathbb{A}_{F^+}^{\overline{S}}), \prod_{\overline{v} \not \in \overline{S}} \widetilde{G}(\mathcal{O}_{F_v}))$.
	\end{definition}
	
	\begin{definition}
		Keep the conditions from the previous definition. For $v \not \in S$ we let $P_v(z) \in \mathbb{T}^S[z]$ be the polynomial with the same name defined in \cite[2.2.5]{10author}. Similarly for $\overline{v} \not \in \overline{S}$, we let $\widetilde{P}_v(z) \in \widetilde{\mathbb{T}}^{\overline{S}} [q_{\overline{v}}^{-1}, z]$ be the other polynomial defined in \cite[2.2.6]{10author}. 
	\end{definition}

	Let $E$ be a finite extension of $\rationals_p$ which contains all the $p$-adic embeddings of $F$ and let $\mathcal O$ be the ring of integers of $E$ with uniformiser $\varpi$. Let $T \subset G$ be the standard torus of $\Res^{F}_{F^+} \GL_n$. Note that $T$ splits over $E$, $T(F^+ \otimes_{\rationals_p} E) = \prod_{\widetilde{\tau} : F \to E} (E^\times)^n$, $G(F^+ \otimes_{\rationals_p} E) = \prod_{\widetilde{\tau} : F \to E} \GL_n(E)$ and $\widetilde{G}(F^+ \otimes_{\rationals_p} E) = \prod_{\tau : F^+ \to E} \GL_{2n}(E)$. 
	
	And with these identifications, the inclusion $G \to \widetilde{G}$ is given by
	\[
	\prod_{\tau : F^+ \to E} \GL_n(E) \times \GL_n(E) \to \prod_{\tau : F^+ \to E} \GL_{2n}(E) : (g_{\widetilde{\tau}})_{\widetilde \tau} \mapsto \left( \begin{bmatrix}
		g_{\widetilde{\tau}} & 0 \\ 0 & \Psi_n g_{\widetilde{\tau}c}^{-T} \Psi_n
	\end{bmatrix}\right)_{\tau}
	\]
	Thus, $(\lambda_{\widetilde{\tau},1}, \dots, \lambda_{\widetilde{\tau},n}) \in X^\bullet(T(E \otimes_{\rationals_p} F^+)) \cong (\ints^n)^{\Hom(F, E)}$ is dominant for $G$ if and only if
	\[
	\lambda_{\widetilde{\tau}, 1} \geq \dots \geq \lambda_{\widetilde{\tau},n}
	\]
	for all $\widetilde{\tau} : F \to E$. It is dominant for $\widetilde{G}$ if and only if
	\[
	- \lambda_{\widetilde{\tau} c, n} \geq \dots \geq - \lambda_{\widetilde{\tau} c, 1} \geq \lambda_{\widetilde{\tau}, 1} \geq \dots \geq \lambda_{\widetilde{\tau}, n}
	\]
	
	\begin{definition} \label{highest_weight_module_def}
		If $\widetilde{\lambda} \in X^\bullet(T(E \otimes_{\rationals_p} F^+))$ is dominant for $\widetilde{G}$, then we can also see $\widetilde{\lambda}$ as a character $(\mathbb{G}_{m, \mathcal{O}}^n)^{\Hom(F,E)} \to \mathbb{G}_{m,\mathcal{O}}$ and we can consider the $\mathcal{O}$-points of the algebraic induction 
		\[ 
			V_{\widetilde{\lambda}} := \left( \Ind_{B_{2n}^{\Hom(F^+,E)}}^{\GL_{2n}^{\Hom(F^+,E)}} \widetilde{\lambda} \right) (\mathcal{O}),
		\] 
		where $B_{2n} \subset \GL_{2n}$ is the standard Borel over $\mathcal{O}$. We view it as a continuous module over $\mathcal{O}[\prod_{\overline{v} \mid p} \widetilde{G}(\mathcal{O}_{F^+, \overline{v}})]$ via the maps $\widetilde{G}(\mathcal{O}_{F^+, \overline{v}} ) \to \widetilde{G}(\mathcal{O}) \cong \GL_{2n}(\mathcal{O})$ induced by the field embeddings $F^+ \to E$.
		 Moreover,
		$V_{\widetilde{\lambda}} = \bigotimes_{\overline{v} \mid p} V_{\widetilde{\lambda_{\overline{v}}}}$, where
		$\widetilde{\lambda}_{\overline{v}} \in (\ints^n)^{\overline{S}_{\overline{v}}}$ is the projection of $\widetilde{\lambda}$ and $\overline{S}_{\overline{v}}$ is the set of embeddings $F^+ \to E$ inducing $\overline{v}$. The heighest weight $\mathcal{O}[\prod_{v \mid p} \GL_n(\mathcal{O}_{F, v})]$-modules $V_\lambda$ for weights $\lambda$ which are dominant for $G$ are defined similarly.
	\end{definition}
	
	By \cite[I.5.12]{jantzen} $V_{\widetilde{\lambda}}$ can be identified with the global sections of a line bundle on a projective $\mathcal{O}$-scheme. Thus, $V_{\widetilde{\lambda}}$ is a finite free $\mathcal{O}$-module such that $V_{\widetilde{\lambda}}[1/p]$ is the irreducible representation of $\widetilde{G}( F \otimes_{\rationals_p} E)$ of highest weight $\widetilde{\lambda}$.
	
	\begin{definition}
		Let $S$ be a set of places of $F$. We say that $S$ satisfies $(\Sigma_p)$ if
		\begin{itemize}
			\item $S$ is finite and contains the archimedean places and those above $p$.
			\item $S$ is stable under complex conjugation.
			\item Let $v$ be a finite place of $F$ not contained in $S$, and let $\ell$ be its residue
			characteristic. Then either $S$ contains no $\ell$-adic places of $F$ and $\ell$ is
			unramified in $F$, or there exists an imaginary quadratic field $F_0 \subset F$ in which $\ell$ splits.
		\end{itemize}
	\end{definition}
	
	\begin{theorem}
		Let $S$ be a set of places of $F$ satisfying $(\Sigma_p)$ and let $\mathfrak{m}$ be a maximal ideal of $\mathbb{T}^S$ occuring in the support of the module $H^\bullet(X_{\GL_n/F}^K, \mathcal{V}_\lambda)$, for some good open compact subgroup $K$ such that $K_v = \GL_n(\mathcal{O}_{F_v})$ for $v \not \in S$.
		Then there exists a continuous semisimple representation $\overline{\rho}_{\mathfrak{m}} : G_{F,S} \to \GL_n(\mathbb{T}^S/\mathfrak{m})$ such that
		\[
		\det(z \cdot \id - \overline{\rho}_\mathfrak{m}(\Frob_v) ) = P_v(z) \in (\mathbb{T}^S/\mathfrak{m})[z] \qquad \forall \, v \not\in S
		\]
		Note that this uniquely characterises $\overline{\rho}_{\mathfrak{m}}$ by semisimplicity.
	\end{theorem}
	
	\begin{proof}
		See \cite[Theorem 2.3.5]{10author}.
	\end{proof}
	
	\begin{definition}
		A maximal ideal of $\mathbb{T}^S$ occuring in the support of $H^\bullet(X_{\GL_n/F}^K, \mathcal{V}_\lambda)$, for some good open compact subgroup $K$ such that $K_v = \GL_n(\mathcal{O}_F)$ for $v \not \in S$, is called non-Eisenstein if $\overline{\rho}_{\mathfrak{m}}$ is absolutely irreducible.
	\end{definition}
	
	\begin{theorem} \label{hecke_product_of_fields}
		Let $K$ be a good open compact subgroup of $\GL_n(\mathbb{A}_F^\infty)$ and denote by $\mathbb{T}^S(R\Gamma(X^K_{\GL_n/F},\mathcal{V}_\lambda))$ the image of $\mathbb{T}^S$ in the endomorphisms of $R\Gamma(X^K_{\GL_n/F},\mathcal{V}_\lambda)$, where $S$ is a finite set of places containing those above $p$ such that $K_v = \GL_n(\mathcal{O}_{F_v})$ for all $v \not \in S$. If $\mathfrak{m} < \mathbb{T}^S$ is a non-Eisenstein maximal ideal, then
		$\mathbb{T}^S(R\Gamma(X^K_{\GL_n/F},\mathcal{V}_\lambda))_{\mathfrak{m}}[1/p]$ is either zero or a product of fields.	
	\end{theorem}
	
	\begin{proof}
		This follows from the decomposition in \cite[\S 2.2]{franke_schwermer} and the fact that only cuspidal representations contribute to the localisation at $\mathfrak{m}$. 
	\end{proof}
	
	\begin{theorem}
		Let $S$ be a set of places of $F$ satisfying $(\Sigma_p)$ and let $\widetilde{\mathfrak{m}}$ be a maximal ideal of $\widetilde{\mathbb{T}}^S$ occuring in the support of the module $H^\bullet(X_{\widetilde{G}}^{\widetilde{K}}, \mathcal{V}_{\widetilde{\lambda}})$, for some good open compact subgroup $\widetilde{K}$ such that $\widetilde{K}_v = \widetilde{G}(\mathcal{O}_{F_v^+})$ for $v \not \in \overline{S}$.
		Then there exists a continuous semisimple representation $\overline{\rho}_{\widetilde{\mathfrak{m}}} : G_{F,S} \to \GL_{2n}(\widetilde{\mathbb{T}}^S/\widetilde{\mathfrak{m}})$ such that
		\[
		\det(z \cdot \id - \overline{\rho}_{\widetilde{\mathfrak{m}}}(\Frob_v) ) = \widetilde{P}_v(z) \in (\widetilde{\mathbb{T}}^S/\widetilde{\mathfrak{m}})[z] \qquad \forall \, v \not\in S
		\]
		Note that this uniquely characterises $\overline{\rho}_{\widetilde{\mathfrak{m}}}$ by semisimplicity.
	\end{theorem}
	
	\begin{proof}
		See \cite[Theorem 2.3.8]{10author}.
	\end{proof}
	
	\begin{prop} \label{non_eisenstein_boundary_prop}
		Let $S$ be a set of places of $F$ satisfying $(\Sigma_p)$ and $\widetilde{K} < \widetilde{G}(\mathbb{A}_{F^+}^\infty)$ a good open compact subgroup such that $\widetilde{K}_{\overline{v}} = \widetilde{G}(\mathcal{O}_{F^+_{\overline{v}}})$ for $\overline{v} \not \in \overline{S}$ and $G(\mathbb{A}_F^\infty) \cap \widetilde{K} \to (P(\mathbb{A}_{F^+}^\infty) \cap \widetilde{K})/(U(\mathbb{A}_{F^+}^\infty) \cap \widetilde{K})$ is an isomorphism.
		If $\widetilde{\lambda}$ is a dominant weight for $\widetilde{G}$, $\lambda$ is the corresponding dominant weight for $G$, $\mathfrak{m}$ is a non-Eisenstein ideal of $\mathbb{T}^S(R\Gamma(X_G^K, p_*^{K} \mathcal{V}_\lambda))$ and $\widetilde{\mathfrak{m}} = \operatorname{Sat}^{-1}(\mathfrak{m}) \subset \widetilde{\mathbb{T}}^{\overline{S}}$, then the map from proposition \ref{hecke_parabolic_satake_prop} makes $E_{P}^{P \cap \widetilde{K}}(V_{\widetilde{\lambda}})_{\widetilde{\mathfrak{m}}}$ a $\widetilde{\mathbb{T}}^{\overline{S}}$-equivariant direct summand of $ \partial E_{\widetilde{G}}^{\widetilde{K}}(V_{\widetilde{\lambda}})_{\widetilde{\mathfrak{m}}}$, 
		where $P$ is the Siegel parabolic.
		Moreover, the natural map 
		\[
		R\Gamma_c(X_P^{\widetilde{K} \cap P}, p_*^{\widetilde{K} \cap P} \mathcal{V}_{\widetilde{\lambda}})_{\widetilde{\mathfrak{m}}} \to R\Gamma(X_P^{\widetilde{K} \cap P}, p_*^{\widetilde{K} \cap P} \mathcal{V}_{\widetilde{\lambda}})_{\widetilde{\mathfrak{m}}} 
		\]
		is an isomorphism, where $\widetilde{\mathbb {T}}^{\overline{S}}$ acts via the map $r_P$.
	\end{prop}
	
	\begin{proof}
		Let us begin with the second part of the claim. 
		Consider 
		\[
		\overline{X}_P^{\widetilde{K} \cap P} := \left( \bigcup_{Q \subset P} \mathfrak{X}_Q \times_{Q(\mathbb{A}_F^\infty)} P(\mathbb{A}_F^\infty) \right) / (\widetilde{K} \cap P)
		\]
		as a locally closed subset of $\overline{X}_{\widetilde{G}}^{\widetilde{K}}$. Then the open embedding
		\[
		X_P^{\widetilde{K} \cap P} \hookrightarrow \overline{X}_P^{\widetilde{K} \cap P}
		\]
		is a homotopy equivalence (see \cite[Lemma 8.3.1]{borel_serre})
		and the corresponding excision distinguished triangle is
		\[
		R\Gamma_c(X_P^{\widetilde{K} \cap P}, p_*^{\widetilde{K} \cap P} \mathcal{V}_{\widetilde{\lambda}})
		\rightarrow
		R\Gamma(X_{P}^{\widetilde{K} \cap P}, p_*^{\widetilde{K} \cap P} \mathcal{V}_{\widetilde{\lambda}}) \to 
		R \Gamma (\overline{X}_{P}^{\widetilde{K} \cap P} \setminus X_{P}^{\widetilde{K} \cap P}, p_*^{\widetilde{K} \cap P} \mathcal{V}_{\widetilde{\lambda}}) \xrightarrow{[1]}
		\]
		Thus, for the second part it suffices to show that $R \Gamma (\overline{X}_{P}^{\widetilde{K} \cap P} \setminus X_{P}^{\widetilde{K} \cap P}, p_*^{\widetilde{K}} \mathcal{V}_{\widetilde{\lambda}})_{\widetilde{\mathfrak{m}}} = 0$.
		For the first part, we note that the Siegel parabolic $P$ is a maximal parabolic subgroup, hence
		theorem \ref{borel_serre_thm} and proposition \ref{stratification_lemma} show that 
		\[
		\mathfrak{X}_P \times_{P(\mathbb{A}_{F^+}^\infty)} \widetilde{G}(\mathbb{A}_{F^+}^\infty) \subset \partial \mathfrak{X}_{\widetilde{G}}
		\]
		is an open subset. For a parabolic subgroup $P \subset \widetilde{G}$ we define 
		\[
		\widetilde{X}_P^{\widetilde{K}} := (\mathfrak{X}_P \times_{P(\mathbb{A}_{F^+}^\infty)} \widetilde{G}(\mathbb{A}_{F^+}^\infty))/\widetilde{K} \subset \partial X_{\widetilde{G}}^{\widetilde{K}}.
		\]
		By the Iwasawa decomposition we have $P(F_v^+) \widetilde{K}_v = \widetilde{G}(F_v^+)$ for all $v \not \in \overline{S}$, hence get a 
		$\widetilde{G}(\mathbb{A}_{F^+}^{\overline{S}}) \times \widetilde{K}_{\overline{S}}$-stable disjoint union
		\[
		\mathfrak{X}_P \times_{P(\mathbb{A}_{F^+}^\infty)} \widetilde{G}(\mathbb{A}_{F^+}^\infty) = \bigcup_{\alpha} \mathfrak{X}_P \times_{P(\mathbb{A}_{F^+}^\infty)} P(\mathbb{A}_{F^+}^\infty) g_{\alpha} \widetilde{K}
		\]
		for some set of representatives $g_{\alpha} \in \prod_{v \in \overline{S}} \widetilde{G}(F^+_{v})$. After quotienting by $\widetilde{K}$, it follows that 
		there is a $\widetilde{\mathbb{T}}^{\overline{S}}$-equivariant decomposition
		\[
		R\Gamma(\widetilde{X}_P^{\widetilde{K}}, p_*^{\widetilde{K}}\mathcal{V}_{\widetilde{\lambda}})  = \bigoplus_{\alpha} R\Gamma(X_P^{P \cap g_{\alpha} \widetilde{K} g_{\alpha}^{-1}}, p_*^{P \cap g_{\alpha} \widetilde{K} g_{\alpha}^{-1}}\mathcal{V}_{\widetilde{\lambda}}).
		\]
		We have an excision distinguished triangle
		\[
		R\Gamma_c(\widetilde{X}_P^{\widetilde{K}}, p_*^{\widetilde{K}} \mathcal{V}_{\widetilde{\lambda}})
		\xrightarrow{i_{P!}}
		R\Gamma(\partial X_{\widetilde{G}}^{\widetilde{K}}, p_*^{\widetilde{K}} \mathcal{V}_{\widetilde{\lambda}}) \to 
		R \Gamma (\partial X_{\widetilde{G}}^{\widetilde{K}} \setminus \widetilde{X}_{P}^{\widetilde{K}}, p_*^{\widetilde{K}} \mathcal{V}_{\widetilde{\lambda}}) \xrightarrow{[1]}
		\]
		and we will show that $i_{P!}$ localised at $\widetilde{\mathfrak{m}}$ is an isomorphism. Let $\widetilde{i}_P : \widetilde{X}_P^{\widetilde{K}} \hookrightarrow \partial X_{\widetilde{G}}^{\widetilde{K}}$, be the inclusion. Then the localization of $\widetilde{i}_P^*$ is also an isomorphism, because the composition of $\widetilde{i}_{P}^*$ after $i_{P!}$ is the natural map 
		\[
		R\Gamma_c(\widetilde{X}_P^{\widetilde{K}}, p_*^{\widetilde{K}} \mathcal{V}_{\widetilde{\lambda}}) \to R\Gamma(\widetilde{X}_P^{\widetilde{K}}, p_*^{\widetilde{K}} \mathcal{V}_{\widetilde{\lambda}}) 
		\] 
		and we will show that the localisation of this map is an isomorphism. The first part of the proposition follows from this because we already observed that 
		$R\Gamma(X_P^{P \cap \widetilde{K}}, p_*^{P \cap \widetilde{K}}\mathcal{V}_{\widetilde{\lambda}})$
		is a $\widetilde{\mathbb{T}}^{\overline{S}}$-equivariant direct summand of
		$R\Gamma(\widetilde{X}_P^{\widetilde{K}}, p_*^{\widetilde{K}}\mathcal{V}_{\widetilde{\lambda}})$.
		In conclusion, it suffices to show that for any choice of $\widetilde{K}$ we have
		\[
		R \Gamma (\partial X_{\widetilde{G}}^{\widetilde{K}} \setminus \widetilde{X}_{P}^{\widetilde{K}}, p_*^{\widetilde{K}} \mathcal{V}_{\widetilde{\lambda}})_{\widetilde{\mathfrak{m}}} = 0 = R \Gamma (\overline{X}_{P}^{\widetilde{K} \cap P} \setminus X_{P}^{\widetilde{K} \cap P}, p_*^{\widetilde{K}} \mathcal{V}_{\widetilde{\lambda}})_{\widetilde{\mathfrak{m}}},
		\]
		since then $i_{P!}$ is an isomorphism by the triangle above.
		Using proposition \ref{stratification_lemma} we can do more excision to reduce this to showing that the compactly supported complexes $R \Gamma_c (X_{Q}^{\widetilde{K} \cap Q},  p_*^{\widetilde{K} \cap Q} \mathcal{V}_{\widetilde{\lambda}})_{\widetilde{\mathfrak{m}}}$ vanish for all proper parabolic subgroups $Q \neq P$ containing the standard Borel $B$ and that
		$R \Gamma (X_{B}^{\widetilde{K} \cap B}, p_*^{B \cap \widetilde{K}} \mathcal{V}_{\widetilde{\lambda}})_{\widetilde{\mathfrak{m}}}$ vanishes. Using the long exact sequence in cohomology attached to
		\[
			0 \to \mathcal{V}_{\widetilde{\lambda}} \to \mathcal{V}_{\widetilde{\lambda}} \to \mathcal{V}_{\widetilde{\lambda}}/ \varpi \mathcal{V}_{\widetilde{\lambda}} \to 0,
		\]
		we see that it suffices to show the vanishing of 
		$R \Gamma_c (X_{Q}^{\widetilde{K} \cap Q},  p_*^{\widetilde{K} \cap Q} (\mathcal{V}_{\widetilde{\lambda}}/ \varpi \mathcal{V}_{\widetilde{\lambda}}))_{\widetilde{\mathfrak{m}}}$ and
		$R \Gamma (X_{B}^{\widetilde{K} \cap B}, p_*^{B \cap \widetilde{K}} (\mathcal{V}_{\widetilde{\lambda}}/ \varpi \mathcal{V}_{\widetilde{\lambda}}))_{\widetilde{\mathfrak{m}}}$.
		Let $\widetilde{K}' < \widetilde{K}$ be a small enough good open normal subgroup such that $\widetilde{K}'_v = \widetilde{K}_v$ for all places $v \nmid p$ and $p_*^{\widetilde{K}'} (\mathcal{V}_{\widetilde{\lambda}}/\varpi  \mathcal{V}_{\widetilde{\lambda}})$ is trivial. Then we have Hochschild-Serre spectral sequences
		 \begin{linenomath*} \begin{align*}
			H^{i}(\widetilde{K}/\widetilde{K}', H^j_c(X_{Q}^{\widetilde{K}' \cap Q},  p_*^{\widetilde{K} \cap Q} (\mathcal{V}_{\widetilde{\lambda}}/ \varpi \mathcal{V}_{\widetilde{\lambda}}))_{\widetilde{\mathfrak{m}}} ) & \implies H^{i + j}_c(X_{Q}^{\widetilde{K} \cap Q},  p_*^{\widetilde{K} \cap Q} (\mathcal{V}_{\widetilde{\lambda}}/ \varpi \mathcal{V}_{\widetilde{\lambda}}))_{\widetilde{\mathfrak{m}}} \\
						H^{i}(\widetilde{K}/\widetilde{K}', H^j(X_{B}^{\widetilde{K} \cap B}, p_*^{B \cap \widetilde{K}} (\mathcal{V}_{\widetilde{\lambda}}/ \varpi \mathcal{V}_{\widetilde{\lambda}}))_{\widetilde{\mathfrak{m}}} ) & \implies H^{i + j}(X_{B}^{\widetilde{K} \cap B}, p_*^{B \cap \widetilde{K}} (\mathcal{V}_{\widetilde{\lambda}}/ \varpi \mathcal{V}_{\widetilde{\lambda}}))_{\widetilde{\mathfrak{m}}}
		\end{align*} \end{linenomath*} 
		which reduce the problem to showing that
		\[
		R \Gamma_c(X_Q^{\widetilde{K} \cap Q}, p_*^{\widetilde{K} \cap Q} \underline{k})_{\widetilde{\mathfrak{m}}} = 0 = R \Gamma(X_Q^{\widetilde{K} \cap Q}, p_*^{\widetilde{K} \cap Q} \underline{k})_{\widetilde{\mathfrak{m}}}
		\]
		where $Q \neq P$, $k$ is a finite field of characteristic $p$ and we have replaced $\widetilde{K}$ by $\widetilde{K}'$. By duality it is enough to only show that 
		\[
		R \Gamma(X_Q^{\widetilde{K} \cap Q}, p_*^{\widetilde{K} \cap Q} \underline{k})_{\widetilde{\mathfrak{m}}} = E_Q^{Q \cap \widetilde{K}}(k)_{\widetilde{\mathfrak{m}}} = 0
		\]
		Now proposition \ref{hecke_parabolic_satake_prop} implies that
		\[
		E_Q^{Q \cap \widetilde{K}}(k) \cong E_M^{M \cap \widetilde{K}}( R\Gamma((U \cap \widetilde{K})_S, k)),
		\]
		where $U$ is the unipotent radical of $Q$ and $M$ is a Levi subgroup of $Q$. By shrinking $\widetilde{K}$ even further we can assume that $R\Gamma((U \cap \widetilde{K})_S, k)$ is a complex of trivial $M \cap \widetilde{K}$-modules.
		Now  $ E_M^{M \cap \widetilde{K}}( R\Gamma((U \cap \widetilde{K})_S, k))$ is a direct sum of shifts of $E_M^{M \cap \widetilde{K}}(k)$. Hence it suffices to show that $E_M^{M \cap \widetilde{K}}(k)_{\widetilde{\mathfrak{m}}} = 0$, where we see $E_M^{M \cap \widetilde{K}}(k)$ as a $\widetilde{\mathbb{T}}^S$-module via the map $\operatorname{Sat}$. After possibly shrinking $\widetilde{K}$ we can assume that $\widetilde{K} \cap M = \prod_{i = 1}^t K_i$, where $M = \prod_{i = 1}^t M_i$ and each $K_i$ is a good open compact subgroup of $M_i(\mathbb{A}_{F^+}^\infty)$. Now there is an isomorphism
		\[
		\mathcal{H}(M(\mathbb{A}_{F^+}^S), \widetilde{K} \cap M) \to \bigotimes_{i = 1}^t \mathcal{H}(M_i(\mathbb{A}_{F^+}^S), K_i)
		\]
		Hence if $\mathfrak{m}' < \mathcal{H}(M(\mathbb{A}_{F^+}^S), \widetilde{K} \cap M)$ is a maximal ideal in the support of $E_M^{M \cap \widetilde{K}}(k)$ such that $\operatorname{Sat}^{-1}(\mathfrak{m}') = \widetilde{\mathfrak{m}}$, then there exist maximal ideals $\mathfrak{m}_i < \mathcal{H}(M_i(\mathbb{A}_{F^+}^S), K_i)$ such that $\mathfrak m$ is the product of these $\mathfrak{m}_i$ in the obvious sense. Similarly we have a factorisation
		\[
		E_{M}^{\widetilde{K} \cap M}(k) = \bigotimes_{i = 1}^t E_{M_i}^{K_i}(k)
		\]
		compatible with this isomorphism. Hence each $\mathfrak{m}_i$ is in the support of $E_{M_i}^{K_i}(k)$ and there exist semisimple Galois representations $\overline{\rho}_{\mathfrak{m}_i}$ by \cite[Theorem 2.3.5 and Theorem 2.3.8]{10author}. 
		Since $\operatorname{Sat}^{-1}(\mathfrak{m}') = \widetilde{\mathfrak{m}}$, a computation like \cite[Lemma 4.6]{newton_thorne_16}
		shows that if $Q \neq P$ is a proper parabolic subgroup, then 
		$\overline{\rho}_{\widetilde{\mathfrak{m}}}$ has at least 3 simple factors. 
		But from the computation of $\operatorname{Sat}(\widetilde{P}_v)$ after \cite[5.3]{newton_thorne_16} we know that $\overline{\rho}_{\widetilde{\mathfrak{m}}} \cong \overline{\rho}_{\mathfrak m} \oplus \overline{\rho}_{\mathfrak m}^{c,\vee}(1 - 2n)$ has only 2 simple factors by the non-Eisenstein assumption. 
		Thus, there are no maximal ideals $\mathfrak{m}' < \mathcal{H}(M(\mathbb{A}_{F^+}^S), \widetilde{K} \cap M)$ in the support of $E_M^{M \cap \widetilde{K}}(k)$ such that $\operatorname{Sat}^{-1}(\mathfrak{m}') = \widetilde{\mathfrak{m}}$. Consequently, $\widetilde{\mathbb{T}}^S(E_M^{M \cap \widetilde{K}}(k))/\widetilde{\mathfrak{m}}$ has no maximal ideals. This is absurd, hence $\widetilde{\mathfrak{m}}$ is not in the support of $E_M^{M \cap \widetilde{K}}(k)$ as desired.
	\end{proof} 
	
	\section{Hunting Hecke Eigenvalues}
	
	Our basic strategy for accessing Galois representations is by finding suitable congruences of a system of Hecke eigenvalues to another one, where the attached Galois representations are known to exist and have the desired properties.
	In the notation of the discussion after diagram \eqref{locally_symmetric_spaces_diagram} we would like to show that the Hecke eigensystem on $H^q(X_G^K, \mathcal{V}_\lambda)$ occurs in $H^d(X_{\widetilde{G}}^{\widetilde{K}}, \mathcal{V}_{\widetilde{\lambda}})$, where $\widetilde{\lambda}$ is a suitable weight of $\widetilde{G}$ and $\lambda$ is the weight of $G$ corresponding to $\widetilde{\lambda}$. We achieve something slightly weaker but sufficient in corollary \ref{final_hecke_congruence_cor} below. After that the properties about Galois representations follow from the general theorems on integral $p$-adic Hodge theory of \cite{torsion_p_adic}.
	
	\subsection{Hecke Eigenvalues Occur in Things}
	
	Since we will have to obtain our congruences of Hecke eigenvalues through a somewhat involved argument by chasing through a spectral sequence and everything is ``up to nilpotents", we introduce a notation to abbreviate the precise notion of Hecke congruence we mean. 
	
	\begin{definition}
		For a ring $T$ and $T$-module $M$, we set $T(M) := T/ \Ann_T(M) \hookrightarrow \End(M)$.
	\end{definition}
	
	\begin{lemma}
		If $T$ is a ring and $M$ and $N$ are $T$-modules such that $M$ is a subquotient of $N$, then $T(M)$ is a quotient of $T(N)$.
	\end{lemma}
	
	\begin{proof}
		If $t \in T$ annihilates $N$, then it also annihilates $M$. This implies that $T(M)$ is a quotient of $T(N)$.
	\end{proof}
	
	\begin{definition} \label{occur_def}
		Let $T$ be a ring and $M, N$ be $T$-modules. We say that $M$ \emph{occurs in} $N$ if 
		$\Ann(N)^n \subset \Ann(M)$ for some $n \geq 1$. We write $M \prec N$. If we want to emphasize that the exponent $n = n(x)$ only depends on some variable $x$, we write $M \prec_x N$.
	\end{definition}
	
	\begin{lemma}
		If $M \prec N$, then there is a natural surjection $T(N) \to T(M)/J$ for some nilpotent ideal $J$. 
	\end{lemma}
	
	\begin{proof}
		By assumption we know that $\Ann(N)^n \subset \Ann(M)$, hence 
		$T/\Ann(N)^n$ surjects onto $T/\Ann(M)$ and if $J$ denotes image of $\Ann(N)$ in $T/\Ann(M)$, then $J^n = 0$ and $T(N)$ surjects onto $T(M)/J$.
	\end{proof}
	
	\begin{lemma}
		The relations $\prec$ and $\prec_x$ are transitive and reflexive.
	\end{lemma}
	
	\begin{proof}
		Clear.
	\end{proof}
	
	\begin{lemma} \label{occur_ses_lemma}
		Let $T$ be a ring and 
		\[
		0 \to M' \to M \to M'' \to 0
		\]
		be an exact sequence of $T$-modules. Then
		\[
		M \prec M' \oplus M''.
		\]
	\end{lemma}
	
	\begin{proof}
		We have 
		\[
		(\Ann(M') \cap \Ann(M''))^2 \subset \Ann(M') \Ann(M'') \subset \Ann(M)
		\] 
		and the result follows since $\Ann(M' \oplus M'') = \Ann(M') \cap \Ann(M'')$. 
	\end{proof}
	
	\begin{lemma}
		Let $T$ be a ring and $E_{r}^{p,q}, r \geq 2$ a first quadrant spectral sequence of $T$-modules. For any $r \geq 2$, $E_\infty^{p,q}$ is a subquotient of $E_r^{p,q}$.
	\end{lemma}
	
	\begin{proof}
		It suffices to show that $E_{r+1}^{p,q}$ is a subquotient of $E_r^{p,q}$. But this is the case by definition.
	\end{proof}
	
	\begin{lemma} \label{spectral_occurs_lemma}
		Let $T$ be a ring and let $E_{r}^{p,q}, r \geq 2$ and $\overline{E}_{r}^{p,q}, r \geq 2$ be spectral sequences of $T$-modules.
		Suppose that there is a morphism of spectral sequences $\phi_r : E_r \to \overline{E}_r$ and that
		the spectral sequence $\overline{E}_r$ is trivial, i.e. all its differentials $\overline{d}_r$ vanish for $r \geq 2$.
		Then the images $I_r^{p,q} := \phi_r(E_r^{p,q})$ satisfy
		\[
		I_r^{p,q} \prec I_{r+1}^{p,q} \oplus I_{r}^{p + r, q -r + 1}.
		\]
	\end{lemma}
	
	\begin{proof}
		Consider the short exact sequence
		\[
		0 \to \underbrace{\phi_r(\ker(d_r^{p,q}))}_{=: A} \to I_r^{p,q} \to \underbrace{I_r^{p,q}/\phi_r(\ker(d_r^{p,q}))}_{=:B} \to 0.
		\]
		Firstly, we analyse $A$. Note that by assumption we have 
		$\phi_r \circ d_r^{p - r, q + r - 1} = \overline{d}_r^{p - r, q + r -1} \circ \phi_r = 0$
		, hence $\phi_r : \ker(d_r^{p,q}) \to I_r^{p,q}$ gives a well-defined map 
		$E_{r+1}^{p,q} \to I_r^{p,q}$ which coincides with $\phi_{r+1}$ if we identify
		$\overline{E}_r^{p,q}$ and $\overline{E}_{r+1}^{p,q}$. In conclusion
		\[
		A = I_{r+1}^{p,q}.
		\]
		On the other hand $E_r^{p,q}/ \ker(d_r^{p,q})$ surjects onto $B$ via $\phi_r$. Moreover, 
		$E_r^{p,q}/ \ker(d_r^{p,q})$ embeds into $E_{r}^{p + r, q - r + 1}$ via $d_r^{p,q}$. Hence
		$B$ is a subquotient of $\phi_r (E_{r}^{p + r, q - r + 1}) = I_{r}^{p + r, q - r + 1}$. The claim now follows from lemma \ref{occur_ses_lemma}.
	\end{proof}
	
	\subsection{Degree Shifting Revisited}
	
	In this section, we adapt the degree shifting argument of \cite[Section 4]{10author}
	to the case when $p$ ramifies in $F$. The difference is that a certain spectral sequence does not degenerate anymore. To address this we follow an idea of Ana Caraiani and James Newton and show that for high enough level the differentials in the aforementioned spectral sequence are divisible by high powers of $p$. Moreover, thanks to an observation of Ana Caraiani this approach is also able to completely avoid Kostant's formula! 
	
	\begin{definition}
		Let $\mathcal{A}$ be an abelian category and $C$ a complex in $\mathbf{D}(\mathcal{A})$. $C$ is called formal if it is quasi-isomorphic to 
		the complex $H^\bullet(C)$ with zero differentials.
	\end{definition}
	
	\begin{lemma}[Caraiani--Newton] \label{caraiani_newton_lemma}
		Let $A$ be an Artinian ring. Let $G$ be a profinite group and $C$ be a bounded complex of smooth $A[G]$-modules such that $C$ is perfect in $\mathbf{D}(A)$. Then there exists an open subgroup $H < G$ such that $C \in \mathbf{D}(\mathrm{Mod}_{H}^{sm}(A))$ is quasi-isomorphic to a complex of $A$-modules with trivial $H$-action.
	\end{lemma}
	
	\begin{proof}
		Choose a bounded complex $P$ of finitely generated projective $A$-modules which is quasi-isomorphic to $C$ in $\mathbf{D}(A)$. Let $G$ act trivially on $P$. Now we have
		\[
		\Hom_{\mathbf{D}(\mathrm{Mod}_H^{sm}(A))}(P, C) = \Hom_{\mathbf{D}(A)}(P, R\Gamma(H, C)) = H^0( P^\vee \otimes_A^{\mathbb{L}} R\Gamma(H, C)),
		\]
		by \cite[\href{https://stacks.math.columbia.edu/tag/07VI}{Lemma 07VI}]{stacks-project}. But the $H$-action on $P^\vee$ is trivial, so we also have
		\[
		\Hom_{\mathbf{D}(\mathrm{Mod}_H^{sm}(A))}(P, C) = H^0( R\Gamma(H, P^\vee \otimes_{A}^{\mathbb{L}} C))
		\]
		Now choose a complex $I$ of injective, smooth $A[G]$-modules representing $C$. Then 
		\[
		\Hom_{\mathbf{D}(\mathrm{Mod}_H^{sm}(A))}(P, C) = H^0( (P^\vee \otimes_{A}^{\mathbb{L}} I)^H ).
		\] 
		Since filtered colimits are exact, 
		we find that the natural map 
		\[
		\operatorname{colim}_{\substack{H < G}} \Hom_{\mathbf{D}(\mathrm{Mod}_H^{sm}(A))}(P, C) \to \Hom_{\mathbf{D}(A)}(P, C)
		\]
		is an isomorphism. Now the identity map on the right hand side has to lie in the image of $\Hom_{\mathbf{D}(\mathrm{Mod}_H^{sm}(A))}(P, C)$ for some open subgroup $H < G$, hence for such an $H$ we have $P \cong C$ in $\mathbf{D}(\mathrm{Mod}_{H}^{sm}(A))$.
	\end{proof}
	
	\begin{corollary}
		Suppose $C$ is a complex of smooth $A[G]$-modules such that $C$ is perfect and formal in $\mathbf{D}(A)$, then there exists an open subgroup $H < G$ such that 
		$C$ is formal in $\mathbf{D}(\operatorname{Mod}_H^{sm}(A))$. 
	\end{corollary}
	
	\begin{proof}
		With $H$ as in the lemma, the quasi-isomorphism $C \cong H^\bullet(C)$ is $H$-equivariant. 
	\end{proof}
	
	\begin{remark}
		The lemma and its corollary replace \cite[Lemma 4.2.3]{10author} following a suggestion of Peter Scholze. I learnt the proof from Ana Caraiani and James Newton and thank them for allowing me to include it here. I do not know in general when we can expect that the group cohomology complex $R\Gamma(U, \mathcal{O}/\varpi^m)$ is formal as a complex of $\mathcal{O}/\varpi^m [ G]$-modules. It might be interesting to answer this ``elementary" algebraic question.
	\end{remark}
	
	\begin{lemma} \label{invariants_direct_sum_lemma}
		Let $G = N \rtimes M $
		be a profinite group. If $A$ is a smooth $R[G]$-module such that there exists $R[G]$-modules $B,C$ and a decomposition $A = B \oplus C$ as $R[M]$-modules, with $B^N = B$, then $B[0]$
		is a direct summand of $R\Gamma(N, A)$ in the derived category of smooth $R[G/N]$-modules.
	\end{lemma}
	
	\begin{proof}
		We use inhomogeneous cochains to represent $R \Gamma(N, A)$. 
		Namely, define a complex with graded pieces $C^i(N, A) := \{f : N^i \xrightarrow{cont} A \}$ and differentials
		 \begin{linenomath*} \begin{align*}
			d &: C^i(N,A) \to C^{i+1}(N,A) : f \mapsto df \\
			df(g_0,\dots, g_i) & = g_0 f(g_1, \dots, g_i) \\ &+ \sum_{j = 1}^{i} (-1)^j f(g_0, \dots, g_{j-1} g_{j}, \dots, g_i) + (-1)^{i + 1} f(g_0, \dots, g_{n-1})
		\end{align*} \end{linenomath*} 
		Then a standard argument shows that $C^\bullet(N,A)$ is quasi-isomorphic to $R\Gamma(N,A)$, where
		$\sigma \in G/N$ acts on $f \in C^i(N,A)$ by
		\[
		(\sigma f)(g_1, \dots, g_n) = \sigma f(\sigma^{-1} g \sigma, \dots, \sigma^{-1} g \sigma).
		\]
		Now let $e : A \to B$ be the projection, then we define 
		$e' : C^\bullet(N,A) \to C^\bullet(N,A)$ by
		\[
		e' f = \begin{cases}
			e f & f \in C^0(N,A) \\
			0 & f \in C^{i}(N,A), i > 0 
		\end{cases}
		\]
		We have $ e' d = d e' = 0$ since each $b \in B$ is $N$-invariant
		, hence $db (n) = n b - b = 0$ for all $n \in N$.
	\end{proof}
	
	Let $E$ be a finite extension of $\rationals_p$ which contains all the $p$-adic embeddings of $F$ and let $\mathcal O$ be the ring of integers of $E$.
	In the following, the symbol $\prec$ always refers $\widetilde{\mathbb{T}}^{\overline{S}}_{\mathcal O} = \widetilde{\mathbb{T}}^{\overline{S}} \otimes \mathcal{O}$-modules for some finite set of places $S$, as defined in the previous section.
	
	\begin{prop} \label{degree_shifting_prop}
		Suppose we have a disjoint union $\overline{S}_p = \overline{S}_1 \cup \overline{S}_2$ and $S_p = S_1 \cup S_2$ a compatible decomposition. Suppose $\widetilde K$ is a good open compact subgroup of $\widetilde {G}(\mathbb{A}_{F^+}^\infty)$ such that $G(\mathbb{A}_{F^+}^\infty) \cap \widetilde K \to (P(\mathbb{A}_{F^+}^\infty) \cap \widetilde K)/(U(\mathbb{A}_{F^+}^\infty) \cap \widetilde K)$ is an isomorphism, where $U$ is the unipotent radical of $P$. Let $K = \widetilde{K} \cap G$ and $S \supset S_p$ a finite set of places such that $\widetilde{K}_v = \widetilde{G}(\mathcal{O}_{F^+_{v}})$ for $v \not \in \overline{S}$. Let
		$V_i$ be continuous $\mathcal{O}[K_{S_i}]$-modules for $i=1,2$, where $\mathcal{O}$ is the ring of integers  in a finite extension of $\rationals_p$ with uniformizer $\varpi$. Assume further that the $V_i$ are finite free over $\mathcal{O}$. Let $\widetilde{V}_1$ be a continuous $\mathcal{O}[\widetilde{K}_{\overline{S}_1}]$-module, finite free over $\mathcal{O}$.
		
		Let $d = 2n^2[F^+: \rationals] = 1 + \dim X_{\GL_n/F}$ and suppose that the following conditions are satisfied:
		\begin{enumerate}[(1)]
			\item We have 
			\[
			\sum_{\overline{v} \in \overline{S}_2} [F^+_{\overline{v}} : \rationals_p] > \frac{1}{2} [F^+: \rationals].
			\]
			\item We have a $K$-equivariant direct sum decomposition $\widetilde{V}_1 \cong W_1 \oplus W_2$ with an isomorphism of $K$-modules $V_1 \cong W_1$ such that $W_1 \subset \widetilde{V}_1^{\widetilde{K} \cap U}$.
		\end{enumerate} 
		Fix an integer $m \geq 1$.
		Then for $m' \geq m$ large enough, we have
		\[
		H^q(X^{K(m')}_G, \mathcal{V}/\varpi^m) \prec H^d(X^{\widetilde{K}(m')}_{P}, \widetilde{\mathcal{V}})
		\]
		for $\lfloor d/2 \rfloor \leq q \leq d - 1$,
		where
		 \begin{linenomath*} \begin{align*}
		\widetilde{K}(m') & := \left\{ g \in \widetilde{K} : \forall \overline{v} \in \overline{S}_2, \,\, g \equiv \begin{bmatrix}
			\id & * \\ 0 & \id
		\end{bmatrix} \pmod {\varpi_{\overline{v}}^{m'}} \right\} \\
		K(m') & := \left\{ g \in K : \forall v \in S_2, \,\, g \equiv \id \pmod {\varpi_{v}^{m'}} \right\} 
		\end{align*} \end{linenomath*} 
		and $\mathcal{V}$ is the local system associated with $V_1 \otimes V_2$ and $\widetilde{\mathcal{V}}$ is the local system associated with $\widetilde{V}_1$.
	\end{prop}
	
	\begin{proof}
		By proposition \ref{hecke_parabolic_satake_prop} we have a $\widetilde{\mathbb{T}}^{\overline{S}} \otimes \mathcal{O}$-equivariant isomorphism
		\[
		R \Gamma(X_P^{\widetilde{K} \cap P}, \widetilde{\mathcal{V}}) = 
		R \Gamma(X_G^K, p^K_* s^* R \Gamma((U \cap \widetilde{K})_{\overline{S}}, \widetilde{V}_1)).
		\]
		Let $m' \geq m$ be an integer. Since $\widetilde{V}_1/\varpi^m$ is an abelian $p$-group, we find that 
		 \begin{linenomath*} \begin{align*}
			R \Gamma((U \cap \widetilde{K}(m'))_S, \widetilde{V}_1/\varpi^m) &= R \Gamma( (U \cap \widetilde{K}(m'))_{\overline{S}_p}, \widetilde{V}_1/\varpi^m)
		\end{align*} \end{linenomath*} 
		and by the K\"unneth formula, this is equal to
		 \begin{linenomath*} \begin{align*}
			R\Gamma((U \cap \widetilde{K})_{\overline{S}_1}, \widetilde{V}_1/\varpi^m) \otimes_{\mathcal{O}/\varpi^m}^{\mathbb{L}}  R\Gamma((U \cap \widetilde{K}(m'))_{\overline{S}_2}, \mathcal{O}/\varpi^m).
		\end{align*} \end{linenomath*} 
		Note that $U$ is abelian for the Siegel parabolic $P$, hence $(U \cap \widetilde{K}(m')) \cong \ints_p^N$ for some $N$. Hence
		all the $H^j((U \cap \widetilde{K}(m'))_{\overline{S}_2}, \mathcal{O})$ are torsion free and
		\[
		H^i((U \cap \widetilde{K}(m'))_{\overline{S}_2}, \mathcal{O}/\varpi^m) = H^i((U \cap \widetilde{K}(m'))_{\overline{S}_2}, \mathcal{O}) \otimes_{\mathcal{O}} \mathcal{O}/\varpi^m.
		\]
		The natural morphism
		\[
		V_1 \otimes_{\mathcal{O}}^{\mathbb{L}}  R\Gamma((U \cap \widetilde{K}(m'))_{\overline{S}_2}, \mathcal{O}) \to V_1/\varpi^m \otimes_{\mathcal{O}/\varpi^m}^{\mathbb{L}}  R\Gamma((U \cap \widetilde{K}(m'))_{\overline{S}_2}, \mathcal{O}/\varpi^m),
		\]
		induces a morphism of spectral sequences $E_r^{i,j} \to \overline{E}_r^{i,j}$,
		where
		 \begin{linenomath*} \begin{align*}
			E_2^{i,j} &= H^i(X_G^{K(m')}, p_*^{K(m')} s^* V_1 \otimes H^j((U \cap \widetilde{K}(m'))_{\overline{S}_2}, \mathcal{O} )) \\
			\overline{E}_2^{i,j} &= H^i(X_G^{K(m')}, p_*^{K(m')} s^* V_1/\varpi^m \otimes H^j((U \cap \widetilde{K}(m'))_{\overline{S}_2}, \mathcal{O})/\varpi^m ).
		\end{align*} \end{linenomath*} 
		By assumption, $V_1$ is a direct summand of $\widetilde{V}_1$, hence by lemma \ref{invariants_direct_sum_lemma} also of $R\Gamma((U \cap \widetilde{K})_{\overline{S}_1}, \widetilde{V}_1)$. Consequently, $E_2^{i,j}$ converges to a direct summand of $H^{i + j}(X_P^{\widetilde{K}(m')}, \widetilde{\mathcal{V}})$. Similarly, $\overline{E}_2^{i,j}$ converges to a direct summand of $H^{i + j}(X_P^{\widetilde{K}(m')}, \widetilde{\mathcal{V}}/ \varpi^m)$. In particular, $E_\infty^{i,j}$ is a subquotient of $H^{i + j}(X_P^{\widetilde{K}(m')}, \widetilde{\mathcal{V}})$, hence $E_\infty^{i,j} \prec H^{i + j}(X_P^{\widetilde{K}(m')}, \widetilde{\mathcal{V}})$.
		
		By lemma \ref{caraiani_newton_lemma} we see that we can choose $m'$ large enough, so that $R\Gamma((U \cap \widetilde{K}(m'))_{\overline{S}_2}, \mathcal{O})/\varpi^m$ is formal in $\mathbf{D}(\mathrm{Mod}_{K(m')}^{sm}(\mathcal{O}/\varpi^m))$. Thus, for such $m'$ the spectral sequence $\overline{E}_2^{i,j}$ degenerates (recall that all the constructions implicitly depend on $m'$). Hence the hypotheses of lemma \ref{spectral_occurs_lemma} are satisifed.
		(Both $E_r$ and $\overline{E}_r$ are Hecke equivariant since the construction of Grothendieck spectral sequences is functorial and we have proposition \ref{hecke_parabolic_satake_prop}.)
		Consequently,
		\[
		I_r^{i,j} \prec I_{r+1}^{i,j} \oplus I_r^{i + r, j - r + 1},
		\]
		where $I_r^{i,j}$ is the image of $E_r^{i,j}$ in $\overline{E}_r^{i,j}$.
		Additionally, it follows from a standard long exact sequence argument that the natural map
		\[
		E_2^{i,j}/\varpi^m \to \overline{E}_2^{i,j}
		\]
		is injective.
		
		We prove the following claim by reverse induction on $n$:
		\begin{itemize}
			\item Let $n \geq \lfloor d/2 \rfloor$, then for all $q \geq n$ and $m \geq 1$, there exists an integer $m_0$ such that for all $m' \geq m_0$ we have
			\[
			H^q(X^{K(m')}_G, \mathcal{V}/\varpi^m) \prec H^d(X^{\widetilde{K}(m')}_P, \widetilde{\mathcal{V}})
			\]
		\end{itemize}
		For $n = d$, the claim is vacuous so we may assume that $\lfloor d/2 \rfloor \leq n \leq d - 1$ and that the claim holds for $n + 1$. We need to prove something for $q = n$. Fix an $m \geq 1$ and let $m'$ be sufficiently large so that the spectral sequence $\overline{E}_2^{i,j}$ above degenerates and such that $K(m')$ acts trivially on $V_2/\varpi^m$ and $H^{\bullet}((U \cap \widetilde{K}(m'))_{\overline{S}_2}, \mathcal{O} )/\varpi^m$.
		
		By assumption we have that
		\[
		2 \dim \prod_{\overline{v} \in \overline{S}_2} U(F^+_{\overline{v}}) = \sum_{\overline{v} \in \overline{S}_2} 2 n^2 [F^+_{\overline{v}} : \rationals_p] > d/2 \geq d - q,
		\] 
		hence $H^{d - q}((U \cap \widetilde{K}(m'))_{\overline{S}_2}, \mathcal{O} )$ is not zero and free over $\mathcal{O}$ by \cite[Lemma 4.2.2]{10author}. Thus, we may simply replace $V_2$ by $H^{d - q}((U \cap \widetilde{K}(m'))_{\overline{S}_2}, \mathcal{O} )$ in the claim we are proving! 
		Now we have $H^q(X^{K(m')}_G, \mathcal{V}/\varpi^m) = \overline{E}_2^{q,d-q}$ and the long exact sequence in cohomology associated with $0 \to \mathcal{V} \xrightarrow{\varpi^m} \mathcal{V} \to \mathcal{V}/\varpi^m \mathcal{V}$ implies that
		\[
		\overline{E}_2^{i,j} \prec I_2^{q,d - q} \oplus H^{q+1}(X^{K(m')}_G,\mathcal{V})[\varpi^m].
		\]
		For $m''$ large enough, we also have 
		\[
		 H^{q+1}(X^{K(m')}_G,\mathcal{V})[\varpi^m] \prec H^{q+1}(X^{K(m')}_G,\mathcal{V}/\varpi^{m''}).
		\] 
		Moreover, from the above application of lemma \ref{spectral_occurs_lemma} we obtain
		\[
		I^{q,d-q}_2 \prec_{n^2 [F^+ : \rationals]} I^{q,d-q}_\infty \bigoplus_{r \geq 2} I_2^{q + r, d - q - r + 1}
		\]
		By the inductive hypothesis we may choose $m'$ large enough so that
		\[
		I_2^{q + r, d - q - r + 1} \prec \overline{E}_2^{q + r, d - q - r + 1} \prec H^d(X_P^{\widetilde{K}(m')}, \widetilde{\mathcal{V}})
		\]
		and 
		\[
		H^{q+1}(X^{K(m')}_G,\mathcal{V}/\varpi^{m''}) \prec H^d(X_P^{\widetilde{K}(m')}, \widetilde{\mathcal{V}}).
		\]
		We conclude the proof of the claim by using that 
		$E_\infty^{q, d - q}$ is a subquotient of the module $H^d(X_P^{\widetilde{K}(m')}, \widetilde{\mathcal{V}})$ and $I_\infty^{q, d- q} \prec E_\infty^{q, d -q}$.
	\end{proof}
	
	\begin{corollary} \label{final_hecke_congruence_cor}
		Let $\overline{v}, \overline{v}'$ be distinct elements of $\overline{S}_p$, and let $\lambda \in (\ints^n)^{\Hom(F,E)}$ be a dominant weight for $G$. Let $\widetilde{K} \subset \tilde{G}(\mathbb{A}_{F^+}^\infty)$ be a good open compact subgroup and $S$ a set of places of $F$ satisfying $(\Sigma_p)$ such that $\widetilde{K}_{\overline{v}} = \widetilde{G}(\mathcal{O}_{F^+_{\overline{v}}})$ for $\overline{v} \in \overline{S}$ and $G(\mathbb{A}_{F^+}^\infty) \cap \widetilde K \to (P(\mathbb{A}_{F^+}^\infty) \cap \widetilde K)/(U(\mathbb{A}_{F^+}^\infty) \cap \widetilde K)$ is an isomorphism.
		Suppose that the following conditions are satisfied:
		\begin{enumerate}[(1)]
			\item For each embedding $\tau : F \hookrightarrow E$ inducing the place $\overline{v}$ of $F^+$, we have $- \lambda_{\tau c, 1} - \lambda_{\tau,1} \geq 0$
			\item We have 
			\[
			\sum_{\overline{v}'' \in \overline{S}_p \setminus \{\overline{v}, \overline{v}'\}} [F^+_{\overline{v}''} : \rationals_p] > \frac{1}{2} [F^+: \rationals].
			\]
			\item $\mathfrak{m} \subset \mathbb{T}^S$ is a non-Eisenstein maximal ideal in the support of $H^\bullet(X^{K}_{\GL_n/F}, \mathcal{V}_\lambda)$ such that $\overline{\rho}_{\widetilde{\mathfrak{m}}}$ is decomposed generic, where $\widetilde{\mathfrak{m}} = \operatorname{Sat}^{-1}(\mathfrak{m})$.
		\end{enumerate} 
		Define a weight dominant weight $\widetilde \lambda \in (\mathbb{Z}^{2n})^{\Hom(F^+, E)}$ for $\widetilde{G}$ as follows: if $\tau : F^+ \to E$ does not induce either $\overline{v}$ or $\overline{v}'$, then
		$\widetilde{\lambda}_\tau = 0$. If $\tau$ induces $\overline{v}$, then we set
		\[
		\widetilde{\lambda}_{\tau} = (- \lambda_{\tilde{\tau}c,n}, \dots, -\lambda_{\tilde{\tau}c,1}, \lambda_{\tilde{\tau},1}, \dots, \lambda_{\tilde{\tau},n}).
		\] 
		If $\tau$ induces $\overline{v}'$, then $\widetilde{\lambda}_{\tau}$ may be chosen arbitrarily from $\ints_+^{2n}$.
		Then for all integers $m$, there exists an integer $m' \geq m$
		such that 
		 \begin{linenomath*} \begin{align*}
				H^q(X^{K}_{G}, \mathcal{V}_\lambda/\varpi^m)_{\mathfrak{m}} & \prec H^d(X^{\widetilde{K}(m')}_{\widetilde{G}}, \mathcal{V}_{\widetilde{\lambda}})_{\widetilde{\mathfrak{m}}} \qquad \lfloor d /2 \rfloor \leq q \leq d - 1 \\
		j_* H^q(X^{K}_{G}, \mathcal{V}_\lambda/\varpi^m)_{j^{-1} \mathfrak{m}} & \prec H^d(X^{\widetilde{K}(m')}_{\widetilde{G}}, \mathcal{V}_{\widetilde{\lambda}}^\vee)_{\widetilde{j^{-1} \mathfrak{m}}} \qquad 0 \leq q < \lfloor d/2 \rfloor,
		\end{align*} \end{linenomath*} 
		where
		\[
		\widetilde{K}(m') :=
		\left\{ g \in \widetilde{K} :\forall \overline{v}'' \neq \overline{v}, \,\, g \equiv \begin{bmatrix}
			\id & * \\ 0 & \id
		\end{bmatrix} \pmod {\varpi_{\overline{v}''}^{m'}} \right\}
		\]
		and $j : \mathbb{T}^S \to \mathbb{T}^S$ is the isomorphism $[K^S g K^S] \mapsto [K^S g^{-1} K^S]$.
	\end{corollary}
	
	\begin{proof}
		Let $\overline{S}_1 = \{\overline{v}, \overline{v}'\}$ and 
		$\overline{S}_2 = \overline{S}_p \setminus \overline{S}_1$. Now let
		$\widetilde{V}_1$ be the $\widetilde{K}_{\overline{S}_1}$-module corresponding to $\widetilde{\lambda}$ from definition \ref{highest_weight_module_def}.
		Choose a decomposition $\Hom(F, \complex) = H \cup H c$, where $c$ denotes the complex conjugation of $F$. The restriction of a field homomorphism to $F^+$ induces a bijection $\alpha : H \cong \Hom(F^+, \complex)$. For $\tau \in H$, set
		 \begin{linenomath*} \begin{align*}
			\lambda'_{\tau} &= (\widetilde{\lambda}_{\alpha(\tau), n + 1}, \dots,  \widetilde{\lambda}_{\alpha(\tau), 2n}) \\
			\lambda'_{\tau c} &= (- \widetilde{\lambda}_{\alpha(\tau), n }, \dots,  - \widetilde{\lambda}_{\alpha(\tau), 1})
		\end{align*} \end{linenomath*} 
		Now $\lambda_{\tau}' = \lambda_\tau$ for all $\tau$ which induce a place above $\overline{v}$ and by \cite[2.11]{newton_thorne_16}, the $K$-module 
		$V_{\lambda'}$ is a direct summand of $V_{\widetilde{\lambda}}^{U \cap \widetilde{K}}$. If we choose $m'$ large enough, then
		$V_{\lambda}/\varpi^m$ and $V_{\lambda'}/\varpi^m$ are isomorphic $K(m')$-modules and by proposition \ref{degree_shifting_prop}, 
		we can increase $m'$ so that
		 \begin{linenomath*} \begin{align*}
			H^q(X^{K(m')}_G, \mathcal{V}_{\lambda}/ \varpi^m) \cong 	H^q(X^{K(m')}_G, \mathcal{V}_{\lambda'}/ \varpi^m)  \prec H^d(X_P^{K(m')}, \mathcal{V}_{\widetilde{\lambda}})
		\end{align*} \end{linenomath*} 
		for $\lfloor d /2 \rfloor \leq q \leq d - 1$. When $q < \lfloor d/2 \rfloor$, we apply \cite[Proposition 2.10]{newton_thorne_16} to $V_{\widetilde{\lambda}}$ and the opposite parabolic $P^{-}$ to obtain a direct sum decomposition $V_{\widetilde{\lambda}} = W_1 \oplus W_2$ of $K$-modules, where $W_1 \cong V_{w \lambda'}$ with 
				 \begin{linenomath*} \begin{align*}
			w\lambda'_{\tau} &:= (- \widetilde{\lambda}_{\alpha(\tau), n }, \dots,  - \widetilde{\lambda}_{\alpha(\tau), 1})\\
			w\lambda'_{\tau c} &:= (\widetilde{\lambda}_{\alpha(\tau), n + 1}, \dots,  \widetilde{\lambda}_{\alpha(\tau), 2n}) 
		\end{align*} \end{linenomath*} 
	 and $W_1^{U^-} = W_1$. Moreover, the main theorem of \cite{cabanes84} shows that that $W_2$ is $P \cap \widetilde{K}$-stable since $W_2 \otimes \overline{E} = ( 1 - U) V_{\widetilde{\lambda}} \otimes \overline{E}$. Dualising, we find that $V_{w \lambda'}^\vee$ is a direct summand of $V_{\widetilde{\lambda}}^\vee$ as required for proposition \ref{degree_shifting_prop}. 
		Thus, we can find $m' \geq m$ such that 
		 \begin{linenomath*} \begin{align*}
			H^{d - 1 - q}(X^{K(m')}_G, \mathcal{V}_{w\lambda'}^\vee/\varpi^m) \prec H^d(X_P^{\widetilde{K}(m')}, \mathcal{V}_{\widetilde{\lambda}}^\vee )
		\end{align*} \end{linenomath*} 
		for $0  \leq q < \lfloor d / 2 \rfloor$. With Verdier duality and \cite[Proposition 3.7]{newton_thorne_16} we find an isomorphism of Hecke modules
		\[
		\Hom(H^q_c(X^{K(m')}_G, \mathcal{V}_{w\lambda'}/ \varpi^m), \mathcal{O}/\varpi^m) \cong H^{d - 1 - q}(X^{K(m')}_G, \mathcal{V}_{w\lambda'}^\vee/\varpi^m),
		\]
		where $[K^S g K^S]$ on the right hand side acts as $[K^S g^{-1} K^S]$ on the left hand side. Since the left hand side has the same annihilator as $H^q_c(X^{K(m')}_G, \mathcal{V}_{w\lambda'}/ \varpi^m)$, we have 
		\[
		j_*H^q_c(X^{K(m')}_G, \mathcal{V}_{w\lambda'}/ \varpi^m) \prec H^{d - 1 - q}(X^{K(m')}_G, \mathcal{V}_{w\lambda'}^\vee/\varpi^m).
		\]
		Moreover, for $m'$ large enough, we have an isomorphism of Hecke modules
		\[
			H^q_c(X^{K(m')}_G, \mathcal{V}_{w\lambda}/ \varpi^m) \cong H^q_c(X^{K(m')}_G, \mathcal{V}_{w\lambda'}/ \varpi^m),
		\]
		where $w\lambda_{\tau} = \lambda_{\tau c}$. Consider the element 
		\[
			w = \begin{bmatrix}
			0 & - \id \\ \id & 0 
			\end{bmatrix} \in \widetilde{G}(\mathcal{O}_{F^+}),
		\]
		then for all $\overline{v} \not\in \overline{S}$, we have $w \in \widetilde{K}_{\overline{v}}$, thus conjugation by $w$ is the identity on the ring $\mathcal{H}(\widetilde{G}(F^+_{\overline{v}}), \widetilde{K}_{\overline{v}})$. On the other hand, $w$ normalises $G(\mathcal{O}_{F_v})$ and maps $V_\lambda$ to $V_{w \lambda}$, inducing an isomorphism of Hecke modules
		\[
		H^q_c(X^{K(m')}_G, \mathcal{V}_{\lambda}/ \varpi^m) \cong H^q_c(X^{w K(m') w^{-1}}_G, \mathcal{V}_{w\lambda}/ \varpi^m).
		\]
		Since $\mathfrak{m}$ is a non-Eisenstein ideal, we have 
		$H^q_c(X, \mathcal{V})_\mathfrak{m} = H^q(X, \mathcal{V})_\mathfrak{m}$
		for all $q$ by \cite[4.2]{newton_thorne_16}.
		By the Hochschild--Serre spectral sequence and duality
		we have
		\[
		H^q(X^K_G, \mathcal{V}_{\lambda}/\varpi^m)_{\mathfrak{m}} \prec \bigoplus_{i = q}^{d - 1} H^i(X_G^{K(m')}, \mathcal{V}_{\lambda}/\varpi^m)_{\mathfrak{m}},
		\] 
		hence we obtain
		 \begin{linenomath*} \begin{align*}
		H^q(X^K_G, \mathcal{V}_{\lambda}/\varpi^m)_{\mathfrak{m}} &\prec H^{d}(X^{\widetilde{K}(m')}_P, \mathcal{V}_{\widetilde{\lambda}} )_{\widetilde{\mathfrak{m}}} \qquad \lfloor d /2 \rfloor \leq q \leq d - 1 \\
		j_*H^q(X^K_G, \mathcal{V}_{\lambda}/\varpi^m)_{j^{-1} \mathfrak{m}} &\prec H^{d}(X^{w \widetilde{K}(m') w^{-1}}_{P^{-}}, \mathcal{V}_{\widetilde{\lambda}}^\vee )_{\widetilde{j^{-1} \mathfrak{m}}} \qquad 0 \leq q < \lfloor d/2 \rfloor
		\end{align*} \end{linenomath*} 
		and with proposition \ref{non_eisenstein_boundary_prop} and the main theorem of  \cite{caraiani_scholze_non_compact}, we conclude that
		 \begin{linenomath*} \begin{align*}
				H^q(X^{K}_{G}, \mathcal{V}_\lambda/\varpi^m)_{\mathfrak{m}} & \prec H^d(X^{\widetilde{K}(m')}_{\widetilde{G}}, \mathcal{V}_{\widetilde{\lambda}})_{\widetilde{\mathfrak{m}}} \qquad \lfloor d /2 \rfloor \leq q \leq d - 1 \\
		j_* H^q(X^{K}_{G}, \mathcal{V}_\lambda/\varpi^m)_{j^{-1} \mathfrak{m}} & \prec H^d(X^{w \widetilde{K}(m') w^{-1}}_{\widetilde{G}}, \mathcal{V}_{\widetilde{\lambda}}^\vee)_{\widetilde{j^{-1} \mathfrak{m}}} \qquad 0 \leq q < \lfloor d/2 \rfloor. 
		\end{align*} \end{linenomath*} 
		The claim follows from the $\widetilde{\mathbb{T}}^{\overline{S}}$-module isomorphism 
		\[
		H^d(X^{w \widetilde{K}(m') w^{-1}}_{\widetilde{G}}, \mathcal{V}_{\widetilde{\lambda}}^\vee) \cong H^d(X^{\widetilde{K}(m')}_{\widetilde{G}}, \mathcal{V}_{\widetilde{\lambda}}^\vee)
		\] 
		induced by conjugation by $w$.
	\end{proof}

	\subsection{Weak Semistable Local-Global Compatiblity}
	
	We now use the results of the previous section to show a weak form of local-global compatibility in the semistable case using the ring defined in \cite{torsion_p_adic}.
	We are able to show that the relevant automorphic Galois representations are semistable and obtain a bound on their Hodge--Tate weights but unfortunately we cannot say anything more about the exact Hodge--Tate weights. This is why we have the distinct Hodge--Tate weights assumption in our main theorem.
	The proof of the following uses the determinant laws introduced in \cite{chenevier_determinant} and a similar argument as in the proof of \cite[Proposition 4.4.6]{10author}.
	
	\begin{theorem} \label{weak_semistable_local_global_thm}
		Suppose that $S$ is a  set of places of $F$ satisfying $(\Sigma_p)$ and that $|\overline{S}_p| > 2$ and let $\lambda \in (\ints^n)^{\Hom(F,E)}$ be a dominant weight for $G$. Let $K \subset \GL_n(\mathbb{A}_{F}^\infty)$ be a good subgroup.
		Suppose that the following conditions are satisfied:
		\begin{enumerate}[(1)]
			\item For each embedding $\tau : F \hookrightarrow E$, we have $- \lambda_{\tau c, 1} - \lambda_{\tau,1} \geq 0$;
			\item For every place $v$ of $F$ above $p$, the group $K_v$ contains the Iwahori $Iw_v$;
			\item We have 
			\[
			\sum_{\overline{v}'' \in \overline{S}_p \setminus \{\overline{v}, \overline{v}'\}} [F^+_{\overline{v}''} : \rationals_p] > \frac{1}{2} [F^+: \rationals].
			\]
			for all pairs of distinct places $\overline{v}, \overline{v}'$ of $F^+$ lying above $p$;
			\item $\mathfrak{m} \subset \mathbb{T}^S(R \Gamma(X_G^K, \mathcal{V}_\lambda))$ is a non-Eisenstein maximal ideal such that $\overline{\rho}_{\mathfrak{m}}$ is absolutely irreducible and decomposed generic,
		\end{enumerate} 
		then for every integer $m \geq 1$, there exists a continuous representation
		\[
		\rho_{\mathfrak{m}} : G_{F,S} \to  \GL_n(\mathbb{T}^S(R \Gamma(X_G^K, \mathcal{V}_\lambda/\varpi^m))_{\mathfrak{m}} /I),
		\]
		such that
		\begin{itemize}
			\item The representation $\rho_{\mathfrak{m}}$ is torsion semistable with Hodge--Tate weights contained in the interval $[-r, r]$, as defined in \cite{torsion_p_adic}, where $r$ only depends on $\lambda$. 
			\item If $v \mid p$ is a place such that $K_v = \GL_n(\mathcal{O}_{F_v})$ and $K_{v^c} = \GL_n(\mathcal{O}_{F_{v^c}})$, then $\rho_{\mathfrak{m}}$ is torsion crystalline at $v$.
			\item For every $v \not \in S$, we have 
			$\det(X - \rho_{\mathfrak{m}}(\Frob_v)) = P_v(X)$.
			\item $I$ is a nilpotent ideal such that $I^x = 0$ for some $x$ that only depends on $[F:\rationals]$ and $n$.
		\end{itemize}
	\end{theorem}
	
	\begin{proof}
		By \cite{scholze_torsion} we already have a nilpotent ideal $J'$ and a representation $\rho_{\mathfrak{m}}$ with the right Frobenius eigenvalues. Thus, it is enough to show that after possibly increasing $J'$, $\rho_{\mathfrak{m}}$ becomes torsion semistable with Hodge--Tate weights bounded by $r$ for some $r$ only depending on $\lambda$.
		We first do this under the additional assumption that $\overline{\rho}_{\widetilde{\mathfrak{m}}}$ is decomposed generic.
		
		It is enough to show the claim for one place $v \mid p$ of $F$ at a time.
		So let us fix one and let $\overline{v}$ be the place of $F^+$ below $v$. Moreover, by assumption there exists a place $\overline{v}' \neq \overline{v}$ of $F^+$ above $p$. Now we let $\widetilde{K}$ be a good subgroup of $\widetilde{G}(\mathbb{A}_{F^+}^\infty)$ such that $\widetilde{K} \cap G = K$ and apply corollary \ref{final_hecke_congruence_cor} to $\overline{v}, \overline{v}'$ and a $\widetilde{\lambda}$ that is CTG \cite[Definition 4.3.5]{10author} using \cite[Lemma 4.3.6]{10author}. Note that this also implies that $\widetilde{\lambda}^\vee$, the highest weight of $V_{\widetilde{\lambda}}^\vee$, is CTG.
		
		Consider the rings 
		 \begin{linenomath*} \begin{align*}
			\widetilde{A}(\widetilde{K}(m')) &:= \widetilde{\mathbb{T}}^S(H^d(X_{\widetilde{G}}^{\widetilde{K}(m')}, \mathcal{V}_{\widetilde{\lambda}})_{\widetilde{\mathfrak{m}}} \oplus H^d(X_{\widetilde{G}}^{\widetilde{K}(m')}, \mathcal{V}_{\widetilde{\lambda}}^\vee)_{\widetilde{j^{-1} \mathfrak{m}}} ), \\
			A(m,\lambda) &:= \mathbb{T}^S( R\Gamma(X_G^{K}, \mathcal{V}_\lambda/\varpi^m))_{\mathfrak{m}}.
		\end{align*} \end{linenomath*} 
		Corollary \ref{final_hecke_congruence_cor}
		gives a large enough $m' \geq m$ and a homomorphism
		\[
		\widetilde{A}(\widetilde{K}(m')) \to \prod_{q = 0}^{d - 1} \mathbb{T}^S(H^q(X_G^K, \mathcal{V}_\lambda/ \varpi^m)_{\mathfrak{m}})/ J_q
		\]
		for some nilpotent ideals $J_q$ with exponent of nilpotency only depending on $[F : \rationals]$ and $n$, 
		such that the projections with $ q \geq \lfloor d / 2 \rfloor$ commute with $\operatorname{Sat} : \widetilde{\mathbb{T}}^{S} \to \mathbb{T}^S$ and the projections with $q < \lfloor d / 2 \rfloor$ commute with $j \circ \operatorname{Sat}$. Thus, by composing with the map $j \times \dots \times  j \times \id \times \dots \times \id$ we obtain a map
		\[
		\Psi : \widetilde{A}(\widetilde{K}(m')) \to \prod_{q = 0}^{d - 1} \mathbb{T}^S(H^q(X_G^K, \mathcal{V}_\lambda/ \varpi^m)_{\mathfrak{m}})/ J_q',
		\]		
		commuting with $\operatorname{Sat}$, where $J_q' = j(J_q)$ for $q < \lfloor d / 2 \rfloor$ and $J_q' = J_q$ otherwise. Now \cite[Lemma 2.2.3]{10author} shows that the kernel $J_{-1}$ of the natural map
		\[
			A(m, \lambda) \to \prod_{q = 0}^{d - 1} \mathbb{T}^S(H^q(X_G^K, \mathcal{V}_\lambda/ \varpi^m)_{\mathfrak{m}})
		\]
		satisfies $J_{-1}^{d} = 0$. Let $I < A(m,\lambda)$ be the ideal generated by $J', J_{-1}$ and the preimages of the $J_q'$ for $0 \leq q \leq d - 1$. Then $\Psi$ factors through a surjection
		\[
		\widetilde{A}(\widetilde{K}(m')) \twoheadrightarrow A(m, \lambda)/I,
		\]
		commuting with $\operatorname{Sat}$. The degree of nilpotency of $I$ still only depends on $[F:\rationals]$ and $n$.
		
		The main theorem of \cite{caraiani_scholze_non_compact} implies that $\widetilde{A}(\widetilde{K}(m'))$ is $\mathcal{O}$-flat.
		 Moreover, \cite[Theorem 2.4.11]{10author} together with Clozel's purity lemma \cite[Lemma 4.9]{clozel1990} shows that $\widetilde{A}(\widetilde{K}(m')) \otimes_{\mathcal{O}} \overline{\rationals}_p$ is semisimple and can be computed in terms of cuspidal automorphic representations of $\widetilde{G}$. 
		Now  by \cite[Theorem 2.3.3]{10author} there exists a continuous representation
		\[
		\widetilde{\rho}_{m'} : G_{F,S} \to \GL_{2n}(\widetilde{A}(\widetilde{K}(m')) \otimes_{\mathcal{O}} \overline{\rationals}_p) 
		\]
		such that $\widetilde{\rho}_{m'}$ is semistable at $v$ and $v^c$ of Hodge--Tate weights
		\[
		HT_\tau(\widetilde{\rho}) = \{- \lambda_{\tau c,n} + 2n - 1, \dots, -\lambda_{\tau c, 1} + n, \lambda_{\tau, 1} + n - 1, \dots, \lambda_{\tau,n} \}
		\]
		Moreover, if $K_v = \GL_n(\mathcal{O}_{F_v})$ and $K_{v^c} = \GL_n(\mathcal{O}_{F_{v^c}})$, then we can choose $\widetilde{K}_{\overline{v}}$ to be hyperspecial as well and $\widetilde{\rho}_{m'}$ is crystalline at $v$ and $v^c$. 
		Let $r = \max_{\tau} \max_{\mu \in HT_\tau(\widetilde{\rho})} |\mu|$.
		
		Set $\widetilde{A} := \widetilde{A}(\widetilde{K}(m'))$ and let $\widetilde{D}$ be the continuous $\widetilde{A}$-valued determinant of $G_{F,S}$ attached to $\widetilde{\rho}_{m'}$ by \cite[2.32]{chenevier_determinant}.
		The formation of kernels commutes with flat base change, so $\ker(\widetilde{D}) \otimes_{\mathcal{O}} \overline{\rationals}_p \cong \ker(\widetilde{D} \otimes_{\mathcal{O}} \overline{\rationals}_p)$.
		Hence \cite[2.12]{chenevier_determinant} 
		implies that there is an algebra embedding
		\[
		(\widetilde{A}[G_{F,S}]/\ker(\widetilde{D})) \otimes_{\mathcal{O}} \overline{\rationals}_p = (\widetilde{A} \otimes_{\mathcal{O}} \overline{\rationals}_p)[G_{F,S}]/\ker(\widetilde{D} \otimes_{\mathcal{O}} \overline{\rationals}_p) \hookrightarrow M_{2n}(\widetilde{A} \otimes_{\mathcal{O}} \overline{\rationals}_p).
		\]
		Seen as a Galois representations, $\widetilde{A}[G_{F,S}]/\ker(\widetilde{D})$ is a subobject of $\widetilde{\rho}^{2n}_{m'} \cong M_{2n}(\widetilde{A} \otimes_{\mathcal{O}} \overline{\rationals}_p)$.
		By \cite[1.18 (iii)]{chenevier_determinant} there is a natural surjection
		\[
		\widetilde{A}[G_{F,S}]/\ker(\widetilde{D}) \to (A(m,\lambda)/I)[G_{F,S}]/\ker(\widetilde{D}_{A(m,\lambda)/I}).
		\]
		Hence $(A(m,\lambda)/I)[G_{F,S}]/\ker(\widetilde{D}_{A(m,\lambda)/I})$ is a subquotient of $\widetilde{\rho}^{2n}_{m'}$ and also torsion semi-stable at $v$ and $v^c$ with weights contained in $[-r,r]$.
		Now by Chebotarev density theorem and the computation of $\operatorname{Sat}( \widetilde{P}_{v})(X)$ right after \cite[5.3]{newton_thorne_16}
		we find
		\[
		\widetilde{D}_{A(m,\lambda)/I} = D( \rho_{\mathfrak{m}} \oplus \rho_{\mathfrak{m}}^{c,\vee}(1 - 2n) ).
		\]
		We obtain an induced surjective $A(m,\lambda)$-algebra morphism
		\[
		(A(m,\lambda)/I)[G_{F,S}]/\ker(\widetilde{D}) \to M_n(A(m,\lambda)/I)
		\]
		mapping $g \in G_{F,S}$ to $\rho_{\mathfrak{m}}(g)$.
		Hence $\rho_{\mathfrak{m}}$ is a quotient of $(A(m,\lambda)/I)[G_{F,S}]/\ker(\widetilde{D})$ and also torsion semistable (resp. crystalline) with weights contained in $[-r,r]$.
		
		Now we can remove the assumption that $\overline{\rho}_{\widetilde{\mathfrak{m}}}$
		is decomposed generic in exactly the same way as at the end of the proof of \cite[4.4.8]{10author}.
	\end{proof}
	
	\begin{theorem}[Wang-Erickson] \label{wang-erickson_ring_thm}
		Let $\mathbb{F}$ be a finite field of characteristic $p$ and $S$ a finite set of places of $F$ containing $S_p$. Let $\overline{D} : G_{F,S} \to \mathbb{F}$ be a multiplicity-free determinant of dimension $n$ and $S_p = S_{cris} \cup S_{ss}$ a partition and $\mu = (\mu_\tau)_{\tau : F \to \overline{\rationals}_p}$ a tuple of $n$-element multisets of integers. Then there exists a quotient $R^{cs, \mu}_{\overline{D}}$ of the pseudodeformation ring $R_{\overline{D}}$ such that for any finite extension $E/\rationals_p$, a homomorphism $f : R_{\overline{D}}$ factors through $R_{\overline{D}}^{cs, \mu}$ if and only if the semisimple representation corresponding to $f$ is semistable at all places above $p$ with Hodge--Tate weights $\mu$ and crystalline at the places in $S_{cris}$.
	\end{theorem}
	
	\begin{proof}
		This is proved in the same way as \cite[Theorem 7.9]{wang-erickson_algebraic_families}, except that we first consider the formal closed substacks 
		\[
		\mathcal{R}ep^{ss/cris,\mu_{v}}_{\overline{D} \rvert_{G_{F_v}}} \subset \mathcal{R}ep_{\overline{D} \rvert_{G_{F_v}}}
		\]
		from \cite[Theorem 6.7]{wang-erickson_algebraic_families} for $v \mid p$ and form the fibre product
		\[
			\begin{tikzcd}
				\mathcal{R}ep^{cs,\mu}_{\overline{D}} \arrow{r} \arrow{d} & \mathcal{R}ep_{\overline{D}} \arrow{d} \\
				\prod_{v \in S_{ss}} \mathcal{R}ep^{ss,\mu_{v}}_{\overline{D} \rvert_{G_{F_v}}} \times \prod_{v \in S_{cris}} \mathcal{R}ep^{cris,\mu_{v}}_{\overline{D} \rvert_{G_{F_v}}} \arrow{r} &   \prod_{v \mid p} \mathcal{R}ep_{\overline{D} \rvert_{G_{F_v}}}
			\end{tikzcd}
		\]
		The formal stack $\mathcal{R}ep_{\overline{D}} \to \Spf R_{\overline{D}}$ has an algebraization $\operatorname{Rep}_{\overline{D}} \to \Spec R_{\overline{D}}$, which is a good moduli space by \cite[Theorem 3.8]{wang-erickson_algebraic_families}. 
		By \cite[Theorem 3.16]{wang-erickson_algebraic_families} formal GAGA holds for the morphism $\operatorname{Rep}_{\overline{D}} \to \Spec R_{\overline{D}}$, 
		hence there exists a unique closed substack $\operatorname{Rep}_{\overline{D}}^{cs, \mu} \subset \operatorname{Rep}_{\overline{D}}$ whose formal completion along 
		$\mathfrak{m}_{R_{\overline{D}}}$ is $\mathcal{R}ep^{cs,\mu}_{\overline{D}}$. Let $\Spec R^{cs, \mu}_{\overline{D}}$ be the scheme theoretic image of this substack in $\Spec R_{\overline{D}}$. Now one can check that $R_{\overline{D}}^{cs, \mu}$ has the desired property in a similar way to \cite[Corollary 6.8 and Theorem 7.9]{wang-erickson_algebraic_families}.
	\end{proof}
	
	\begin{theorem} \label{derham_theorem}
Let $F$ be a CM field, $\iota : \complex \to \overline{\rationals_p}$ an isomorphism and $\Pi$ a cohomological cuspidal automorphic representation of $\GL_n(\mathbb{A}_F)$. If the reduction $\overline{r_\iota(\Pi)}$ is 
		absolutely irreducible and decomposed generic, then for every place $v \mid p$ of $F$ $r_{\iota}(\Pi) \rvert_{G_{F_v}}$ is potentially semistable (hence also de Rham) with Hodge--Tate weights
		\[
					HT_{\tau} := \{ \lambda_{\tau, 1} + n - 1, \lambda_{\tau,2} + n -2,  \dots, \lambda_{\tau,n}\},
		\]
		where $\lambda \in (\ints^n)^{\Hom(F, \overline{\rationals}_p)}$ is the unique dominant weight for $\GL_n/F$ such that the infinitesimal character of $\Pi_\infty$ coincides with the infinitesimal character of $(V_\lambda \otimes_{\mathcal{O}, \iota} \complex)^\vee$ (see definition \ref{highest_weight_module_def}).
		If $\Pi_v$ and $\Pi_{v^c}$ are unramified, then $r_{\iota}(\Pi) \rvert_{G_{F_v}}$ is crystalline.
	\end{theorem}
	
	\begin{proof} 
		Let $S$ be the union of the places where $\Pi$ is ramified and those lying above $p$ or $\infty$.
		Let $S_{cris}$ be the set of places $v \mid p$ such that $\Pi_v$ and $\Pi_{v^c}$ are unramified.
		We show that $r_\iota(\Pi)$ is semistable at all places above $p$ and crystalline at places above $S_{cris}$ when restricted to $G_{F'}$ for some finite extension $F'/F$ which is totally split at all places in $S_{cris}$. By solvable base change (see section 5.2 below) we can find an extension $F'/F$ which is totally split at all places in $S_{cris}$ such that for the level of the base change of $\Pi$ to $\GL_n(\mathbb{A}_{F'})$, the conditions \emph{2} and \emph{3} of theorem \ref{weak_semistable_local_global_thm} are satisfied and $S$ satisfies $(\Sigma_p)$. We replace $F$ by $F'$ and $\Pi$ by its base change. To ensure condition 
		\emph{1} we twist $\Pi$ by a suitable power of $\| \cdot \| \circ \det$ and $r_\iota(\Pi)$ by the corresponding power of the cyclotomic character.
		Then $\Pi$ is still cohomological cuspidal automorphic 
		and $\overline{r_\iota(\Pi)}$ is still absolutely irreducible and decomposed generic. It suffices to prove the theorem for this twisted $\Pi$ since a Galois representation is potentially semistable if and only if it has a Tate twist which is potentially semistable. By twisting with a suitable crystalline character we may also assume that $\overline{r_\iota(\Pi)} \not \cong \overline{r_\iota(\Pi)}^{c, \vee}(1 - 2n)$. 
		
		Now by the decomposition in \cite[\S 2.2]{franke_schwermer}
		there exists a good open compact subgroup $K < \GL_n(\mathbb{A}_F)$ such that the module $(\Pi^\infty)^K$
		is a $\mathbb{T}^S$-equivariant direct summand of $R\Gamma(X_{\GL_n/F}^K, \mathcal{V}_\lambda) \otimes_{\iota} \complex$, where $K_v = \GL_n(\mathcal{O}_{F,v})$ for $v \not \in S$. In particular there is a non-trivial map
		$\mathbb{T}^S_{\mathcal{O}} \to \End( (\Pi^\infty)^K)$ and since $\Pi$ is irreducible, its kernel is a prime ideal $\mathfrak{p}$. After enlarging $\mathcal{O}$, we can assume that $\mathbb{T}^S_{\mathcal{O}}/\mathfrak{p} \cong \mathcal{O}$. Now let $\mathfrak{m}$ be any maximal ideal of $\mathbb{T}^S_{\mathcal{O}}$, containing $\mathfrak{p}$. 
		The residual Galois representation associated with $\mathfrak{m}$ by \cite{scholze_torsion} is $\overline{r_\iota(\Pi)}$, hence condition 
		\emph{4} of theorem \ref{weak_semistable_local_global_thm} is satisfied. We obtain a sequence of Galois representations
		\[
		\rho_{j} :G_{F,S} \to  \GL_n(\mathbb{T}^S(R \Gamma(X_G^K, \mathcal{V}_\lambda/\varpi^j))_{\mathfrak{m}} /I_j)
		\]
		which are all torsion semistable with Hodge--Tate weights bounded independently of $j$
		and crystalline at all places in $S_{cris}$.
		Moreover, with theorem \ref{wang-erickson_ring_thm} we see that the proof of theorem \ref{weak_semistable_local_global_thm} implies that 
		the homomorphism $\alpha_j$ corresponding to the determinant $D = D(\rho_j \oplus \rho_j^{c, \vee}(1 - 2n))$ factors as follows
		\[
		\begin{tikzcd} 
		R_{\overline{D}} \arrow{r}{\alpha_j} \arrow{d} & \mathbb{T}^S(R \Gamma(X_G^K, \mathcal{V}_\lambda/\varpi^j))_{\mathfrak{m}} /I_j \\
		R_{\overline{D}}^{cs, \mu} \arrow[dashed]{ru}[swap]{\beta_j}
		\end{tikzcd}
		\] 
		where 
		\[
		\mu_\tau = \left\{ - \lambda_{\tau c,n} + 2n - 1, \dots, -\lambda_{\tau c, 1} + n, \lambda_{\tau, 1} + n - 1, \dots, \lambda_{\tau,n} \right\}.
		\]
		Consider the product of the $\beta_j$ 
		\[
			\beta : R^{cs,\mu}_{\overline{D}} \to \prod_{j}  \mathbb{T}^S(R \Gamma(X_G^K, \mathcal{V}_\lambda/\varpi^j))_{\mathfrak{m}} /I_j.
		\]
		By the computation of the characteristic polynomials of Frobenius elements we know that its image is contained in the image of the natural map
		\[
		\mathbb{T}^S(R \Gamma(X_G^K, \mathcal{V}_\lambda))_{\mathfrak{m}} \to \prod_{j}  \mathbb{T}^S(R \Gamma(X_G^K, \mathcal{V}_\lambda/\varpi^j))_{\mathfrak{m}} /I_j,
		\]
		which has a nilpotent kernel itself. Hence we get a map $R_{\overline{D}}^{cs, \mu} \to \mathbb{T}^S(R \Gamma(X_G^K, \mathcal{V}_\lambda))_{\mathfrak{m}}/I$ for some nilpotent ideal $I$.
		Quotienting out by $\mathfrak{p}$ we obtain the determinant attached to 
		\[
			\widetilde{\rho} = r_\iota(\Pi) \oplus r_\iota(\Pi)^{c,\vee}(1 - 2n).
		\]
		Now theorem \ref{wang-erickson_ring_thm} implies that $\widetilde{\rho}$ is semistable with Hodge--Tate weights $\mu$ and crystalline at all places in $S_{cris}$.
		It follows that $r_\iota(\Pi)$ is semistable with  Hodge--Tate weights a subset of $\mu$.
		Moreover, compatibility with central characters shows that
		\[
		HT_{\tau}( \det r_{\iota}(\Pi)) = \left \{ \sum_{i = 1}^n ( \lambda_{\tau, i} + n - i) \right\}
		\]
		By condition \emph{1} of theorem \ref{weak_semistable_local_global_thm} we have that 
		$ \sum_{i = 1}^n ( \lambda_{\tau, i} + n - i) < \sum_{i \in J} (\lambda_{\tau, i} + n - i)$ for every $n$ elements subset $J_{\tau} \subset \mu$ such that 
		$J \neq HT_{\tau}$. Thus, $HT_\tau(r_\iota(\Pi)) = HT_\tau$.
	\end{proof}

	\section{Vanishing of the Bloch-Kato Selmer Group}
	
	In this section we combine the results of the previous sections to prove the main theorem. More concretely, we check that the conditions of lemma \ref{patching_lemma} are satisfied. The rings $T_N$ will be special cases of the Hecke algebras discussed above. The rings $R_N$ will be Galois deformation rings which we introduce now.

	\subsection{Galois Deformations}
	
	We now fix notations for Galois deformation rings and recall the necessary results. The basic ideas go back to \cite{mazur_deformation} but we also require $p$-adic Hodge theory conditions studied in \cite{kisin_potentially_semistable} and \cite{torsion_p_adic} and the smoothness condition from \cite{allen_polarizable}. 
	Let $E/\rationals_p$ be a finite extension with ring of integers $\mathcal{O}$, uniformizer $\varpi$ and residue field $k = \mathcal{O}/\varpi \mathcal{O}$. Let $F$ be a number field and $S$ a finite set of places of $F$ containing $S_p$, the places above $p$. 
	We assume that $E$ contains all $p$-adic embeddings of $F$.
	
	Let $G_F$ be the absolute
	Galois group of $F$ and $n$ a positive integer. Fix continuous homomorphisms 
	$\rho : G_F \to \GL_n(\mathcal{O})$ and $\overline{\rho} : G_F \to \GL_n(k)$ which are unramified outside $S$ and so that
	$\rho \equiv \overline{\rho} \pmod{\varpi}$.
	We are interested in the following two types of questions:
	
	\begin{enumerate}[(A)]
		\item Given an Artinian local $\mathcal{O}$-algebra $A$ with an isomorphism $A/\mathfrak{m}_A \cong k$, what are the continuous (here $A$ has the discrete topology) homomorphisms $\rho_A : G_F \to \GL_n(A)$ such that $\rho_A \equiv \overline{\rho} \pmod{\mathfrak{m}_A}$?
		\item Given an Artinian local $E$-algebra $A$ with an isomorphism $A/\mathfrak{m}_A \cong E$, what are the continuous (here $A$ has the topology induced by the $E$-vector space structure on $A$) homomorphisms $\rho_A : G_F \to \GL_n(A)$ such that $\rho_A \equiv \rho \pmod{\mathfrak{m}_A}$?
	\end{enumerate}
	
	Question (A) asks about (framed) Galois deformations in the style of Mazur \cite{mazur_deformation}. It concerns integral questions and asks about congruences between Galois representations.
	Question (B) asks about (framed) Galois deformations in characteristic $0$ like 
	$\rho' : G_F \to \GL_n(E[t]/(t^m))$. Taking $A = E[t]/(t^2)$, this is equivalent to classifying
	all extensions $0 \to \rho \to \rho' \to \rho \to 0$, where $\rho'$ is a continous Galois representation $G_F \to \GL_{2n}(E)$. Ultimately, we will deduce the vanishing of the Bloch--Kato Selmer group by analysing a case of question (B).
	
	Let us make more precise definitions now. If $A$ is an Artinian local $\mathcal{O}$-algebra with an isomorphism $A/\mathfrak{m}_A \cong k$, then a continous homomorphism $\rho_A : G_F \to \GL_n(A)$ such that $\rho_A \equiv \overline{\rho} \pmod{\mathfrak{m}_A}$ is called a \emph{lift} of $\overline{\rho}$. Two such homomorphisms are \emph{strictly equivalent} if there is an element $\gamma \in \ker \left( \GL_n(A) \to \GL_n(k) \right)$ such that $\gamma \rho_A \gamma^{-1} = \rho_A'$. A strict equivalence class is called a \emph{deformation} of $\overline{\rho}$. We use the analogous terminology in question (B) for lifts and deformations of $\rho$ to Artinian local $E$-algebras $A$ with an isomorphism $A/\mathfrak{m}_A \cong E$. We let $\mathcal{D}$ denote the functor that sends $A$ to the set of all deformations of $\overline{\rho}$ over $A$ and $\mathcal{D}^\square$ the the functor that sends $A$ the set of all lifts of $\overline{\rho}$ over $A$. The analogous functors for question (B) are denoted by
	$\mathcal{D}_\rho$ and $\mathcal{D}_\rho^\square$.
	
	If $v$ is a place of $F$, a local lift at $v$ is a continuous homomorphism
	$\rho_A : G_{F_v} \to \GL_n(A)$. We denote the functor that sends $A$ to the set of local lifts of $\overline{\rho}$ over $A$ by $\mathcal{D}_v^\square$. It is well-known that $\mathcal{D}_v^\square$ is represented by a complete Noetherian local $\mathcal{O}$-algebra $R_v^\square$ with universal lift $\rho^{\square, univ} : G_{F_v} \to \GL_n(R_v^\square)$.
	
	\begin{definition}
		A global deformation problem is a tuple
		\[
		\mathcal{S} = (\overline{\rho}, S, \{\mathcal{D}_v\}_{v \in S}),
		\]	
		where $\overline \rho$ and $S$ are as above and each $\mathcal{D}_v$ is a subfunctor of $\mathcal{D}_v^\square$ which is represented by a quotient $R_v^\square \to R_v$
		and if $\rho_A \in \mathcal{D}_v(A)$ is strictly equivalent to $\rho_A' \in \mathcal{D}_v^\square$, then also $\rho_A' \in \mathcal{D}_v(A)$. For a subset $T \subset S$, we define $R_{\mathcal{S}}^{T,loc} := \widehat{\bigotimes}_{v \in T} R_v$.
	\end{definition} 
	
	\begin{definition}
		A deformation $\rho_A \in \mathcal{D}(A)$ (resp. lift) is of \emph{type} $\mathcal{S}$ if 
		one element (equivalently all elements) of the strict equivalence class of $\rho_A\rvert_{G_{F_v}}$ lie in $\mathcal{D}_v(A)$ for all $v \in S$. We note by $\mathcal{D}_{\mathcal{S}}$ the subfunctor of $\mathcal{D}$ of deformations of type $\mathcal{S}$. If $T \subset S$, then a $T$-\emph{framed lift of type } $\mathcal{S}$ is a tuple $(\rho_A, \{\alpha_v\}_{v \in T})$, where $\rho_A$ is a lift of $\overline{\rho}$ of type $\mathcal{S}$ and $\alpha_v$ are elements of $\ker\left(\GL_n(A) \to \GL_n(k)\right)$. Two $T$-framed lifts of type $\mathcal S$ are \emph{strictly equivalent}
		if there exists $\gamma \in \ker\left(\GL_n(A) \to \GL_n(k)\right)$ such that
		$\rho_A' = \gamma \rho_A \gamma^{-1}$ and $\alpha_v' = \gamma \alpha_v$. The functor of $T$-framed deformations of type $\mathcal S$, i.e. strict equivalence classes of $T$-framed lifts of type $\mathcal S$, is denoted by $\mathcal{D}_{\mathcal{S}}^T$.
	\end{definition}
	
	\begin{theorem} \label{deformation_representable_thm}
		If $\mathcal{S} = (\overline{\rho}, S, \{\mathcal{D}_v\}_{v \in S})$ is a global deformation problem, where $\overline{\rho}$ is absolutely irreducible and $T \subset S$, then $\mathcal{D}_{\mathcal{S}}$ and
		$\mathcal{D}_{\mathcal{S}}^T$ are representable by complete noetherian local $\mathcal{O}$-algebras $R_{\mathcal{S}}$ and $R_{\mathcal{S}}^T$ with residue fields isomorphic to $k$.	
	\end{theorem}
	
	\begin{proof}
		This is an application of Schlessinger's criterion as explained in \cite{mazur_deformation}.
	\end{proof}
	
	\begin{theorem} \label{deformation_ring_generic_fibre_thm}
		Let $\mathcal{S} = (\overline{\rho}, S, \{\mathcal{D}_v^{\square}\}_{v \in S})$, where $\overline{\rho}$ is absolutely irreducible and let $[\rho] : R_{\mathcal{S}} \to E$ be the homomorphism corresponding to $\rho$. Then  $\widehat{(R_{\mathcal{S}})_{[\rho]}}$
		pro-represents the functor
		$\mathcal{D}_{\mathcal{S},\rho}$
		which sends an Artinian local $E$-algebra $A$
		with residue field $E$ to
		the set of deformations of $\rho$,
		unramified outside $S$.	
		
		The tangent space of $\widehat{(R_{\mathcal{S}})_{[\rho]}}$
		is canonically isomorphic to $H^1(G_{F,S}, \ad \rho)$.
	\end{theorem}
	
	\begin{proof}
		See \cite[Lemma 2.3.3 and Proposition 2.3.5]{kisin_moduli_ff} for the identification of deformation rings. 
		The tangent space computation follows from a standard argument which identifies
		\[
		\mathcal{D}_{\mathcal{S}, \rho}(E[t]/(t^2)) \cong \Ext^1_{E[G_{F,S}]}(\rho, \rho) \cong \Ext^1_{E[G_{F,S}]}(E, \ad \rho) \cong H^1(G_{F,S}, \ad \rho),
		\]
		where the $\Ext$ groups are computed in the category of continuous $E[G_{F,S}]$-modules.
	\end{proof}
	
	\begin{lemma} \label{framing_variables_lemma}
		If $\mathcal{S} = (\overline{\rho}, S, \{\mathcal{D}_v\}_{v \in S})$ is a global deformation problem, where $\overline{\rho}$ is absolutely irreducible and $T$ a nonempty subset of $S$, then
		$R_{\mathcal{S}}^T \cong R_{\mathcal{S}}[[X_1, \dots, X_{n^2 |T| - 1}]]$.
	\end{lemma}
	
	\begin{proof}
		If $A$ is an Artinian local $\mathcal{O}$-algebra with residue field $k$, then 
		\[
		\mathcal{D}_{\mathcal{S}}^T(A) = \{ (\rho_A, (\alpha_v)_{v \in T} ) : \rho_A \in \mathcal{D}_{\mathcal{S}}^\square(A) \} / \sim.\]
		For each $v \in T$, we write
		$\alpha_v = \exp( \begin{bmatrix}
			X_{ij}^v
		\end{bmatrix}_{1 \leq i,j \leq n})$ for some $X_{ij}^v \in \mathfrak{m}_A$.
		Now since $\overline{\rho}$ is absolutely irreducible, the only endomorphisms
		of $\rho_A$ are scalar. Hence to compute the quotient by $\sim$ we only need to consider conjugation by elements $\gamma = \exp( Y \cdot \id )$ for $Y \in \mathfrak{m}_A$. Consequently,
		$\mathcal{D}_{\mathcal{S}}^T(A) \cong \mathcal{D}_{\mathcal{S}} \times \{ (\alpha_v)_{v \in T} \} / \sim$, where $(\alpha_v)_{v \in T} \sim (\alpha_v')_{v \in T}$ if and only if $X_{ij}^v = Y + (X_{ij}')^v$ for some $Y \in \mathfrak{m}_A$. 
		Now it is easy to count variables and we find $\mathcal{D}_{\mathcal{S}}^T(A) \cong \mathcal{D}_{\mathcal{S}}(A) \times \mathfrak{m}_A^{n^2 |T| - 1}$ and the result follows.
	\end{proof}
	
	\begin{theorem} 
		Assume that $\overline{\rho} \rvert_{G_{F_v}}$ is torsion semistable with Hodge--Tate weights contained in $[-r,r]$. Then there is a subfunctor $\mathcal{D}_v^{ss,r}$ of $\mathcal{D}_v^\square$ consisting of lifts which are torsion semistable with Hodge--Tate weights contained in $[-r,r]$. Moreover, $\mathcal{D}_v^{ss,r}$ is represented by a quotient $R_v^\square \to R_v^{ss,r}$ and two strictly equivalent lifts are torsion semistable if and only if one of them is. The same works with ``crystalline" in place of ``semistable". 
		Hence if $S_{cris} = \{v \mid p : \overline{\rho} \rvert_{G_{F_v}} \text{ is torsion crystalline.}\}$ and $S_{ss} = S_p \setminus S_{cris}$, then 
		\[
		\mathcal{S} = (\overline{\rho}, S, \{\mathcal{D}_v^{cris,r}\}_{v \in S_{cris}} \cup \{\mathcal{D}_v^{ss,r}\}_{v \in S_{ss}} \cup \{\mathcal{D}_v^\square\}_{v \in S \setminus S_p})
		\]
		is a well-defined global deformation problem.
	\end{theorem}
	
	\begin{proof}
		See the main theorem of \cite{torsion_p_adic}.
	\end{proof}
	
	\begin{theorem} \label{tangent_space_bloch_kato_thm}
		Let $\mathcal{S} = (\overline{\rho}, S, \{\mathcal{D}_v^{cris,r}\}_{v \in S_{cris}} \cup \{\mathcal{D}_v^{ss,r}\}_{v \in S_{ss}} \cup \{\mathcal{D}_v^\square\}_{v \in S \setminus S_p})$ be as in the previous proposition. Let $x : R_{\mathcal{S}} \to E$ be a homomorphism and let $\rho_x$ be the corresponding representation. Then $\widehat{(R_{\mathcal{S}})_{x}}$ pro-represents the functor $\mathcal{D}_{\mathcal{S},\rho_x}$ which sends
		an Artinian local $E$-algebra with residue field $E$ to the set of deformations of $\rho_x$ which are unramified outside $S$, semistable with Hodge--Tate weights equal to those of $\rho_x$ and crystalline at the places in $S_{cris}$.
		
		The tangent space of $\widehat{(R_{\mathcal{S}})_{x}}$
		is isomorphic to $H^1_g(G_{F,S}, \ad \rho_x)$.
	\end{theorem}
	
	\begin{proof}
		The interpretation of $\widehat{ (R_{\mathcal{S}})_x}$ follows from the main theorem of \cite{torsion_p_adic}.
		See \cite[Proposition 1.3.12]{allen_polarizable} for the tangent space.
	\end{proof}
	
	\begin{lemma} \label{generic_fibre_dim_bound}
		For each $v \nmid p$, we have $\dim R_v^\square[1/p] \leq n^2$. 
		Moreover, for each $v \mid p$ we have $\dim R_v^{ss,r}[1/p] \leq n^2 + \frac{n(n-1)}{2} [F_v : \rationals_p]$ and $\dim R_v^{cris,r}[1/p] \leq n^2 + \frac{n(n-1)}{2} [F_v : \rationals_p]$.
	\end{lemma}
	
	\begin{proof}
		See \cite[Proposition 1.2.2 and Theorem 1.2.4]{allen_polarizable} and \cite[Theorem 3.3.8]{kisin_potentially_semistable}.
	\end{proof}
	
	\begin{prop} \label{smooth_generic_fibre_prop}
		Let $\rho : G_F \to \GL_n(E)$ be a continuous representation, unramified outside $S$ such that 
		$\rho$ is generic \cite[Definition 1.1.2]{allen_polarizable} at all places $v \in S \setminus S_{cris}$.
		Then $(R_v)_{\mathfrak{p}}$ is regular for all $v \in S$, where $R_v$ are the local framed deformation rings
		of \[
		\mathcal{S} = (\overline{\rho}, S, \{\mathcal{D}_v^{cris,r}\}_{v \in S_{cris}} \cup \{\mathcal{D}_v^{ss,r}\}_{v \in S_{ss}} \cup \{\mathcal{D}_v^\square\}_{v \in S \setminus S_p})
		\] 
		and $\mathfrak{p}$ is the kernel of the homomorphism $R_v \to E$ induced by $\rho$.
	\end{prop}
	
	\begin{proof}
		This follows from theorem \ref{tangent_space_bloch_kato_thm} and \cite[Proposition 1.2.2 and Theorem 1.2.7]{allen_polarizable} together with \cite[Theorem 3.3.8]{kisin_potentially_semistable}. 
	\end{proof}
	
	\begin{definition}
		Let $\mathcal{S}$ be a global deformation datum, then a \emph{Taylor--Wiles datum} for $\mathcal{S}$ is a tuple $(Q, (\alpha_{v,1}, \dots, \alpha_{v,n})_{v \in Q})$, where
		$Q$ is a finite set of places of $F$ disjoint from $S$ such that
		\begin{itemize}
			\item for all $v \in Q$, $q_v \equiv 1 \pmod{p}$;
			\item $(\alpha_{v,1}, \dots, \alpha_{v,n})$ are the eigenvalues of $\overline{\rho}(\Frob_v)$, which are assumed to be pairwise distinct and $k$-rational.
		\end{itemize}	
	\end{definition}
	
	Given a Taylor--Wiles datum we define the augmented global deformation problem 
	\[
	\mathcal{S}_Q = (\overline{\rho}, S \cup Q, \{\mathcal{D}_v^{\square}\}_{v \in Q} \cup \{\mathcal{D}_v\}_{v \in S})
	\] 
	
	\begin{lemma} \label{taylor_wiles_def_ring_lemma}
		For a subset $T \subset S$ and Taylor--Wiles datum $(Q, (\alpha_{v,1}, \dots, \alpha_{v,n})_{v \in Q})$, the deformation ring $R_{\mathcal{S}_Q}^{T}$
		obtains a natural local $\mathcal{O}[\Delta_Q^n]$-algebra structure, where
		$\Delta_Q$ is the maximal $p$-quotient of $\prod_{v \in Q} k(v)^\times$
		and the maximal ideal of $\mathcal{O}[\Delta_Q^n]$ is the kernel of $\mathcal{O}[\Delta_Q^n] \to k : g \mapsto 1$. Moreover, the kernel of the natural map $R_{\mathcal{S}_Q}^{T} \to R_{\mathcal{S}}^T$, is generated by the augmentation ideal of $\mathcal{O}[\Delta_Q^n]$. 
	\end{lemma}
	
	\begin{proof}
		This follows from \cite[Lemma 6.2.19]{10author} and the discussion afterwards and goes back to \cite[Lemma 2.44]{darmon_diamond_taylor}.
	\end{proof}
	
	\subsection{Base Change}
	
	Ultimately, we wish to show that a certain Bloch--Kato Selmer group attached to a representation of $G_F$ vanishes. It turns out that to do so one can first restrict to $G_L$ for any finite extension $L/F$. 
	
	\begin{lemma} \label{BK_basechange_lemma}
		Let $L/F$ be a Galois extension of number fields unramified outside a finite set of places $S$ of $F$ containing all places above $p$
		and $\Sigma$ be the set of places of $L$ above $S$. If $\rho : G_{F,S} \to \GL_n(E)$ is a continuous Galois representation, then $H^1_g(G_{L,\Sigma}, \ad \rho) = 0$ implies that
		$H^1_g(G_{F,S}, \ad \rho) = 0$.
	\end{lemma}
	
	\begin{proof}
		From the inflation-restriction sequence we have a commutative diagram
		\[
		\begin{tikzcd}[column sep = small, row sep = small]
			H^1(L/F,(\ad \rho)^{G_L}) \arrow{r} & H^1(G_{F,S}, \ad \rho) \arrow{r} \arrow{d} & H^1(G_{L,\Sigma}, \ad \rho) \arrow{d} \\
			&
			\prod_{v \mid p} H^1(F_v, B_{dR} \otimes_{\rationals_p} \ad \rho) \arrow{r}
			&  \prod_{v \mid p} H^1(L_v, B_{dR} \otimes_{\rationals_p} \ad \rho)
		\end{tikzcd}
		\]
		where for each $v$ we make a choice of place $w$ of $L$ above $v$ and set $L_v = L_w$. In group cohomology these choices do not matter since these places are all conjugate. Now the lemma will follow if we can show $H^1(L/F, (\ad \rho)^{G_L}) = 0$ but this is immediate from corestriction-restriction since $L/F$ is a finite extension and $(\ad \rho)^{G_L}$ is divisible.
	\end{proof}
	
	This lemma allows us to liberally use automorphic base change \cite{arthur_clozel} as a reduction step in the proof of the main theorem. Paired with class field theory \cite{artin_tate} it becomes a very powerful method. Here we state a few such results which we will use below.
	
	\begin{prop} \label{solvable_basechange_prop}
		Let $F$ be a number field, $L/F$ a solvable Galois extension and $\pi$ be a cohomological cuspidal automorphic representation of $\GL_n(\mathbb{A}_F)$. If the Galois representation $r_\iota(\pi) \rvert_{G_L}$ is absolutely irreducible, then there exists a unique cohomological cuspidal automorphic representation $\Pi$ of $\GL_n(\mathbb{A}_L)$ such that
		\[
		r_\iota(\Pi) \cong r_\iota(\pi) \rvert_{G_L}.
		\]
	\end{prop} 
	
	\begin{proof}
		By induction it suffices to treat the case when $L/F$ is cyclic of prime degree $l$. By class field theory there is a finite order Hecke character $\eta : F^\times \backslash \mathbb{A}_F^\times \to \complex^\times$ which cuts out $L$, i.e. satisfies $\ker (\eta \circ \operatorname{Art}_F^{-1}) = G_L$. If we can show that $\pi \otimes (\eta \circ \det) \not \cong \pi$, then \cite[Chapter 3, Theorem 4.2 (a)]{arthur_clozel} will imply the proposition by the Chebotarev density theorem and the explicit description of Hecke eigenvalues in unramified local base change \cite[Chapter 1, \S4.2]{arthur_clozel}.
		
		In particular, it also suffices to show that $r_\iota(\pi \otimes (\eta \circ \det)) \not \cong r_\iota(\pi)$.
		By the Chebotarev density theorem we find that
		\[
		r_\iota(\pi \otimes (\eta \circ \det)) \cong r_\iota(\pi) \otimes (\eta^{-1} \circ \operatorname{Art}_F^{-1})
		\]
		But if $r_\iota(\pi) \otimes (\eta^{-1} \circ \operatorname{Art}_F^{-1}) \cong r_\iota(\pi)$, then the intertwining operator realising such an isomorphism is a non-scalar $G_L$-equivariant endomorphism of $r_\iota(\pi) \rvert_{G_L}$. Otherwise, there would exist $\alpha \in \overline{\rationals_p}^\times$ and $g \in G_F$ such that $\eta(\operatorname{Art}^{-1}_F(g)) \neq 1$ and $\alpha r(g) v = \eta(\operatorname{Art}^{-1}_F(g)) r(g) \alpha v$ for all vectors $v$ of $r_\iota(\pi)$. By Schur's lemma this contradicts the assumption that $r_\iota(\pi) \rvert_{G_L}$ is absolutely irreducible.
	\end{proof}
	
	\begin{prop} \label{iwahori_basechange_prop}
		Let $F/ \rationals_\ell$ be a finite extension and $\pi$ a irreducible smooth representation of $\GL_n(F)$. Then there exists a finite solvable extension $F'/F$ 
		such that the base change $\Pi$ of $\pi$ to $\GL_n(F')$ exists and has Iwahori fixed vectors.
	\end{prop}
	
	\begin{proof}
		We use the functorial properties of the local Langlands correspondence \cite{harris_taylor}.  The precise statement we need is conveniently stated in \cite{scholze_LLC}. 
		Let $\pi \mapsto \sigma(\pi)$ be the normalized Weil group representation (without the monodromy) attached to a smooth representation $\pi$ of $\GL_n(F)$. By construction of the topology on the Weil group there exists a finite Galois extension $F'/F$ such that $\sigma(\pi) \rvert_{W_{F'}}$ is unramified. Since $G_F$ is pro-solvable, any such $F'/F$ is solvable. By \cite[Proposition 3.19]{bernstein_zelevinsky}, $\pi$ can be embedded into a normalized parabolic induction $i_P^{\GL_n(F)}(\pi')$ for some supercuspidal representation $\pi'$ of $M$, where $P = MN$ is a standard parabolic subgroup of $\GL_n(F)$.
		The Weil representation $\sigma(\pi)$ is determined by $\pi'$ \cite[Theorem 12.1 (ii)]{scholze_LLC} and we can inductively apply \cite[Theorem 12.1 (iv)]{scholze_LLC} to $\pi'$, to obtain the base change $\Pi$ of $\pi$ satisfying
		\[
		\sigma(\Pi) = \sigma(\pi) \rvert_{W_{F'}}.
		\]
		Now \cite[Theorem 12.1 (v)]{scholze_LLC} shows that $\Pi$ has Iwahori fixed vectors.
	\end{proof}
	
	\begin{prop} \label{cft_existence_prop}
		Let $F$ be a number field, $p$ a prime such that $\zeta_p \not \in F$, $S$ a finite set of places of $F$ and $\phi : G_F \to H$ a continuous homomorphism to a finite group $H$. Suppose that for each $v \in S$ we are given a finite Galois extension $F'_v / F_v$, then there exists a solvable extension $L/F$ such that 
		\begin{itemize}
			\item $\zeta_p \not \in L$;
			\item $\phi(G_{L(\zeta_p)}) = \phi(G_{F(\zeta_p)})$;
			\item $\phi(G_L) = \phi(G_F)$;
			\item For each $v \in S$ and place $w \mid v$ of $L$ we have $L_w \cong F'_v$. 
		\end{itemize}
	\end{prop}
	
	\begin{proof}
		By induction it suffices to treat the case when all the $F'_v/ F_v$ are cyclic.
		So let $\chi_v : F_v^\times \to \complex^\times$ be continuous local characters cutting out the $F_v'$, i.e. such that $\ker( \chi_v \circ \operatorname{Art}_{F_v}^{-1}) = G_{F_{v'}}$. 
		Moreover, by the Chebotarev density theorem, there exists a finite set of places $T'$ of $F$ such that $\{ \phi(\Frob_v) : v \text{ lies above } T' \}$ generates $\phi(G_{F(\zeta_p)})$. Moreover, there exists a place $v_0$ of $F$ such that 
		$\Frob_{v_0}$ generates $\Gal(F(\zeta_p)/F)$. 
		Let $T = T' \cup \{ v_0\}$ and $\chi_v = 1$ for $v \in T$.
		By \cite[Chapter X, Theorem 5]{artin_tate}, there exists a global character $\chi : F^\times \backslash \mathbb{A}_F^\times \to \complex^\times$ with local components equal to $\chi_v$ at all $v \in S \cup T$. Now the cyclic Galois extension $L/F$ cut out by $\chi$ satisfies 
		\begin{itemize}
			\item $\zeta_p \not \in L$ since $\zeta_p \not \in F_{v_0}$;
			\item Every place $v \in T$ is totally split in $L$;
			\item For each $v \in S$ and place $w \mid v$ of $L$ we have $L_w \cong F_v'$.
		\end{itemize}
		In particular the second property implies that for $v \in T$, the conjugacy class $\Frob_v \subset G_F$ is equal to the union of the conjugacy classes $\Frob_w \subset G_L$, for $w \mid v$. Now by construction of $T'$ we find that $\phi(G_{L(\zeta_p)}) = \phi(G_{F(\zeta_p)})$. Since $v_0 \in T$ is totally split and $\Frob_{v_0}$ generates $\Gal(F(\zeta_p)/F)$ we also find that
		$\phi(G_L) = \phi(G_F)$.
	\end{proof}
	
	\subsection{Main Theorem} We collect some final preliminary results on the existence of Taylor--Wiles primes and the existence of Hecke algebra valued Galois representations before finally proving the main theorem.
	
	\begin{definition}
		A subgroup $H \subset \GL_n(k)$ is called \emph{enormous} over $k$ (the fixed finite residue field) if it satisfies the following
		\begin{enumerate}
			\item The representation of $H$ acting on $k^n$, given by the inclusion $H \hookrightarrow \GL_n(k)$ is absolutely irreducible.
			\item $H$ has no non-trivial quotients of $p$-power order.
			\item $H^0(H,\ad^0) = H^1(H, \ad^0) = 0$, where $\ad^0$ is the vector space of traceless $n \times n$ matrices over $k$ on which $\GL_n(k)$ acts by conjugation.
			\item For any simple $k[H]$-submodule $M \subset \ad^0$, there is a regular semisimple $h \in H$ such that $M^h \neq 0$.
		\end{enumerate}
	\end{definition}
	
	\begin{prop} \label{taylor_wiles_existence}
		Let $\mathcal{S} = (\overline{\rho}, S, \{\mathcal{D}_v\}_{v \in S})$ be a global deformation datum. Assume that 
		\begin{itemize}
			\item $p \nmid 2n$,
			\item $\overline{\rho}$ is absolutely irreducible,
			\item $F = F^+ F_0$ with $F^+$ totally real and $F_0$ an imaginary quadratic field,
			\item $\zeta_p \not \in F$ and $\overline{\rho}(G_{F(\zeta_p)})$ is enormous.
		\end{itemize}
		Then for a large enough integer $q$ and any $N \geq 1$, there exists a Taylor--Wiles datum $(Q_N, (\alpha_{v,1}, \dots, \alpha_{v,n})_{v \in Q_N})$ satisfying
		\begin{enumerate}
			\item $\# Q_N = q$.
			\item For each $v \in Q_N$, $q_v \equiv 1 \pmod {p^N}$ and the rational prime below $v$ splits  in $F_0$.
			\item There is a local $\mathcal{O}$-algebra surjection $R_{\mathcal{S}}^{S,loc}[[X_1, \dots, X_g]] \to R_{\mathcal{S}_{Q_N}}^{S}$, with 
			\[
			g = qn - n^2 [F^+ : \rationals]
			\]
		\end{enumerate}
	\end{prop}
	
	\begin{proof}
		This is \cite[Proposition 6.2.32]{10author}.	
	\end{proof}

	\begin{prop} \label{complexes_hecke_ideals_prop}
		Let $\mathfrak{m} < \mathbb{T}^S$ be a non-Eisenstein ideal and $\lambda$ a dominant weight of $\GL_n /F$. Suppose $K$ is a good open compact subgroup $K < \GL_n(\mathbb{A}_F^\infty)$ such that $K_v = \GL_n(\mathcal{O}_{F_v})$ for all $v \not \in S$ for some finite set of places $S$ containing those above $p$.
		Let $(Q,(\alpha_{v,1}, \dots, \alpha_{v,n})_{v \in Q})$ be a Taylor--Wiles datum for 
		\[
		(\overline{\rho}_{\mathfrak{m}}, S, \{\mathcal{D}_v^{cris,r}\}_{v \in S_{cris}} \cup \{\mathcal{D}_v^{ss,r}\}_{v \in S_{ss}} \cup \{\mathcal{D}_v^\square\}_{v \in S \setminus S_p})
		\]
		We define level subgroups 
		$K_1(Q) = \prod_{v} K_1(Q)_v \subset K_0(Q) = \prod_{v} K_0(Q)_v \subset K$ by
		$K_1(Q) = K_0(Q)_v = K_v$ for $v \not\in Q$ and for $v \in Q$ we set 
		$K_0(Q) = \operatorname{Iw}_v$ and
		$K_1(Q)_v$ the kernel of the maximal abelian $p$-quotient of $\operatorname{Iw}_v$ for $v \in Q$. We have $K_0(Q)/K_1(Q) \cong \Delta_Q^n$.
		
		Conjugation by elements of $K_0(Q)/K_1(Q)$ induces an embedding 
		\[
		\mathbb{T}_{\mathcal{O}}^{S \cup Q}[\Delta_Q^n] \to \mathcal{H}(\GL_n(\mathbb{A}_F^{S}),K_1(Q)) \otimes_\ints \mathcal{O}
		\]
		and given an object $M$ with Hecke action we write we write 
		$\mathbb{T}^{S \cup Q}_Q(M)$ for the image of this algebra in the endomorphism ring of $M$.
		With these notations, there are maximal ideals $\mathfrak{m}_1$ and $\mathfrak{m}_0$ and local $\mathcal O$-algebra homomorphisms $\phi,\psi$
		\[
		\mathbb{T}^{S \cup Q}_Q(R \Gamma(X^{K_1(Q)}_{\GL_n/F}, \mathcal{V}_\lambda))_{\mathfrak{m}_1} \xrightarrow{\phi} \mathbb{T}^{S \cup Q}(R\Gamma(X^{K_0(Q)}_{\GL_n/F}, \mathcal{V}_\lambda))_{\mathfrak{m}_0} \xrightarrow{\psi} \mathbb{T}^S(R\Gamma(X^{K}_{\GL_n/F}, \mathcal{V}_\lambda))_{\mathfrak{m}}
		\]
		such that $\phi$ is surjective with kernel generated by the augmentation ideal of $\mathcal{O}[\Delta_Q^n]$ and $\psi$ is an isomorphism.
		Moreover, there are corresponding maps
		\[
		R \Gamma(X^{K_1(Q)}_{\GL_n/F},\mathcal{V}_\lambda)_{\mathfrak{m}_1} \xrightarrow{f} R \Gamma(X^{K_0(Q)}_{\GL_n/F}, \mathcal{V}_\lambda)_{\mathfrak{m}_0} \xrightarrow{g} R \Gamma(X^K_{\GL_n/F}, \mathcal{V}_\lambda)_{\mathfrak{m}},
		\]
		where $g$ is a quasi-isomorphism and $f$ induces a quasi-isomorphism 
		\[ 
		R\Gamma(\Delta_Q^n, R\Gamma(X^{K_1(Q)}_{\GL_n/F}, \mathcal{V}_\lambda)_{\mathfrak{m}_1}) \to R \Gamma(X^{K_0(Q)}_{\GL_n/F}, \mathcal{V}_\lambda)_{\mathfrak{m}_0}.
		\]
	\end{prop}
	
	\begin{proof}
		This works in exactly the same way as \cite[Lemma 6.5.9]{10author}.
	\end{proof}

	\begin{theorem} \label{galois_rep_existence_thm}
		Let $\mathfrak{m} < \mathbb{T}^S$ be a non-Eisenstein ideal and $\lambda$ a dominant weight of $\GL_n /F$. Suppose $K$ is a good open compact subgroup $K < \GL_n(\mathbb{A}_F^\infty)$ such that $K_v = \GL_n(\mathcal{O}_{F_v})$ for all $v \not \in S \cup R$ for some finite set of places $S \cup R$ satisfying $(\Sigma_p)$. Moreover, let $S_p = S_{cris} \cup S_{ss} \subset S$ be a partition, stable under complex conjugation such that $K_v = \GL_n(\mathcal{O}_{F_v})$ for $v \in S_{cris}$ and $K_v = Iw_v$ for $v \in S_{ss}$ and $K_v = Iw_{v,1}$ for $v \in R$.
		Consider the global deformation datum 
		\[
		\mathcal{S} = 	(\overline{\rho}_{\mathfrak{m}}, S \cup R, \{\mathcal{D}_v^{cris,r}\}_{v \in S_{cris}} \cup \{\mathcal{D}_v^{ss,r}\}_{v \in S_{ss}} \cup \{\mathcal{D}_v^\square\}_{v \in R \cup S \setminus S_p})
		\] and
		let $(Q,(\alpha_{v,1}, \dots, \alpha_{v,n})_{v \in Q})$ be a Taylor--Wiles datum.  Suppose that the following are satisfied
		\begin{enumerate}[(1)]
			\item For each embedding $\tau : F \hookrightarrow E$ inducing a place above $p$, we have $- \lambda_{\tau c, 1} - \lambda_{\tau,1} \geq 0$;
			\item For each place $v \mid p$ of $F$ let $\overline{v}$ be the place of $F^+$ lying below $v$. Then there exists a place $\overline{v}' \neq \overline{v}$ of $F^+$ such that $\overline{v}' \mid p$ and
			\[
			\sum_{\overline{v}'' \neq \overline{v}, \overline{v}'} [F_{\overline{v}''}^+ : \rationals_p] > \frac{1}{2} [F^+ : \rationals];
			\]
			\item The residual representation $\overline{\rho}_{\mathfrak{m}}$ is absolutely irreducible and decomposed generic \cite[Definition 4.3.1]{10author},
		\end{enumerate}
		then there is an integer $\delta \geq 1$ depending only on $[F^+ : \rationals]$ and $n$, a nilpotent ideal $I \subset \mathbb{T}^{S \cup R \cup  Q}_Q(R\Gamma(X^{K_1(Q)}_{\GL_n/F}, \mathcal{V}_\lambda))_{\mathfrak{m}_1}$ such that
		$I^\delta = 0$ and a surjective local $\mathcal{O}[\Delta_Q]$-algebra homomorphism
		\[
		x : R_{\mathcal{S}_Q} \to \mathbb{T}^{S\cup R \cup Q}_Q( R\Gamma(X^{K_1(Q)}_{\GL_n/F}, \mathcal{V}_\lambda))_{\mathfrak{m}_1}/I
		\]
		such that
		\[
		\det(z \cdot \id - \rho_x(\Frob_v) ) = P_v(z) \in (\mathbb{T}^{S\cup Q}_Q( R\Gamma(X^{K_1(Q)}_{\GL_n/F}, \mathcal{V}_\lambda))_{\mathfrak{m}_1}/I)[z]
		\]
		for $v \not \in S \cup R \cup Q$.
	\end{theorem}
	
	\begin{proof}
		When $Q = \emptyset$, this follows from theorem \ref{weak_semistable_local_global_thm}. In general the $\Delta_Q$-equivariance follows from of \cite[Theorem 3.1.1]{10author}.
	\end{proof}

	\begin{theorem} \label{main_theorem}
		Let $F \subset L$ be CM fields and let
		$\rho : G_F \to \GL_n(E)$ be a continuous representation,
		where $E \subset \overline{\rationals}_p$ is a finite extension of $\rationals_p$
		containing the images of all field homomorphisms $L \to \overline{\rationals}_p$. Moreover, let $\Pi$ be a cohomological  cuspidal
		automorphic representation of $\GL_n(\mathbb{A}_L)$ and let $S$ be a finite set of places of $L$, stable under complex conjugation, containing the archimedean ones and those above $p$ such that $\Pi_v$ is unramified for all $v \not \in S$. Let $S_p = S_{cris} \cup S_{ss}$ be a partition of the places of $L$ lying above $p$ which is stable under complex conjugation such that the following are satisfied.
		\begin{enumerate}[(a)]
			\item $p \nmid 2n$;
			\item There exists an isomorphism $\iota : \overline{\rationals}_p \to \complex$ such that $\rho \rvert_{G_L} \otimes_{E} \overline{\rationals}_p \cong r_\iota(\Pi)$;
			\item The residual representation $\overline{\rho} \rvert_{G_L}$ is absolutely irreducible and decomposed generic \cite[Definition 4.3.1]{10author}. Moreover, $\zeta_p \not \in L$ and $\overline{\rho} \rvert_{G_{L(\zeta_p)}}$ has enormous image \cite[Definition 6.2.28]{10author}, where $\zeta_p \in \overline{L}$ is a primitive $p$th root of unity;
			\item If $v \in S \setminus S_p$, then the Weil--Deligne representation $\operatorname{WD}(\rho \rvert_{G_{L_v}})$ is generic
			\cite[Definition 1.1.2]{allen_polarizable};
			\item If $v \in S_{cris}$, then $\Pi_v$ is unramified. If $v \in S_{ss}$, then
			$\rho \rvert_{G_{L_v}}$ is de Rham and for any finite extension $L_v'/L_v$, $\operatorname{WD}(\rho \rvert_{G_{L'_v}})$ is generic.
		\end{enumerate}
		Then the Bloch--Kato Selmer group $H^1_g(G_{F,S}, \ad \rho )$ vanishes, i.e. $\rho$ is rigid. Moreover, 
		$\rho \rvert_{G_{L_v}}$ is crystalline for all $v \in S_{cris}$.
	\end{theorem}

	\begin{proof}
		We immediately note that we can assume $L = F$ since it suffices to prove the Selmer vanishing over $L$ by lemma \ref{BK_basechange_lemma}.
		We now apply proposition \ref{solvable_basechange_prop} multiple times and enlarge $F$ further, while keeping the places $v \in S \setminus S_{ss}$ totally split over the field $L$ that we started with.
		
		By proposition \ref{iwahori_basechange_prop} we can find solvable local extensions
		$F_v'/F_v$ for all $v \in S_{ss}$ such that the local base change of $
		\Pi_v$ to $\GL_n(F_v')$ has Iwahori fixed vectors.
		Using proposition \ref{cft_existence_prop} there is a solvable global extension which interpolates these, is totally split at all $v \in S \setminus S_{ss}$ and preserves condition \emph{(c)}. Hence by proposition \ref{solvable_basechange_prop} and condition \emph{(e)} we may assume that 
		\begin{enumerate}[(1)]
			\item If $v \in S_{ss}$, then $\operatorname{WD}(\rho \rvert_{G_{F_v}})$ is generic and $\Pi_v^{Iw_v} \neq 0$.
		\end{enumerate}
		by choosing more local extensions at arbitrary places outside of $S$ we can moreover assume that $S$ satisfies $(\Sigma_p)$ and
		\begin{enumerate}[(1)]
			\setcounter{enumi}{1}
			\item For each place $v \mid p$ of $F$ let $\overline{v}$ be the place of $F^+$ lying below $v$. Then there exists a place $\overline{v}' \neq \overline{v}$ of $F^+$ such that $\overline{v}' \mid p$ and
			\[
			\sum_{\overline{v}'' \neq \overline{v}, \overline{v}'} [F_{\overline{v}''}^+ : \rationals_p] > \frac{1}{2} [F^+ : \rationals].
			\]
		\end{enumerate}
		Moreover, we can twist $\Pi$ by a suitable power of $\| \cdot \| \circ \det$  to arrange for
		the following condition to be satisfied.
		\begin{enumerate}[(1)]
			\setcounter{enumi}{2}
			\item If $c$ is the complex conjugation of $F$, then
			\[
			- \lambda_{\tau c,1} - \lambda_{\tau, n} \geq 0
			\]
			for all $\tau \in \Hom(F, \complex)$.
		\end{enumerate}
		Then $r_\iota(\Pi)$ is twisted by a power of the cyclotomic character by which $\ad r_\iota(\Pi)$ is unaffected and by possibly enlarging the power of the character we can make sure that $\overline{r_\iota(\Pi)}$ remains the same.
		Hence all the above conditions are still satisfied.
		
		We apply proposition \ref{taylor_wiles_existence} and let $R$ be a large enough set of Taylor--Wiles primes for the deformation datum
		\[
		\mathcal{S} = (\overline{\rho}_\mathfrak{m}, S, \{\mathcal{D}_v^{ss,r}\}_{v \in S_{ss}} \cup \{\mathcal{D}_v^{cris,r}\}_{v \in S_{cris}} \cup \{\mathcal{D}_v^\square\}_{v \in S \setminus S_p} ).
		\]
		In particular, for each $v \in R$, $\operatorname{WD}(\rho \rvert_{F_v})$ is generic.
		We define the level subgroup $K = \prod_{v} K_v$ of $\GL_n(\mathbb{A}_F^\infty)$ by 
		\begin{itemize}
			\item 	$K_v = \GL_n(\mathcal{O}_{F_v})$ for $v \not \in R \cup S \setminus S_{cris}$,
			\item $K_v = Iw_v$ for $v \in S_{ss}$,
			\item $K_v$ is any small enough open compact subgroup so that $\Pi_v^{K_v} \neq 0$ for $v \in S \setminus S_p$,
			\item $K_v = Iw_{v,1}$ is 
			the group subgroup of matrices which are unipotent and upper triangular mod $\varpi_v$ for $v \in R$.
		\end{itemize}
		Since $R$ is sufficiently large, it contains a place $v$ which does not divide any of the numbers $\Phi_j(1)$ for $1 < j \leq n [F: \rationals] + 1$, where $\Phi_j$ denotes the $j$th cyclotomic polynomial. Now lemma \ref{neat_lemma} shows that $K$ is neat.

		It is proven in \cite[Theorem 2.4.10]{10author} that there exists a maximal ideal $\mathfrak{m} \subset \mathbb{T}^{S \cup R}_{\mathcal{O}}(R\Gamma(X_{\GL_n/F}^K, \mathcal{V}_\lambda))$ such that $\overline{\rho}_\mathfrak{m} \cong \overline{r_\iota(\Pi)}$,
		where $\mathcal{O}$ is the ring of integers of $E$.
		Now theorem \ref{weak_semistable_local_global_thm} shows that $\rho$ is semistable at $S_{ss}$ and crystalline at $S_{cris}$ with Hodge--Tate weights bounded by some large enough $r$.
		
		Consider the global deformation problem
		\[
		\mathcal{S} = (\overline{\rho}_\mathfrak{m}, S \cup R, \{\mathcal{D}_v^{ss,r}\}_{v \in S_{ss}} \cup \{\mathcal{D}_v^{cris,r}\}_{v \in S_{cris}} \cup \{\mathcal{D}_v^\square\}_{v \in R \cup S \setminus S_p} )
		\]
		The functor $\mathcal{D}_\mathcal{S}$ is represented by a complete local $\mathcal O$-algebra $R_{\mathcal{S}}$ by theorem \ref{deformation_representable_thm}. Let $\mathfrak{p}$ be
		the kernel of $R_{\mathcal{S}} \to E$ induced by $\rho = r_{\iota}(\Pi)$. Then theorem \ref{tangent_space_bloch_kato_thm} shows that the tangent space of $(R_{\mathcal{S}})_{\mathfrak{p}}$ is isomorphic to $H^1_g(G_{F,S}, \ad \rho)$. Set 
		\[
			C_0 = R\Hom_{\mathcal{O}}(R\Gamma(X_{\GL_n/F}^K,\mathcal{V}_\lambda)_\mathfrak{m}, \mathcal{O})
		\] 
		and $T_0 = \mathbb{T}^S_{\mathcal{O}}(C_0)$.
		
		By theorem \ref{galois_rep_existence_thm}, there
		is a surjection $R_{\mathcal{S}} \to T_0/I_0$
		of local $\mathcal{O}$-algebras, where $I_0 < T_0$ is the nilradical.
		Theorem \ref{hecke_product_of_fields} implies that $T_0[1/p]$ is a product of fields
		and thus the localisation at $\mathfrak{p}$ is a surjection $(R_{\mathcal{S}})_{\mathfrak{p}} \to E'$
		for some field $E'$.
		If this map is also injective, then the tangent space $H^1_g(G_{F,S}, \ad \rho)$ vanishes.
		Thus, we are left with applying lemma \ref{patching_lemma}.
		
		Let $q$ be large enough
		and $g = qn - n^2 [F^+ : \rationals]$.
		For each $N \geq 1$, use proposition \ref{taylor_wiles_existence} to choose a Taylor--Wiles datum $(Q_N, (\alpha_{v,1}, \dots, \alpha_{v,n})_{v \in Q_N})$
		satisfying
		\begin{enumerate}[(a)]
			\item $\# Q_N = q$.
			\item For each $v \in Q_N$, $q_v \equiv 1 \pmod{p^N}$ and the rational prime below $v$ splits in some imaginary quadratic subfield of $F$.
			\item There is a surjection of local $\mathcal{O}$-algebras $R^{S \cup R,loc}_\mathcal{S}[[X_1, \dots, X_g]] \to R_{\mathcal S_{Q_N}}^{S \cup R}$
		\end{enumerate}
		
		Let $\mathcal{T} = \mathcal{O}[[X_1, \dots, X_{n^2 |S \cup R| - 1}]]$ and $S_\infty = \mathcal{T}[[Y_1, \dots, Y_{nq}]]$ with augmentation ideal $\mathfrak{a} < S_\infty$ given by
		\[
		\mathfrak{a} := (X_1, \dots, X_{n^2 |S \cup R| - 1}, Y_1, \dots, Y_{nq}).
		\]
		For $N \geq 1$ we define $\mathfrak{a}_N := ((X_i + 1)^{p^N} - 1, (Y_{i} + 1)^{p^N} - 1)$.
		To apply the main lemma (\ref{patching_lemma}) there are seven things to verify/define so we do the rounds.
		\begin{enumerate}[(1)]
			\item Let $R_\infty = R_\mathcal{S}^{S \cup R,loc} [[X_1, \dots, X_g]]$. It is a complete local $\mathcal{O}$-algebra since $R_\mathcal{S}^{S \cup R,loc}$ is.
			
			\item We put $R_0 = R_{\mathcal{S}}$ and for $N \geq 1$ we set $R_N = R_{\mathcal{S}_{Q_N}}^{S \cup R}/\mathfrak{a}_{N}$. By assumption on the Taylor--Wiles datum $Q_N$, each $R_N$ is a quotient of $R_\infty$. Using lemmas \ref{taylor_wiles_def_ring_lemma} and \ref{framing_variables_lemma} we find that $R_N/\mathfrak{a} = R_0$ for all $N$.
			
			\item The complex $C_0$ is finitely generated since $X^K_{\GL_n/F}$ is homotopy equivalent to its Borel--Serre compactification (see proposition \ref{borel_serre_prop}) and the cohomology of a compact manifold is finitely generated.
			For $N \geq 1$, we let 
			$C_N$ be a minimal representative of 
			\[
			R\Hom_{\mathcal{O}}(R\Gamma(X^{K_1(Q)}_{\GL_n/F},\mathcal{V}_\lambda)_{\mathfrak{m}_1}, \mathcal{T})/\mathfrak{a}_{N}.
			\] Then $C_N / \mathfrak{a} = C_0$ by proposition \ref{complexes_hecke_ideals_prop}.
			
			\item We let $q_0 = [F^+:\rationals] n(n-1)/2$ and $l_0 = n [F^+ : \rationals] - 1$, then
			\cite[Theorem 2.4.10]{10author} implies that $C_0[1/p]$ is not exact since it contains $(\Pi^\infty)^K$ in its cohomology. The same theorem also states that $C_0[1/p]$ is concentrated in degrees $[q_0, q_0 + l_0]$. 
			On the other hand we have $\dim (S_\infty)_{\mathfrak{a}} = n^2 |S \cup R| + nq - 1$ and lemma \ref{generic_fibre_dim_bound} implies
			 \begin{linenomath*} \begin{align*}
				\dim R_\infty[1/p] & \leq g + n^2 |S \cup R| + n(n-1)[F^+:\rationals] \\
				&= qn + n^2 |S \cup R| -n [F^+:\rationals] \\
				&= \dim (S_\infty)_{\mathfrak{a}} - l_0 
			\end{align*} \end{linenomath*} 
			
			\item For $N \geq 1$ we put
			$T_N = \mathbb{T}^{S \cup R \cup Q_N}_{Q_N} (C_N)_{\mathfrak{m}_1} \cdot \mathcal{T} \subset \End_{\mathbf{D}(S_N)}(C_N)_{\mathfrak{m}_1}$. Then it is clear that the image of $T_N$ in $\End_{\mathbf{D}(S_0)}(C_0)$ equals \[
			\mathbb{T}^{S \cup R \cup Q_N}(C_0) \subset \mathbb{T}^{S \cup R}(C_0) = T_0.\] 
			
			\item By lemma \ref{framing_variables_lemma} we can tensor the surjections from theorem \ref{galois_rep_existence_thm} with $\mathcal{T}$ to obtain surjections $R_N \to T_N/I_N$, where $I_N$ are
			ideals satisfying $I_N^\delta$ for some constant $\delta$. The commutativity of the square follows from the Chebotarev density theorem and the computation of characteristic polynomials of Frobenius elements in terms of Hecke polynomials in the theorem \ref{galois_rep_existence_thm}.
			
			\item We already defined $\mathfrak{p}$ as the point corresponding to $\rho$. It is clear that this becomes a maximal ideal of $R_0[1/p]$. Since it corresponds to $\rho = r_\iota(\Pi)$ it is the pullback of the maximal ideal of $T_0[1/p]$ corresponding to the summand $(\Pi^\infty)^K$. Now $\widehat{(R_\infty)_\mathfrak{p}} = \widehat{(R_{\mathcal{S}}^{S \cup R,loc})_\mathfrak{p}}[[X_1, \dots, X_g]]$ and $\widehat{(R_{\mathcal{S}}^{S \cup R,loc})_{\mathfrak{p}}}$ is regular by proposition \ref{smooth_generic_fibre_prop}. Hence $(R_\infty)_{\mathfrak{p}}$ is regular, too. \qedhere
		\end{enumerate}
	\end{proof}
	
	\begin{corollary} \label{elliptic_curve_corollary}
		Let $A/F$ be an elliptic curve without complex multiplication over a CM field $F$ and $p \geq 7$ a prime such that $\zeta_p \not \in F$ and the image of $G_F$ in $\operatorname{Aut}(A[p])$ contains $\SL_2(\mathbb{F}_p)$.
		Let
		\[
		V_p A = \left( \lim_{\longleftarrow} A[p^n](\overline{F}) \right) \otimes_{\ints_p} \rationals_p
		\]
		be the $p$-adic $G_F$-representation associated with $A$. Then $V_p A$ is rigid, or equivalently any short exact sequence
		\[
		0 \to V_p A \to V \to V_p A \to 0,
		\]
		where $V$ is a de Rham $G_F$-representation, splits. Furthermore, if $F/\rationals$ is Galois and $n$ is a positive integer, such that $p > 2n + 3$, then the symmetric power $\operatorname{Sym}^n (V_p A)$ is rigid, too.
	\end{corollary}
	
	\begin{proof}
		Let $\rho = \operatorname{Sym}^n (V_p A)$. Since the determinant of $V_p A$ is the cyclotomic character, we find that in fact the image of $G_{F(\zeta_p)}$ in $\Aut(A[p])$ contains $\SL_2(\mathbb{F}_p)$. Hence \cite[Lemma 7.1.4]{10author} implies that $\overline{\rho}(G_{F'(\zeta_p)})$ is enormous for all $F'/F$ which are linearly disjoint from $\overline{F}^{\ker \overline{\rho}}$.
		Moreover, since the symmetric powers of the standard representation of $\SL_2$ are irreducible we find that
		$\overline{\rho}\rvert_{G_{F'}}$ is absolutely irreducible for all such $F'$. 
		
		If $n = 1$, then
		one can use the Chebotarev density theorem to show that $\overline{\rho}\rvert_{G_F'}$ is decomposed generic for any $F'/F$ disjoint from $\overline{F}^{\ker \overline{\rho}}$. If $n > 1$ and $F/\rationals$ is Galois, then \cite[Lemma 7.1.6]{10author} implies that $\overline{\rho} \rvert_{G_{F'}}$ is decomposed generic for any $F'/F$ disjoint from $\overline{F}^{\ker \overline{\rho}}$.
		
		Now \cite[Theorem 7.1.11]{10author} shows that there exists an $F'/F$, linearly disjoint from $\overline{F}^{\ker \overline{\rho}}$ such that 
		$\rho \rvert_{G_{F'}} \otimes_{\rationals_p} \overline{\rationals}_p \cong r_\iota(\Pi)$ for some cohomological cuspidal automorphic representation $\Pi$ of $\GL_{n + 1}(\mathbb{A}_{F'})$. Finally, to apply our main theorem it is enough to show that the Weil--Deligne representation $(r,N) = \operatorname{WD}(\rho \rvert_{G_{L}})$ is generic for all finite extensions $L/F'_v$. If $A$ has potentially good reduction at $v$, then by the Hasse bound, $r$ and $r(1)$ share no common Frobenius eigenvalues, hence $\Hom((r,N), (r(1),N)) = 0$. If $A$ has potentially multiplicative reduction at $v$, then $N$ has maximal rank and any non-zero element of $\Hom((r,N), (r(1),N))$ would preserve the one-dimensional space $\ker N$. But $r(\Frob_v)$ acts on $\ker N$ by a root of unity and $r(1)(\Frob_v)$ acts on $\ker N$ by an algebraic number of complex absolute value $q_v$, implying that $\Hom((r,N), (r(1),N)) = 0$.
	\end{proof}

	\bibliographystyle{alpha}
	\bibliography{refs}
\end{document}